\documentclass[11pt,reqno]{amsart}
 \allowdisplaybreaks
 
\usepackage{amsthm,amssymb,amsmath,epsfig,graphics,appendix,bm,mathrsfs,bbm,float}
\usepackage{graphicx}
\usepackage{hyperref}
\usepackage{bm}
\usepackage[dvipsnames, svgnames, x11names]{xcolor}
\usepackage{stmaryrd}  
\usepackage[left=1in, right=1in, top=1.1in,bottom=1.1in]{geometry}
\usepackage{enumitem}

\def\R{\mathbb R}
\def\E{\mathbb E}

\newtheorem{theorem}{Theorem}
\newtheorem{lemma}{Lemma}[section]
\newtheorem{definition}{Definition}[section]
\newtheorem{corollary}{Corollary}[section]
\newtheorem{remark}{Remark}[section]
\newtheorem{proposition}{Proposition}[section]
\newtheorem{example}{Example}[section]
\newtheorem{assumption}{Assumption}
\newtheorem{assumptionalt}{Assumption}[assumption]
\newenvironment{assumptionp}[1]{
\renewcommand\theassumptionalt{#1}
\assumptionalt
}{\endassumptionalt}

\DeclareMathOperator*{\esssup}{ess\,sup}

\usepackage{geometry}
\usepackage{xcolor}
\usepackage{amsfonts}
\usepackage{lipsum}                     
\usepackage{xargs}         
\usepackage[colorinlistoftodos,prependcaption,textsize=tiny, textwidth=2cm]{todonotes} 
\numberwithin{equation}{section}

\begin{document}
\author{}
\date{}
\title[BSDEs with nonlinear Young drivers]{Backward stochastic differential equations with nonlinear Young drivers I}

\author[J. Song]{Jian Song}\address{Research Center for Mathematics and Interdisciplinary Sciences, Frontiers Science Center for Nonlinear Expectations (Ministry of Education), State Key Laboratory of Cryptography and Digital Economy Security, Shandong University, Qingdao, 266237, China}
\email{txjsong@sdu.edu.cn}

\author[H. Zhang]{Huilin Zhang}\address{Research Center for Mathematics and Interdisciplinary Sciences, Frontiers Science Center for Nonlinear Expectations (Ministry of Education), Shandong University, Qingdao 266237, China; Department of Mathematics, Humboldt University, Berlin 10099, Germany}
\email{huilinzhang@sdu.edu.cn}

\author[K. Zhang]{Kuan Zhang}\address{Research Center for Mathematics and Interdisciplinary Sciences, Shandong University, Qingdao 266237, China}
\email{201911819@mail.sdu.edu.cn}

\begin{abstract}
This paper (alongside its companion, Part II \cite{BSDEYoung-II}) investigates backward stochastic differential equations (BSDEs) involving a nonlinear Young integral of the form $\int_{t}^{T}g(Y_{r})\eta(dr,X_{r})$, where the driver $\eta(t,x)$ is a space-time H\"older continuous function and $X$ is a diffusion process. Solutions to such equations provide a probabilistic interpretation of the solutions to stochastic partial differential equations (SPDEs) driven by space-time noise. 

Assuming the driver $\eta(t,x)$ is bounded, we establish the existence and uniqueness of the solutions to these BSDEs via a modified Picard iteration method. We then derive a comparison principle by analyzing the associated linear BSDEs and establish regularity properties of the solutions. As an application, we obtain Feynman-Kac formulae for a class of linear stochastic heat equations subject to Neumann boundary conditions.

\end{abstract}

\maketitle
\tableofcontents

\section{Introduction}
Let $(\Omega, \mathcal F, (\mathcal F_t)_{t\in[0,T]}, \mathbb{P})$ be a probability space
supporting a  $d$-dimensional Brownian motion $W$. Given progressively measurable functions $f,g$ and $\eta$ and a  continuous adapted stochastic process $X$ with finite $p$-variation for some $p>2$ (e.g., a diffusion process), we consider the following backward stochastic differential equations (BSDEs) on $[0,T]$ with the solution pair $(Y_t,Z_t)\in \R^N\times \R^{N\times d}$:
\begin{equation}\label{e:ourBSDE}
Y_{t} = \xi + \int_{t}^{T}f(r,Y_{r},Z_{r})dr + \sum_{i=1}^{M}\int_{t}^{T}g_{i}(Y_{r})\eta_{i}(dr,X_{r}) - \int_{t}^{T}Z_{r}dW_{r},\ t\in[0,T],
\end{equation}
where  $\int g(Y_r) \eta(dr,X_{r})$ is a {\it nonlinear Young integral} introduced by Catellier-Gubinelli \cite{Catellier-Gubinelli-2016} and Hu-L\^e \cite{HuLe} (see Section~\ref{sec:Young-integral} for details). A typical example of $\eta(t,x)$ is a space-time H\"older continuous function, such as a realization of a fractional Brownian sheet. In this work, we establish the well-posedness of Equation~\eqref{e:ourBSDE} when $\text{sup}_{(t,x)\in [0,T]\times \R^d} |\eta(t,x)| <\infty$, which is referred to as the bounded case (see also Remark \ref{rem:eta}); the unbounded case will be addressed in Part II of this study \cite{BSDEYoung-II}.
 
\subsection{Motivations and literature review}\label{sec:background}
The existence and uniqueness of general BSDEs were established by Pardoux-Peng \cite{PardouxPeng1990}. Nonlinear Feynman-Kac formulae via BSDEs were then introduced in \cite{Peng-1991,pp92}, providing probabilistic interpretations for solutions to semilinear parabolic partial differential equations (PDEs). Since then, BSDE theory has received significant attention and found numerous applications in fields such as mathematical finance, stochastic optimal control,  PDEs, and data science.  We refer to \cite{ElKaroui1997, zhang2017backward} and the references therein for BSDE theory and applications. Our motivation for studying BSDE~\eqref{e:ourBSDE} stems from the following two problems. \\

(a) The first motivation arises from Pontryagin's stochastic maximum principle (SMP) for stochastic differential equations (SDEs) with singular  drifts. Consider the following control system:
\begin{equation}\label{e:SDE-control}
\left\{\begin{aligned}
&dX^{\alpha}_{t} = \mu(t,X^{\alpha}_{t},\alpha_{t})dt + \upsilon(t,X^{\alpha}_{t},\alpha_{t})dW_{t}, \\
&X^{\alpha}_{0} = x_{0}, 
\end{aligned}\right.
\end{equation}
where $\{\alpha_t\}_{t\ge 0}$ is the control process. Suppose the cost function and value function are given by 
\[
J(\alpha):=\mathbb{E}\left[\int_{0}^{T}F(t,X^{\alpha}_{t},\alpha_{t})dt + h(X^{\alpha}_{T})\right], \ \ \mathcal{V}(0,x_0):= \inf_{\alpha \in \mathcal{A}} J(\alpha).
\] 
The SMP via the dual argument requires solving a BSDE (the adjoint equation) driven by $\partial_x H$,
where $H(t,x,a,y,z):=\mu(t,x,a)y + \upsilon(t,x,a)z + F(t,x,a)$
is the Hamiltonian. For more details on the SMP and the adjoint equation, see e.g., Yong-Zhou~\cite{Yong1999}. 

A typical example of SDEs with singular drifts arises from the directed polymer model and the Kardar-Parisi-Zhang (KPZ) equation. More precisely, Equation \eqref{e:SDE-control} with $\upsilon\equiv 1$ and $\mu(t,x)=\partial_x u(t,x)$, where $u(t,x)$ is  the solution to the KPZ equation, was proposed by Alberts-Khanin-Quastel \cite{AKQ14} to describe the continuum directed polymer measure, and later on its well-posedness was established by Delarue-Diel \cite{Delarue-Diel16} using rough path theory. Other relevant models, e.g., mixed fast-slow systems, stochastic filtering, and pathwise stochastic control problems, can also be viewed as SDEs with singular drifts. We refer to \cite{pei2021averaging,hairer2022generating,friz2021rough,FLZ24,HZ25} for further discussions.

If the drift in \eqref{e:SDE-control} takes a more general form  
$\mu(t,x,a) dt = g(x, a) \zeta(dt, x)$, the corresponding adjoint BSDE has a singular driver  $\eta(t,x) = \partial_{x}\zeta(t,x)$, and the integral is the
so-called \emph{nonlinear Young integral} (see, e.g., Galeati \cite{galeati2021nonlinear} or Section~\ref{sec:Young-integral} for the definition). Thus, this adjoint BSDE is indeed a special case of our Equation~\eqref{e:ourBSDE}. \\

(b) The second motivation is to find probabilistic interpretations for solutions of SPDEs driven by space-time noise. Consider the following SPDE:
\begin{equation}\label{e:SPDEs}
\left\{
\begin{aligned}
-\partial_{t}u(t,x) &= \mathcal{L}u(t,x) + g\left(u, \nabla u\right)V(t,x),\ (t,x)\in[0,T]\times\mathbb{R}^{d},\\
u(T,x) &= h(x),\ x\in\mathbb{R}^{d},
\end{aligned}\right.
\end{equation}
where $V$ is Gaussian noise and $\mathcal{L}$ is an elliptic differential operator  \[\mathcal{L}u(t,x):=\frac{1}{2}\text{tr}\left\{\nabla^{2}u(t,x)\sigma(x)\sigma(x)^{\text{T}}\right\} + b^{\text{T}}(x)\nabla u(t,x).\] 

$(i)$ For the case $g(y,z) = y$ and $(b,\sigma)=(0,I_{d})$, the Feynman-Kac formula for SPDE~\eqref{e:SPDEs} has been studied in various situations. Let $\{B(t,x), t\in\R_+, x\in\R^d\}$ be a fractional Brownian sheet with Hurst parameters $(H_{0},H_{1},\dots,H_{d})\in(0,1)^{1+d}$. When $V(t,x)=\frac{\partial^{d+1}B}{\partial {t}\partial {x_{1}}\dots\partial {x_{d}}}(t,x)$ and $H_{i}>1/2$ for $i=0,1, \dots, d$, Hu et al. \cite{HuNualartSong-2011} obtained a weak solution by the Feynman-Kac-type representation; a similar result is obtained in \cite{hu2012feynman} when $V(t,x) = \partial_{t}B(t,x)$ and $H_{0}\in(1/4,1/2)$. More recently, Hu et al. \cite{HuLiMi} obtained a mild solution to \eqref{e:SPDEs} by studying a linear BSDE with a driver which is  singular in time and fast (polynomially) vanishing in space. 

$(ii)$ When $g$ is independent of $\nabla u$, and $V(t,x) = \partial_{t}B(t,x)$ with $B$ being a fractional Brownian sheet, SPDE~\eqref{e:SPDEs} has been investigated in a sequence of works \cite{hinz2009gradientI,hinz2009gradientII,hinz2012semigroups}. 
When 
$V(t,x) = \partial_{t}\nabla B(t,x)$, Hinz \cite{hinz2011burgers} explored the solvability of the stochastic Burger equations, and Hinz et al. \cite{hinz2013elementary} studied transport-type SPDEs.
Nevertheless, the corresponding forward-backward SDE (FBSDE) has not been studied, and thus the relation between solutions of FBSDEs and solutions of SPDEs with singular drivers, known as the nonlinear Feynman-Kac formula, is left open. In applications, nonlinear Feynman-Kac formulae may serve as a valuable tool for numerical analysis of SPDEs/PDEs (e.g., Monte Carlo method, see \cite[Chapter 5]{zhang2017backward} and references therein).  Despite the Feynman-Kac formula, FBSDE theory can also be applied to characterize the scaling limits of solutions to SPDEs (see Dunlap-Gu \cite{Dunlap-Gu} for the  stochastic heat equation case). 

$(iii)$ For the nonlinear case \eqref{e:SPDEs}, typical models are the KPZ equation,  the nonlinear version of parabolic Anderson model, Burgers-type equations, etc, where $V$ is space-time noise. Such equations have been  studied, for instance, in Hairer \cite{Hairer-KPZ} and  Gubinelli-Imkeller-Perkowski \cite{Gubinelli-Imkeller-Perkowski}. To study SPDEs in a general form of \eqref{e:SPDEs} via BSDEs, we introduce a novel form \eqref{e:ourBSDE} of BSDE driven by space-time noise $\eta$.  The present work focuses on the case that the driver $\eta$ is less singular than space-time white noise.  When $V$ is space-time white noise, the well-posedness of (F)BSDEs associated with \eqref{e:SPDEs} appears to be unresolved, while the Feynman-Kac formula was studied when $g(u, Du)=u$ in Bertini-Cancrini \cite{bertini1995stochastic}.\\

BSDEs with stochastic (singular) drivers in the form of \eqref{e:ourBSDE} have been studied in various frameworks. When $\eta(t,x)=B(t)$ is a (fractional) Brownian motion or a general martingale independent of $W$, such equations are commonly referred to as backward doubly stochastic differential equations (BDSDEs). BDSDEs were introduced by Pardoux-Peng \cite{PardouxPeng1994} to investigate the probabilistic representation of solutions to semilinear SPDEs. Jing \cite{Jing-2012} considered the case where $\eta(t,x)=B(t)$ is a fractional Brownian motion with Hurst parameter $H > 1/2$. When $\eta(t,x)$ is  martingale with space parameter, Song et al. \cite{SongSongZhang} studied the BDSDE driven by a backward nonlinear It\^o integral (see, e.g., Kunita \cite{kunita1997stochastic}). Note that the approach used in \cite{SongSongZhang} heavily relies on the martingale structure of $\eta$, which is unavailable in our setting. 
When $\eta(dt,x)=b(t,x)dt$ for some $b(\cdot,\cdot)\in C([0,T];H^{-\gamma}_{q})$ with $0<\gamma<1/2$, Issoglio-Russo \cite{IssoglioRusso} studied a linear BSDE driven by $\eta(dt,x)$ and they proved a Feynman-Kac formula for the corresponding PDE. \\

Stochastic processes such as (fractional) Brownian motions, L\'evy processes, etc., can be treated (so-called lifted) as rough paths (see e.g. \cite[Theorem 15.33]{FrizVictoir}). This idea naturally motivates the investigation into the study of PDEs and corresponding BSDEs driven by rough paths, thereby extending stochastic analysis to the rough analysis regime. Rough analysis provides a framework to investigate systems with singular drivers, benefiting the analysis from the continuity and generality sides. 
Topics on PDEs driven by rough signals (rough PDEs), including the well-posedness, stability with respect to rough signals, and space-time regularity of solutions, have been studied widely;
see, for example, \cite{Caruana2009,Gubinelli2010,caruana2011rough,Friz2014a,Diehl2017b,BuckdahnZhang,Liang2024a,Bugini2025a}. Recently, Hocquet-Neamţu~\cite{Hocquet2024} solved quasilinear rough evolution equations in the mild sense, including Dirichlet, Neumann, and periodic boundary value problems. 
Although there are numerous works on rough PDEs, there has been limited literature on related BSDEs, even in the case that $\eta(t,x)=\eta(t)$ is a rough path independent of space variable. Diehl-Friz \cite{DiehlFriz} firstly addressed the case where $\eta(t)$ has finite $p$-variation ($2<p<\infty$).
Using an approximation argument, they established the well-posedness of \eqref{e:ourBSDE}, proving the existence of a formal solution continuous in rough drivers. Liang-Tang \cite{liang2023multidimensional} extended the result of \cite{DiehlFriz} to the multi-dimensional case and provided a strong solution to such BSDEs in the linear case. Diehl-Zhang~\cite{DiehlZhang} considered the case when $\eta(t)$ has finite $p$-variation in the range $1<p<2$ and demonstrated the existence and uniqueness of solutions to \eqref{e:ourBSDE} using a fixed-point argument. Recently, Becherer-Sun \cite{Becherer2025} extended the results in \cite{DiehlZhang} to discontinuous Young drivers.\\

When the driver $\eta(t,x)$ is a function of both time and space, the well-posedness of \eqref{e:ourBSDE} is more challenging:

(i) The Young driver $\eta(\cdot,x)$ does not possess the semi-martingale property, and hence the It\^o calculus developed in Kunita \cite{kunita1997stochastic} does not apply here. In Section~\ref{sec:nonlinear-Y-I}, we develop pathwise analysis for the nonlinear Young integral  $\int g(Y_t) \eta(dt, X_t)$, inspired  by the classical Young integral theory (see, e.g. \cite{young1936, lyons1994differential,Catellier-Gubinelli-2016,HuLe, DiehlZhang,galeati2021nonlinear}).  We adopt the $p$-variation norm rather than the H\"older norm because it is more convenient to study the $p$-variation of $ \int_{\cdot}^{T}Z_{r}dW_{r}$ due to the irregularity of the process $Z$ in the BSDE setting. Note that the $p$-variation analysis only yields local Lipschitz estimates, which makes it difficult to obtain an {\it a priori} estimate for nonlinear BSDEs (see Remark~\ref{re:eu}).

(ii) The methodology used in Diehl-Friz~\cite{DiehlFriz} in the study of rough BSDEs involves the Doss-Sussmann transformation, which requires the smoothness of rough flows. In our setting, the map $x\mapsto\eta(t,x)$ lacks smoothness (see Assumption~\ref{(H1)}), and hence we can not take this approach.

(iii) It is tempting to interpret our equation \eqref{e:ourBSDE} as a Young BSDE in the setting of Diehl-Zhang~\cite{DiehlZhang}, by applying the transformation $\tilde{\eta}_t := \int_{0}^{t} \eta(dr , X_r)$ to get a new driver $\tilde \eta_t$ which does not involve a space variable. Note that it is crucial to assume the driver is deterministic to obtain the \emph{a priori} estimate in \cite{DiehlZhang}: when the driver is deterministic with bounded $p$-variation, the solution to the Young BSDE can be dominated by the solution to a deterministic Young ODE with the same driver by comparison principle. In our situation, clearly the new driver $\tilde{\eta}_t$ is random and its $p$-variation may be unbounded, and hence we can not adopt the strategy in \cite{DiehlZhang} in a straightforward way (see Remark~\ref{rem:differ from DZ} for a detailed discussion).

\

As an application, we apply BSDEs with nonlinear Young drivers to study linear Young SPDEs with Neumann boundary conditions in the weak sense, and provide an explicit representation (i.e., Feynman-Kac formula) for the solution by our BSDE and a reflected Brownian motion.
This BSDE approach circumvents proving the exponential integrability for the $p$-variation of the Brownian motion reflected in a general bounded domain (see Remark \ref{rem:prob-fk}), providing a new method to validate the Feynman-Kac expression. Furthermore, the stability of BSDE solutions with respect to their drivers facilitates the convergence of approximated solutions for the SPDE, and consequently, we obtain the time regularity of  solution to  the SPDE.

\ \ \ \ \ 

This paper is organized as follows. In Section~\ref{sec:nonlinear-Y-I}, we revisit some essential concepts related to the nonlinear Young integral and prove an extended It\^o's formula. Our main result is the well-posedness of BSDE~\eqref{e:ourBSDE}, which is addressed in Section~\ref{sec:picard}. Additionally,  in Section~\ref{sec:Comparison} we establish a comparison theorem by investigating linear BSDEs. The regularity of  solutions is studied in Section~\ref{sec:continuity map}, which plays a key role in solving nonlinear Young BSDEs with \emph{unbounded} driver $\eta$ in the second part  of this project \cite{BSDEYoung-II}.  Finally, an application to linear stochastic heat equations with Neumann boundary conditions is presented in Section~\ref{sec:reflecting Feynman-Kac functional}; some miscellaneous results are summarized in Appendix~\ref{miscellaneous results} and \ref{append:B}.

\subsection{Notations}\label{sec:preliminaries} In this subsection, we gather the notations that will be used in this article. 

\begin{itemize}
  
\item {\it $\lesssim$:} The notation $f\lesssim g$ means $f\le Cg$ for some generic constant $C>0$, and $f \lesssim_{\zeta} g$ represents $f \le C_{\zeta} g$ for some constant $C_{\zeta}>0$ depending on $\zeta$. We use the Euclidean norm  $|\cdot|$ for both vectors and matrices.\\

\item {\it Probability space:} 
Let $T>0$ be fixed, and $\{W_t\}_{t\in[0,T]}$ be a $d$-dimensional standard
Brownian motion on the probability space $(\Omega,\mathcal{F},(\mathcal{F}_{t})_{t\in[0,T]}, \mathbb{P})$, and where $(\mathcal{F}_{t})_{t\in[0,T]}$ is the augmented filtration generated by $W$. We use the notation $\mathbb{E}_{t}[\cdot]:=\mathbb{E}\left[\cdot|\mathcal{F}_t\right]$.

For a  random vector $\xi$, we write $\|\xi\|_{L^{k}(\Omega)}$ for the $L^{k}$-norm of $\xi$ with $k\ge 1$, and the essential supremum norm is  denoted by		\begin{equation*}			
\|\xi\|_{L^{\infty}(\Omega)}:=\esssup_{\omega\in\Omega}|\xi(\omega)|. 
\end{equation*}

\item{\it Spaces of functions:} For  functions $g:[s,t]\rightarrow \mathbb{R}^{m}$ and $f:\mathbb{R}^{n}\rightarrow\mathbb{R}^{m}$, the uniform norms are denoted by
\[\|g\|_{\infty;[s,t]} := \esssup_{r\in[s,t]}|g_{r}|\text{ and }\|f\|_{\infty;\mathbb{R}^{n}} := \esssup_{y\in\mathbb{R}^{n}}|f(y)|.\]
For $k\ge 1$, denote by $L^{k}([s,t];\R^m)$ the space of functions on $[s,t]$ equipped with the norm
\begin{equation*}
\|g\|_{L^{k}(s,t)} := \left\{\int_{s}^{t} |g_{r}|^k dr\right\}^{\frac{1}{k}}.
\end{equation*}
We write 
$g\in C([s,t];\mathbb{R}^{m})$ if $g$ is continuous; $f\in C^{k}(\mathbb{R}^n;\mathbb{R}^{m})$ if $f$ is $k$-th continuously differentiable; $f\in C^{k}_{b}(\mathbb{R}^n;\mathbb{R}^{m})$ if $f\in C^{k}(\mathbb{R}^n;\mathbb{R}^{m})$ with bounded partial derivatives up to the $k$-th order; $f\in C^{\text{Lip}}(\mathbb{R}^n;\mathbb{R}^{m})$ if it is Lipschitz.
For any $p\ge 1$, we denote the $p$-variation of $g$ on $[s,t]$ by  
\[\|g\|_{p\text{-var};[s,t]}:= \Big(\sup_{\pi} \sum_{[t_{i},t_{i+1}]\in \pi}|g_{t_{i+1}}-g_{t_{i}}|^p \Big)^{1/p},\]
where the supremum is taken over all partitions $\pi:=\{[t_i,t_{i+1}];\  s=t_1<t_2<\cdots< t_n=t\}$ of $[s,t]$. We denote
\[C^{p\text{-var}}([s,t];\mathbb{R}^{m}) := \left\{g\in C([s,t];\mathbb{R}^{m})\text{ such that }\|g\|_{p\text{-var};[s,t]}<\infty \right\}.\]
For $g\in C([s,t];\R^m)$ and $\gamma\in(0,1]$, the $\gamma$-H\"older norm of $g$ is defined as
\begin{equation*}
\|g\|_{\gamma\text{-H\"ol};[s,t]} := \sup_{\substack{a,b\in[s,t], \\ a\neq b}} \frac{|g_{a} - g_{b}|}{|a-b|^{\gamma}}\,.
\end{equation*}
 We denote by $C^{\gamma\text{-H\"ol}}([s,t];\R^m)$ the space of functions with finite $\gamma$-H\"older norm.
Clearly, we have for $p\ge 1$,
\begin{equation}\label{e:p-var and Holder}
\left\|g\right\|_{p\text{-var};[s,t]} \le |t - s|^{\frac{1}{p}} \left\|g\right\|_{\frac{1}{p}\text{-H\"ol};[s,t]},
\end{equation}
and hence $C^{p\text{-var}}([s,t];\R^m)\subseteq C^{\frac{1}{p}\text{-H\"ol}}([s,t];\R^m)$. When $m=1$, we may omit the range space $\R^m=\R$ in the notations mentioned above. For instance, we  write $C^{\text{Lip}}([s,t])$ for $C^{\text{Lip}}([s,t];\mathbb{R})$. 
  
A continuous function $w: \{(s,t); 0 \leq s\le t \le T\} \rightarrow [0,\infty)$  is called a control on $[0,T]$, if  $w(s,s)=0$ and $w(s,u) + w(u,t) \le w(s,t)$ for $0\le s\le u\le t\le T$.    For instance, $w(s,t)=|t-s|$ and  $w(s,t) = \|g\|_{p\text{-var};[s,t]}^{p}$ are controls. \\

\item {\it Spaces of the driver $\eta
$:} 
For any $\eta:[0,T]\times\mathbb{R}^{n}\rightarrow\mathbb{R}^{m}$, $\tau, \lambda\in(0,1],\  \beta\in[0,\infty)$,  denote
\begin{equation}\label{e:norm1}
\begin{aligned}
\|\eta\|_{\tau,\lambda; \beta}&:=\sup _{\substack{0 \leq s<t \leq T \\ x, y \in \mathbb{R}^{n} ; x \neq y}} \frac{|\eta(s, x)-\eta(t, x)-\eta(s, y)+\eta(t, y)|}{|t-s|^{\tau}|x-y|^{\lambda}(1+|x|^{ \beta}+|y|^{ \beta})} \\ 
&\quad +\sup _{\substack{0 \leq s<t \leq T\\ x \in \mathbb{R}^{n}}} \frac{|\eta(s, x)-\eta(t, x)|}{|t-s|^{\tau}(1+|x|^{ \beta + \lambda})}+\sup _{\substack{0 \leq t \leq T \\ x, y \in \mathbb{R}^{n} ; x \neq y}} \frac{|\eta(t, y)-\eta(t, x)|}{|x-y|^{\lambda}(1+|x|^{ \beta}+|y|^{ \beta})} ,
\end{aligned}
\end{equation}
\begin{equation}\label{e:norm2}
\begin{aligned}
\|\eta\|_{\tau,\lambda}&:=\sup _{\substack{0 \leq s<t \leq T \\ x, y \in \mathbb{R}^{n} ; x \neq y}} \frac{|\eta(s, x)-\eta(t, x)-\eta(s, y)+\eta(t, y)|}{|t-s|^{\tau}|x-y|^{\lambda}} \\ 
&\quad +\sup _{\substack{0 \leq s<t \leq T \\ x \in \mathbb{R}^{n}}} \frac{|\eta(s, x)-\eta(t, x)|}{|t-s|^{\tau}}+\sup _{\substack{0 \leq t \leq T \\ x, y \in \mathbb{R}^{n} ; x \neq y}} \frac{|\eta(t, y)-\eta(t, x)|}{|x-y|^{\lambda}} ,
\end{aligned}
\end{equation}
and for a compact set $K\subset \R^{n}$, 
\begin{equation*}
\begin{aligned}
\|\eta\|^{(K)}_{\tau,\lambda} &:= \sup_{\substack{0 \leq s<t \leq T\\ x,y\in K; x \neq y}}\frac{|\eta(s,x) - \eta(t,x) - \eta(s,y) + \eta(t,y)|}{|t-s|^{\tau}|x-y|^{\lambda	}}\\
&\quad  + \sup_{\substack{0 \leq s<t \leq T\\ x\in K}}\frac{|\eta(s,x) - \eta(t,x)|}{|t-s|^{\tau}} + \sup_{\substack{0\le t\le T\\ x,y\in K; x \neq y}} \frac{|\eta(t,y) - \eta(t,x)|}{|x-y|^{\lambda}}.
\end{aligned}
\end{equation*}
Here, $\tau$ and $\lambda$ are (weighted) H\"older  exponents in time and in space, respectively. Clearly, we have 
\begin{equation}\label{e:tau,lambda;0}
\|\eta\|^{(K)}_{\tau,\lambda}\lesssim_{\lambda,\beta,K} \|\eta\|_{\tau,\lambda; \beta} \lesssim\|\eta\|_{\tau,\lambda;0} \le  \|\eta\|_{\tau,\lambda}\ ,
\end{equation}
and they are equivalent if the spatial variable $x$ in $\eta(t,x)$ is restricted to $\left\{x\in\mathbb{R}^{n};|x|\le K\right\}$. 
We introduce the following spaces of functions
\begin{equation*}
\begin{split}
&\ \ \ \ \ \ \ \ \ \ \  C^{\tau,\lambda}([0,T]\times\mathbb{R}^{n};\mathbb{R}^{m}):=\left\{\eta: [0,T]\times \mathbb{R}^{n}\rightarrow \mathbb{R}^{m}\text{ such that }\|\eta\|_{\tau,\lambda}<\infty\right\},\\
&\ \ \ \ \ \ \ \ \ \ \  C^{\tau,\lambda;\beta}([0,T]\times\mathbb{R}^{n};\mathbb{R}^{m}):=\left\{\eta: [0,T]\times \mathbb{R}^{n}\rightarrow \mathbb{R}^{m}\text{ such that }\|\eta\|_{\tau,\lambda;\beta}<\infty\right\},\\
&\ \ \ \ \ \ \ \ \ \ \  C^{\tau,\lambda}_{\text{loc}}([0,T]\times \mathbb{R}^{n};\mathbb{R}^{m}):=\Big\{\eta:[0,T]\times\mathbb{R}^{n}\rightarrow\mathbb{R}^{m}  \text{ such that }\|\eta\|^{(K)}_{\tau,\lambda}<\infty  \text{ for any compact } K \Big\}.   
\end{split}    
\end{equation*}
 
\begin{remark}\label{rem:eta}
In particular, if we assume  $\sup_{x\in \R}|\eta(0,x)|< \infty$, then $ \|\eta\|_{\tau,\lambda}<\infty$ implies that $|\eta(t,x)|$ is uniformly bounded for all $(t,x)\in[0,T]\times \R^d$ (see Lemma~\ref{lem:smooth approximation of eta}). Letting $\tilde \eta(t,x):=\eta(t, x)-\eta(0, x)$,  we have $\tilde \eta(0,x)\equiv0$ and  $\tilde \eta(dt, x_t) =\eta(dt, x_t)$. Thus, we may assume throughout the paper that $\eta(0,x)\equiv 0$ for Equation~\eqref{e:ourBSDE}, and then the condition $\|\eta\|_{\tau, \lambda}<\infty$ yields the uniform boundedness of $|\eta(t,x)|$.
\end{remark}

\begin{example}\label{rem:fBs}
A typical example of $\eta$ is a sample path of fractional Brownian sheet $\{B(t,x), (t,x)\in[0,T]\times \mathbb{R}^{n}\}$    with Hurst parameters $H_i\in (0,1),i=0,1, \dots,n$, which is a centered Gaussian random field  with covariance
\begin{align*}
&\mathbb{E}\left[B(t,x)B(s,y)\right] \\
&= \frac{1}{2^{n+1}}\left(|t|^{2H_0} + |s|^{2H_0} - |t-s|^{2H_0} \right)\prod_{i=1}^{n}\left(|x_{i}|^{2H_i}+|y_{i}|^{2H_i} - |x_i - y_i|^{2H_i}\right).
\end{align*}
By Lemma \ref{lem:Hurst parameter}, $B\in C^{\tau, \lambda; \beta}([0,T]\times \R^n; \R)$ a.s., where $\tau := H_{0}-\theta$, $\lambda := \min_{1\le j\le n}H_{j} - \theta$ and $\beta := 2\theta + \sum_{j=1}^{n}H_{j} - \min_{1\le j\le n}H_{j}$ for any sufficiently small $\theta>0$. In particular, when $H_1=\cdots=H_n=H$, one can choose $\tau=H_0-\theta, \lambda=H-\theta$, and $\beta=(n-1)H+2\theta$ for any $\theta\in(0, H_0\wedge H)$, where only the spatial increment parameter $\beta$ depends on the space dimension $n$. 
\end{example}

\item{\it Spaces of solutions $(Y,Z)$:}
Let $[s,t]\subset[0,T]$ be a fixed interval, and $T_{1},T_{0}$ be two stopping times that satisfy $s\le T_{1}\le T_{0}\le t$. Assume that $X:\Omega\times[s,t]\rightarrow \R^n$ is a RCLL adapted process. For $p\ge 1$ and $k\ge 1$, denote
\begin{equation}\label{e:def-norms}
m^{(s,t)}_{p,k}(X;[T_{1},T_{0}]):=\left\{\esssup_{\omega\in\Omega, u\in[s,t]}\mathbb{E}_{u}\left[\|X\|^{k}_{p\text{-var};[T_{1}\vee (u\land T_{0}),T_{0}]}\right]\right\}^{\frac{1}{k}}.
\end{equation}
In the following, we will omit the superscripts $s$ and $t$, if there is no ambiguity.
Denote 
\begin{equation*}
\begin{aligned}
\mathfrak{B}([s,t]):=\big\{(y,z) \text{ such that } &y:\Omega\times[s,t]\rightarrow\mathbb{R}^{n}\text{ is continuous and adapted, and }\\
&z:\Omega\times[s,t]\rightarrow\mathbb{R}^{m}\text{ is  progressively measurable}\big\},
\end{aligned}
\end{equation*}
where the dimensions $n$ and $m$ may vary depending on the context.
Let $T_{1},T_{0}$ be two stopping times satisfying $s\le T_{1}\le T_{0}\le t$. For  $k\ge1 $ and $(y,z)\in\mathfrak{B}([s,t])$, we define the seminorm 
\begin{equation}\label{def:norms'}
\|(y,z)\|_{p,k;[T_{1},T_{0}]} := m_{p, k}(y;[T_{1},T_{0}])+ \|z\|_{k\text{-BMO};[T_{1},T_{0}]} + \|y_{T_{0}}\|_{L^{\infty}(\Omega)},
\end{equation}
where
\begin{equation*}
\|z\|_{k\text{-BMO};[T_{1},T_{0}]} := \left\{\esssup_{\omega\in\Omega, u\in[s,t]}\mathbb{E}_{u}\left[\left(\int_{u}^{t}\left|z_{r}\right|^{2} \mathbf{1}_{[T_1,T_0]}(r)dr\right)^{\frac{k}{2}}\right]\right\}^{\frac{1}{k}}.
\end{equation*}
When $T_1=s$ and $T_0=t$, $\|(\cdot, \cdot)\|_{p,k;[s,t]}$ is a norm on $\mathfrak{B}([s,t])$. 
For real numbers $s\le t$, define the seminorm 
\[|(y,z)|_{p,k;[s,t]} := m_{p, k}(y;[s,t])+ \|z\|_{k\text{-BMO};[s,t]},\]
and for any $R>0,$
\begin{equation}
\begin{aligned}
\mathfrak B_{p, k}(R;[s,t])&:=\Big\{ (y,z) \in \mathfrak{B}([s,t]):\ y \text{ is bounded,}\text{ and } |(y,z)|_{p,k;[s,t]}\le R\Big\}.  \\ \label{e:def of B_p,k}
\mathfrak B_{p,k}(s, t)& :=\bigcup_{R>0} \mathfrak B_{p,k}(R;[s, t]).
\end{aligned}
\end{equation} 
We also define the space 
\begin{equation*}
\mathfrak{H}_{p,k}(s,t) := \left\{ (y,z)\in\mathfrak{B}([s,t]): \ \|(y,z)\|_{\mathfrak{H}_{p,k};[s,t]} < \infty \right\},
\end{equation*}
where 
\begin{equation}\label{e:Hpk}
\|(y,z)\|_{\mathfrak{H}_{p,k};[s,t]} := \|y_{t}\|_{L^{k}(\Omega)} + \left\|\|y\|_{p\text{-var};[s,t]}\right\|_{L^{k}(\Omega)} + \left\|\Big|\int_{s}^{t}|z^{2}_{r}|dr\Big|^{\frac{1}{2}}\right\|_{L^{k}(\Omega)}.
\end{equation}

\begin{remark}\label{rem:boundedness of y}
Note that $\|(y,z)\|_{p,k;[T_{1},T_{0}]}<\infty$ implies the uniform boundedness of $y$. Indeed, we have
\[|y_{r}\mathbf{1}_{[T_{1},T_{0}]}(r)| \le \mathbb{E}_{r}[|y_{T_{0}}|] + \mathbb{E}_{ r}[|y_{T_{1}\vee(r\land T_{0})} - y_{T_{0}}|],\] 
which implies that\begin{equation}\label{eq:ess-mpk} \esssup\limits_{\omega\in\Omega, r\in[s,t]}|y_{r}\mathbf{1}_{[T_1,T_0]}(r)| \le \|y_{T_0}\|_{L^{\infty}(\Omega)}+m_{p,1}(y;[T_1, T_0])\le \|(y,z)\|_{p,1;[T_1,T_0]}.
\end{equation}
\end{remark}
\end{itemize}

\section{Nonlinear Young integral and It\^o's formula}\label{sec:nonlinear-Y-I}

In this section, we  provide some preliminaries on the nonlinear Young integral, which is an extension of the classical  Young integral introduced by Young \cite{young1936}.

\subsection{Nonlinear Young integral} \label{sec:Young-integral}

Let $\eta: [0,T]\times \mathbb{R}^{d}\rightarrow \mathbb{R}$ be a continuous deterministic function with $\eta(0,x) \equiv 0$, and let $x:[0,T]\to \R^d$ and $y:[0,T]\to \R$ be two continuous functions. Assume that $p_1\ge 1$ and $p_2\ge 1$ are two constants. In this section, inspired by \cite{HuLe} under a H\"older framework, we define the nonlinear Young integral
\[\int_{a}^{b}y_{r}\eta(dr,x_{r}),\]
where $x$ has finite $p_1$-variation and $y$ has finite $p_2$-variation respectively.

\begin{proposition}\label{prop:bounds}
Let $ \tau\in(0, 1]$, $ \lambda\in(0,1]$. 
Assume $\eta\in C^{\tau,\lambda}_{\mathrm{loc}}([0,T]\times \mathbb{R}^{d})$, $x\in C^{p_{1}\text{-}\mathrm{var}}([0,T];\mathbb{R}^{d})$, and $y\in C^{p_{2}\text{-}\mathrm{var}}([0,T];\mathbb{R})$ with $\tau +\frac{\lambda}{p_{1}}>1$  and $\tau + \frac{1}{p_{2}}>1$.  For $0\le a\le b\le T$, the following integral is well-defined,		
\[\int_{a}^{b} y_r \eta(dr,x_r) := \mathcal{A}_{a,b} := \mathcal A(b)-\mathcal A(a),\] 
where $\mathcal{A}:[0,T]\to \R$ is the unique function associated with $A(s,t)= y_s\big(\eta(t,x_s) - \eta(s,x_s)\big)$ by Lemma~\ref{lem:sewing}. Moreover, if we assume $\eta\in C^{\tau,\lambda;\beta}([0,T]\times\mathbb{R}^{d})$ for some $\beta\ge 0$, 
\begin{equation}\label{e:p-var2}
\begin{aligned}
\left\|\int_{\cdot}^{t}y_{r}\eta(dr,x_{r})\right\|_{\frac{1}{\tau}\text{-}\mathrm{var};[s,t]}&\le \frac{2^{\delta}}{1-2^{-\delta}}\|\eta\|_{\tau,\lambda;\beta} |t-s|^{\tau} \bigg(\Big(1+\|x\|^{\beta}_{\infty;[s,t]}\Big)\|x\|^{\lambda}_{p_{1}\text{-}\mathrm{var};[s,t]}\|y\|_{\infty;[s,t]}\\ 
&\quad +  \Big(1+\|x\|^{\beta+\lambda}_{\infty;[s,t]}\Big)\Big(\|y\|_{\infty;[s,t]}+\|y\|_{p_{2}\text{-}\mathrm{var};[s,t]}\Big)\bigg),
\end{aligned}
\end{equation}
and if we assume $\eta \in C^{\tau,\lambda}([0,T]\times\mathbb{R}^{d})$, 
\begin{equation}\label{e:p-var1}
\begin{aligned} 
\left\|\int_{\cdot}^{t}y_{r}\eta(dr,x_{r})\right\|_{\frac{1}{\tau}\text{-}\mathrm{var};[s,t]}\le \frac{2^{\delta}}{1-2^{-\delta}}\|\eta\|_{\tau,\lambda} |t-s|^{\tau} \bigg(\Big(1+\|x\|^{\lambda}_{p_{1}\text{-}\mathrm{var};[s,t]}\Big)\|y\|_{\infty;[s,t]} + \|y\|_{p_{2}\text{-}\mathrm{var};[s,t]} \bigg),
\end{aligned} 
\end{equation}
for all $0\le s\le t\le T$, where $\delta = \min\left\{\tau+\frac{1}{p_{2}}-1,\tau+\frac{\lambda}{p_{1}}-1\right\}$.
\end{proposition}

\begin{proof}
The proof is similar to that of Proposition~2.4 in \cite{HuLe}. See \nameref{proof:bounds} in Appendix~\ref{append:B}.
\end{proof}

\begin{remark}
For a continuous path $z$, we have $\|z\|_{p\text{-}\mathrm{var};[s,t]}\le\|z\|_{q\text{-}\mathrm{var};[s,t]}$ if $1\le q\le p$. Thus, the estimates~\eqref{e:p-var2} and \eqref{e:p-var1} are also valid for $\|\int_{\cdot}^{t}y_{r}\eta(dr,x_{r})\|_{p\text{-}\mathrm{var};[s,t]}$ with $p\ge \frac1\tau$.  
\end{remark}

\begin{remark}\label{rem:M_t}
Denoting $M_{t}:= \int_{0}^{t} \eta(dr,x_r)$, where $\eta$ and $x$ satisfy the conditions  in Proposition~\ref{prop:bounds}, then $M\in C^{\frac{1}{\tau}\text{-}\mathrm{var}}([0,T];\mathbb{R})$. As a consequence, for $y\in C^{p_{2}\text{-}\mathrm{var}}([0,T];\mathbb{R})$ with $\tau + \frac{1}{p_{2}}>1$, the classical Young integral $\int_{0}^{t} y_r dM_r$ is well defined, and it coincides with the nonlinear integral $\int_{0}^{t} y_r \eta(dr,x_r)$. Indeed, let $\Gamma(s,t):= y_{s} \int_{s}^{t} \eta(dr,x_r) - y_s \left(\eta(t,x_s)-\eta(s,x_s)\right)$. By \eqref{e:inequality-estimate-Young3} with $y_{r}\equiv y_{s}$ for $r\in[s,t]$, we have $\left|\Gamma(s,t)\right| \lesssim_{\tau,\lambda,p_{1},p_{2}} w_{1}(s,t)^{1+\delta},$ where the control $w_1(s,t)$ is defined by \eqref{e:w-young}, and the constant $\delta$ is defined by \eqref{e:delta-A}. Then, by Lemma~\ref{lem:control's property}, we have  $\int_{0}^{t}y_r dM_r = \int_{0}^{t}y_r \eta(dr,x_r)$.
\end{remark}

The following result is a variant of the estimate~\eqref{e:p-var2} in Proposition~\ref{prop:bounds}.
\begin{lemma}\label{lem:estimate for Young integral-for g is bounded}
Assume the same conditions as in Proposition~\ref{prop:bounds} with $p_{1} = p_{2} = p$, and let $\varepsilon\in(0,1)$ be a constant such that $\tau + \frac{1-\varepsilon}{p}>1$. Then, we have
\begin{equation}\label{e:p-var3}
\begin{aligned}
\Big\|\int_{\cdot}^{t}y_{r}\eta(dr,x_{r})\Big\|_{\frac{1}{\tau}\text{-}\mathrm{var};[s,t]}&\lesssim_{\tau,\lambda,p,\varepsilon} \|\eta\|_{\tau,\lambda;\beta} |t-s|^{\tau} \Big\{ (1 + \|x\|^{\beta}_{\infty;[s,t]})\|x\|^{\lambda}_{p\text{-}\mathrm{var};[s,t]}\|y\|_{\infty;[s,t]}\\
& \quad +\left(1 + \|x\|^{\lambda + \beta}_{\infty;[s,t]}\right)\Big(\|y\|_{\infty;[s,t]} + \|y\|^{\varepsilon}_{\infty;[s,t]}\|y\|^{1-\varepsilon}_{p\text{-}\mathrm{var};[s,t]}\Big) \Big\}.
\end{aligned}
\end{equation}
\end{lemma}

\begin{proof}
See \nameref{proof:epsilon} in Appendix~\ref{append:B}.
\end{proof}

\begin{remark}\label{rem:interpolation}
Unlike \eqref{e:p-var2}, \eqref{e:p-var3} is designated to estimate $ \|\int_{\cdot}^{t} g(y_{r})\eta(dr,x_{r}) \|_{\frac{1}{\tau}\text{-}\mathrm{var};[s,t]}$ when $g$ is a \emph{bounded} function. In this case, $\|g(y_\cdot)\|_{\infty;[s,t]}$ reduces to a constant, while the power of $\|g(y_\cdot)\|_{p\text{-}\mathrm{var};[s,t]}$ decreases from $1$ to $1-\varepsilon$. This allows us to estimate $\|Y\|_{p\text{-}\mathrm{var};[s,t]}$ for the solution
$Y$ of a BSDE by Young's inequality, as will be done in the proofs of Proposition~\ref{prop:SPDE-estimate for Gamma} and Proposition~\ref{prop:L^k estimate on stopping intervals}.
\end{remark}	

It is straightforward to obtain, for 
$p\ge 1$ and $g\in C^1(\R;\R)$,  
\begin{equation*}
\|g(y) \|_{p\text{-var}} \le \|g'\|_{\infty;\R} \|y \|_{p\text{-var}}. 
\end{equation*}		
The following estimate for $\|g(x) - g(y)\|_{p\text{-var}}$ is a multi-dimensional version of \cite[Lemma~1]{DiehlZhang}. 
\begin{lemma}\label{lem:delta-g}
If $g\in C^{2}(\mathbb{R}^{N};\mathbb{R}^{N\times M})$ with bounded first and second derivatives, and if the functions $x$ and $y$ have finite $p$-variation, then we have
\begin{equation*}
\begin{aligned}
&\left\|g(x)-g(y)\right\|_{p\text{-}\mathrm{var};[s,t]}\\
&  \lesssim \|\nabla g\|_{\infty;\mathbb{R}^N}\left\|x-y\right\|_{p\text{-}\mathrm{var};[s,t]}+\|\nabla^2 g\|_{\infty;\mathbb{R}^N}\Big(\|x\|_{p\text{-}\mathrm{var};[s,t]}+ \left\|y\right\|_{p\text{-}\mathrm{var};[s,t]}\Big)\left\|x-y\right\|_{\infty;[s,t]}.
\end{aligned}
\end{equation*}
\end{lemma}

\begin{proof}
See \nameref{proof:g(Y)} in Appendix~\ref{append:B}.
\end{proof}

\subsection{An extended It\^o's formula}\label{sec:Ito's formula}

Given progressively measurable
processes $X, f^{i}, g^{i}_{j}$, and $Z^{i},$ with $\int_0^T|f^{i}_{r}|dr + \int_{0}^{T}|Z^{i}_{r}|^{2} dr < \infty \ a.s.\ \text{for } i=1,...,n,$ suppose $Y=\left(Y^{1},...,Y^{n}\right)^{\top}$ satisfies 
\begin{equation}\label{Y-decomp}	
Y^{i}_{t} = Y^{i}_{0} - \int_{0}^{t} f^{i}_{r} dr - \sum_{j=1}^{M}\int_{0}^{t}g^{i}_{j}(r) \eta_{j}(dr,X_{r}) + \int_{0}^{t}Z^{i}_{r}dW_{r} ,\ i=1,...,n,\ t\in[0,T],
\end{equation}
where $\eta \in C^{\tau,\lambda}_{\text{loc}}([0,T]\times\mathbb{R}^{d};\mathbb{R}^{M})$ and $(X, g^{i}_{j}) \in C^{p_{1}\text{-var}}([0,T];\mathbb{R}^{d}) \times C^{p_{2}\text{-var}}([0,T];\mathbb{R})$ a.s. The following It\^o's formula is an extension of Theorem~3.4 of \cite{HuLe}.

\begin{proposition}[It\^o's formula]\label{prop:Ito's formula} Let $ \tau\in(\frac12, 1]$, $ \lambda\in(0,1]$, $p_1\ge 1$, and  $p_2\ge1$ be parameters such that $\tau + \frac{\lambda}{p_{1}}>1$ and $\tau + \frac{1}{p_{2}}>1$.  Suppose $Y$ satisfies \eqref{Y-decomp}. Then, for any $\phi(t,y)\in C^{\gamma,1}_{\mathrm{loc}}([0,T]\times\mathbb{R}^{n})$ with $\gamma\in (\frac12, 1],$ having continuous spatial partial derivatives $\nabla_{y}\phi$ and $\nabla^{2}_{yy}\phi$, we have
\begin{equation*}
\begin{aligned}
\phi(t,Y_{t}) &=   \phi(0,Y_{0}) + \int_{0}^{t}\ \phi(dr,Y_{r}) -\int_{0}^{t}\nabla_{y}\phi(r,Y_{r})^{\top} f_{r} dr - \sum_{i=1}^{n}\sum_{j=1}^{M}\int_{0}^{t} \partial_{y^i}\phi(r,Y_{r}) g^{i}_{j}(r)\eta_{j}(dr,X_{r})\\
&\quad + \int_{0}^{t}  \nabla_{y}\phi(r,Y_{r})^{\top} Z_{r} dW_{r} + \frac{1}{2}\int_{0}^{t}\text{tr}\left\{\nabla^{2}_{yy}  \phi(r,Y_{r})Z_{r}Z^{T}_{r}\right\} dr,\ t\in[0,T].
\end{aligned}
\end{equation*}
\end{proposition}

\begin{proof}[Proof of Proposition~\ref{prop:Ito's formula}]
\makeatletter\def\@currentlabelname{the proof}\makeatother
\label{proof:Ito}
We only prove the case when $n=d=M=1$; the general case can be proved similarly. Since $\gamma,\tau>\frac{1}{2}$, there exists a real number $p_{3} > 2$  such that $\gamma + \frac{1}{p_{3}} > 1$ and $\tau + \frac{1}{p_{3}}>1$. Furthermore, Since $p_3>2,$ by Proposition~\ref{prop:bounds} and Lemma~\ref{lem:BDG}, we have $Y\in C^{p_{3}\text{-var}}([0,T];\mathbb{R})$ a.s. In addition, since $\phi\in C^{\gamma,1}_{\text{loc}}([0,T]\times\mathbb{R})$ and $\gamma + \frac{1}{p_{3}}>1$, we have, for almost all paths of $Y$,  $\int_{0}^{t}\phi(dr,Y_{r})$ is well defined in the sense of Proposition~\ref{prop:bounds}. Also, by our assumption, $\nabla_{y}\phi(\cdot,x)$ is $\gamma$-H\"older continuous in $[0,T]$ for every fixed $ x\in \mathbb{R}$. 

Let $\left\{\tau_{K}\right\}_{K=1}^{\infty}$   be a sequence of stopping times such that $\tau_{K} \uparrow T$ a.s., and the following terms are all less than $K$ (here $X^{\tau_{K}}_{\cdot} := X_{\cdot\land\tau_{K}}$; $g^{\tau_{K}}$ and $Y^{\tau_{K}}$ are defined in the same manner):
\begin{equation*}
\begin{aligned}
&\|X^{\tau_{K}}_{\cdot}\|_{\infty;[0,T]},\ \|X^{\tau_{K}}_{\cdot}\|_{p_{1}\text{-var};[0,T]},\ \|g^{\tau_{K}}_{\cdot}\|_{\infty;[0,T]},\ \|g^{\tau_{K}}_{\cdot}\|_{p_{2}\text{-var};[0,T]},\ \|Y^{\tau_{K}}_{\cdot}\|_{\infty;[0,T]},\ \|Y^{\tau_{K}}_{\cdot}\|_{p_{3}\text{-var};[0,T]},\\
&\int_{0}^{\tau_{K}}|Z_{r}|^{2}dr,\ \Big\|\int_{0}^{\cdot\land\tau_{K}}Z_{r}dW_{r}\Big\|_{p_{3}\text{-var};[0,T]},\ \int_{0}^{\tau_{K}}|f_{r}|dr,\ \text{and}\ \Big\|\int_{0}^{\cdot\land \tau_{K}}g_{r}\eta(dr,X_{r})\Big\|_{\frac{1}{\tau}\text{-var};[0,T]}.
\end{aligned}
\end{equation*}
Denote $B^{K}:=\{x\in\R;|x|\le K\}$ and  $\|\eta\|^{(K)}_{\tau,\lambda} := \|\eta\|^{(B^{K})}_{\tau,\lambda}$ for simplicity; the notation $\|\phi\|^{(K)}_{\gamma,1}$ is defined in the same manner. Hence, letting $\pi$ be a partition on $[0,t]$, we have 
\begin{equation*}
\begin{aligned}
\Bigg(&\sum_{[t_{i},t_{i+1}]\in\pi}\left|\partial_{y}\phi(t_{i+1},Y^{\tau_{K}}_{t_{i+1}}) - \partial_{y}\phi(t_{i},Y^{\tau_{K}}_{t_{i}})\right|^{p_{3}}\Bigg)^{\frac{1}{p_{3}}}\\
&\le \left(\sum_{i}\left|\partial_{y}\phi(t_{i+1},Y^{\tau_{K}}_{t_{i+1}}) - \partial_{y}\phi(t_{i+1},Y^{\tau_{K}}_{t_{i}})\right|^{p_{3}}\right)^{\frac{1}{p_{3}}} + \left(\sum_{i}\left|\partial_{y}\phi(t_{i+1},Y^{\tau_{K}}_{t_{i}}) - \partial_{y}\phi(t_{i},Y^{\tau_{K}}_{t_{i}})\right|^{p_{3}}\right)^{\frac{1}{p_{3}}}\\
&\le \sup_{r\in[0,T],|y|\le K} |\partial^{2}_{yy}\phi(r,y)|\left(\sum_{i}\left|Y^{\tau_{K}}_{t_{i+1}} - Y^{\tau_{K}}_{t_{i}}\right|^{p_{3}}\right)^{\frac{1}{p_{3}}} + \|\phi\|^{(K)}_{\gamma,1} \left(\sum_{i}\left|t_{i+1}-t_{i}\right|^{\gamma \cdot p_{3}}\right)^{\frac{1}{p_{3}}}.
\end{aligned}
\end{equation*}
Therefore, noting that $\gamma\cdot p_{3} > p_{3} - 1 > 1,$  $\partial_{y}\phi(\cdot,Y^{\tau_{K}}_{\cdot})$ has finite $p_{3}$-variation. Hence, since $\tau + \frac{1}{p_{3}}>1$, the nonlinear Young integral $\int_{a}^{b}\partial_{y}\phi(r,Y^{\tau_{K}}_{r})\eta(dr,X_{r})$ is well-defined for $0\le a <b \le T$. Note that
\begin{equation}\label{decomI}	
\begin{aligned}
\phi(t,Y^{\tau_{K}}_{t}) - \phi(0,Y_{0}) &= \sum_{[t_{i},t_{i+1}]\in \pi}\left[\phi(t_{i+1},Y^{\tau_{K}}_{t_{i+1}}) - \phi(t_{i},Y^{\tau_{K}}_{t_{i}})\right]  =: I_{1} + I_{2},
\end{aligned} 
\end{equation}
where \[I_{1} := \sum_{[t_{i},t_{i+1}]\in\pi}\left[\phi(t_{i+1},Y^{\tau_{K}}_{t_{i+1}}) - \phi(t_{i+1},Y^{\tau_{K}}_{t_{i}})\right]  \text{ and } I_{2} := \sum_{[t_{i},t_{i+1}]\in\pi}\left[\phi(t_{i+1},Y^{\tau_{K}}_{t_{i}}) - \phi(t_{i},Y^{\tau_{K}}_{t_{i}})\right]
.\]

We first deal with the term $I_{2}$, which can be written as	follows,	
\begin{equation*}
\begin{aligned}
I_{2} & = \sum_{[t_{i},t_{i+1}]\in\pi} \left[\int_{t_{i}}^{t_{i+1}}\phi(dr,Y^{\tau_{K}}_{r}) + \phi(t_{i+1},Y^{\tau_{K}}_{t_{i}}) - \phi(t_{i},Y^{\tau_{K}}_{t_{i}}) - \int_{t_{i}}^{t_{i+1}}\phi(dr,Y^{\tau_{K}}_{r})\right]\\
& = \int_{0}^{t}\phi(dr,Y^{\tau_{K}}_{r}) +\sum_{[t_{i},t_{i+1}]\in\pi}\left[ \phi(t_{i+1},Y^{\tau_{K}}_{t_{i}}) - \phi(t_{i},Y^{\tau_{K}}_{t_{i}}) - \int_{t_{i}}^{t_{i+1}}\phi(dr,Y^{\tau_{K}}_{r})\right].
\end{aligned}
\end{equation*}
By \eqref{e:inequality-estimate-Young} we have 	
\begin{align*}	
\left|\phi(t_{i+1},Y^{\tau_{K}}_{t_{i}}) - \phi(t_{i},Y^{\tau_{K}}_{t_{i}}) - \int_{t_{i}}^{t_{i+1}}\phi(dr,Y^{\tau_{K}}_{r})\right| \lesssim_{\gamma,p_{3}} \|\phi\|^{(K)}_{\gamma,1}(t_{i+1}-t_{i})^{\gamma} \|Y^{\tau_{K}}_{\cdot}\|_{p_{3}\text{-var};[t_{i},t_{i+1}]}.
\end{align*}
Noting that $\gamma + \frac{1}{p_{3}}>1$, by Lemma~\ref{lem:control's property}, it follows that
\begin{equation}\label{e:I2}
\lim_{|\pi|\downarrow 0} I_{2} = \int_{0}^{t}\phi(dr,Y^{\tau_{K}}_{r})\ \ \ \ a.s.	
\end{equation}

In the remaining part of the proof, we analyze $I_{1}$, which is decomposed as follows,
\begin{equation}\label{e:I1}	\begin{aligned}
I_{1} &= \sum_{[t_{i},t_{i+1}]\in \pi}\left[\int_{t_{i}}^{t_{i+1}}\partial_{y}\phi(r,Y^{\tau_{K}}_{r})dY^{\tau_{K}}_{r} + \frac{1}{2}\int_{t_{i}}^{t_{i+1}}\partial^{2}_{yy}\phi(r,Y^{\tau_{K}}_{r})|Z_{r}|^{2}\mathbf{1}_{[0,\tau_{K}]}(r) dr\right]\\
& \qquad + J_{1} + J_{2} + J_{3},
\end{aligned}
\end{equation}
where $dY^{\tau_{K}}_{r} := -f_{r}\mathbf{1}_{[0,\tau_{K}]}(r)dr - g_{r}\mathbf{1}_{[0,\tau_{K}]}(r)\eta(dr,X_{r}) + Z_{r}\mathbf{1}_{[0,\tau_{K}]}(r) dW_{r},$ and
\begin{align*}
J_{1} &:= \sum_{[t_{i},t_{i+1}]\in \pi} \Bigg[\phi(t_{i+1},Y^{\tau_{K}}_{t_{i+1}}) - \phi(t_{i+1},Y^{\tau_{K}}_{t_{i}}) - \int_{t_{i}}^{t_{i+1}}\partial_{y}\phi(t_{i+1},Y^{\tau_{K}}_{r}) dY^{\tau_{K}}_{r}\\
&\qquad  - \frac{1}{2}\int_{t_{i}}^{t_{i+1}}\partial^{2}_{yy}\phi(t_{i+1},Y^{\tau_{K}}_{r})|Z_{r}|^{2}\mathbf{1}_{[0,\tau_{K}]}(r)dr\Bigg],\\
J_{2} &:= \frac{1}{2}\sum_{[t_{i},t_{i+1}]\in \pi}\left[\int_{t_{i}}^{t_{i+1}}\partial^{2}_{yy}\phi(t_{i+1},Y^{\tau_{K}}_{r})|Z_{r}|^{2}\mathbf{1}_{[0,\tau_{K}]}(r)dr - \int_{t_{i}}^{t_{i+1}}\partial^{2}_{yy}\phi(r,Y^{\tau_{K}}_{r})|Z_{r}|^{2}\mathbf{1}_{[0,\tau_{K}]}(r)dr\right],\\
J_{3} &:= \sum_{[t_{i},t_{i+1}]\in\pi}\left[\int_{t_{i}}^{t_{i+1}}\big(\partial_{y}\phi(t_{i+1},Y^{\tau_{K}}_{r}) - \partial_{y}\phi(r,Y^{\tau_{K}}_{r})\big) dY^{\tau_{K}}_{r}\right].
\end{align*}
Thus, to prove the desired result, it suffices to prove $|J_1|+|J_2| +|J_3|$ converges to 0 in probability as $|\pi|$ tends to 0. For $J_{2}$, due to the uniform continuity of $\partial^{2}_{yy}\phi(r,y)$ on compact sets and the mean value theorem, we have 
$\lim_{|\pi|\downarrow 0}J_{2} = 0,\ a.s.$  The proof of the convergence of $J_1$, $J_3$ will be separated into the following two steps.

\textbf{Step 1.}
In this step, we shall show that $\lim\limits_{|\pi|\downarrow 0}J_{3} = 0$ a.s.  Note that
\begin{equation*}	\begin{aligned}
J_{3} &= \sum_{[t_{i},t_{i+1}]\in \pi} \int_{t_{i}}^{t_{i+1}}\big(\partial_{y}\phi(r,Y^{\tau_{K}}_{r}) - \partial_{y}\phi(t_{i+1},Y^{\tau_{K}}_{r})\big) f_r \mathbf{1}_{[0,\tau_{K}]}(r)dr\\
&\quad + \sum_{[t_{i},t_{i+1}]\in \pi} \int_{t_{i}}^{t_{i+1}}\big(\partial_{y}\phi(t_{i+1},Y^{\tau_{K}}_{r}) - \partial_{y}\phi(r,Y^{\tau_{K}}_{r})\big) Z_r \mathbf{1}_{[0,\tau_{K}]}(r)dW_{r}\\
&\quad +\sum_{[t_{i},t_{i+1}]\in \pi}
\int_{t_{i}}^{t_{i+1}}\big(\partial_{y}\phi(r,Y^{\tau_{K}}_{r}) - \partial_{y}\phi(t_{i+1},Y^{\tau_{K}}_{r})\big) g_r \mathbf{1}_{[0,\tau_{K}]}(r) \eta(dr,X_{r}) =: K_1 + K_2 +K_3.
\end{aligned}
\end{equation*}
The first summation $K_1$ tends to zero a.s. as $|\pi|\to 0$ due to the uniform continuity of $\partial_y \phi(t, y)$ on compact sets. For $K_{2}$, It\^o's isometry yields that
\begin{equation*}		\begin{aligned}	\mathbb{E}\left[|K_{2}|^{2}\right] = \sum_{[t_{i},t_{i+1}]\in\pi} \E\left[\int_{t_{i}}^{t_{i+1}} |\partial_{y}\phi(t_{i+1},Y^{\tau_{K}}_{r})-\partial_{y}\phi(r,Y^{\tau_{K}}_{r})|^{2}|Z_{r}|^{2}\mathbf{1}_{[0,\tau_{K}]}(r) dr\right].
\end{aligned}
\end{equation*}
Then, by the dominated convergence theorem, we have
$\lim_{|\pi|\downarrow 0}\mathbb{E}\left[|K_{2}|^{2}\right] = 0.$			
As for $K_{3}$, denote $p := p_{2}\vee p_{3}$ and  $l := \sup_{[t_{i},t_{i+1}]\in \pi}\left\{i;\ t_{i} < \tau_{K}\right\}$, we have
\begin{align}\label{e:k2}\nonumber
|K_{3}|
&\le \sum_{[t_{i},t_{i+1}]\in \pi} \left\|\int_{\cdot}^{t_{i+1}\land \tau_{K}}(\partial_{y}\phi(t_{i+1}\land \tau_{K},Y^{\tau_{K}}_{r}) - \partial_{y}\phi(r,Y^{\tau_{K}}_{r}))g_{r}\eta(dr,X_{r})\right\|_{\frac{1}{\tau}\text{-var};[t_{i}\land \tau_{K},t_{i+1}\land \tau_{K}]}\\  
&\qquad + \left\|\int_{\cdot}^{\tau_{K}}(\partial_{y}\phi( \tau_{K},Y^{\tau_{K}}_{r}) - \partial_{y}\phi(t_{l+1},Y^{\tau_{K}}_{r}))g_{r}\eta(dr,X_{r})\right\|_{\frac{1}{\tau}\text{-var};[t_{l},\tau_{K}]}.
\end{align}
By \eqref{e:p-var1}, the second term on the right-hand side of \eqref{e:k2} is bounded by  (up to a multiplicative constant)
\begin{align*}						
&\|\eta\|^{(K)}_{\tau,\lambda}\left(1+\|X^{\tau_{K}}_{\cdot}\|^{\lambda}_{p_{1}\text{-var};[t_{l},\tau_K]}\right)|\tau_K-t_l|^\tau \Big(\|(\partial_{y}\phi(\tau_{K},Y^{\tau_{K}}_{\cdot}) - \partial_{y}\phi(t_{l+1},Y^{\tau_{K}}_{\cdot}))g_{\cdot}\|_{\infty;[t_{l},\tau_{K}]}\\
&\qquad  + \|(\partial_{y}\phi(\tau_{K},Y^{\tau_{K}}_{\cdot}) - \partial_{y}\phi(t_{l+1},Y^{\tau_{K}}_{\cdot}))g_{\cdot}\|_{p\text{-var};[t_{l},\tau_{K}]}\Big)\\
&\quad \lesssim  \|\eta\|^{(K)}_{\tau, \lambda} (1+K^{\lambda})|\pi|^{\tau }\Big( K\sup_{r\in[0,T], |y|\le K}|\partial_{y}\phi(r,y)| + \sup_{r\in[0,T]}\| \partial_{y}\phi(r,Y^{\tau_{K}}_{\cdot})g_{\cdot}\|_{p\text{-var};[t_{l},\tau_{K}]}\Big),
\end{align*}
which converges to 0 as $|\pi|\to0$ in view of the uniform boundedness of $\| \partial_{y}\phi(r,Y^{\tau_{K}}_{\cdot})g_{\cdot}\|_{p\text{-var};[t_{l},\tau_{K}]}$. For the first term on the right-hand side of \eqref{e:k2}, by \eqref{e:p-var1} and the fact that 
\begin{align*}		
\big(\partial_{y}\phi(t_{i+1}\land{\tau_{K}},Y^{\tau_{K}}_{r}) - \partial_{y}\phi(r,Y^{\tau_{K}}_{r})\big)\big|_{r=t_{i+1}\land\tau_{K}}=0, 
\end{align*}
it is bounded by (up to a multiplicative constant)
\begin{equation}\label{e:partial phi1}
\|\eta\|^{(K)}_{\tau, \lambda} (1+K^{\lambda}) \sum_{[t_i, t_{i+1}]\in\pi} |t_{i+1} - t_{i}|^{\tau} \big\|\big(\partial_{y}\phi(t_{i+1},Y^{\tau_{K}}_{\cdot}) - \partial_{y}\phi(\cdot,Y^{\tau_{K}}_{\cdot})\big)g_{\cdot}\big\|_{p\text{-var};[t_{i}\land \tau_{K},t_{i+1}\land \tau_{K}]}.
\end{equation}
The $p$-variation above can be bounded as follows,
\begin{align}	\nonumber
\big\|&(\partial_{y}\phi(t_{i+1},Y^{\tau_{K}}_{\cdot}) - \partial_{y}\phi(\cdot,Y^{\tau_{K}}_{\cdot}))g_{\cdot}\big\|_{p\text{-var};[t_{i}\land \tau_{K},t_{i+1}\land \tau_{K}]}\\ \nonumber
&\le 	\big\|(\partial_{y}\phi(t_{i+1},Y^{\tau_{K}}_{\cdot}) - \partial_{y}\phi(\cdot,Y^{\tau_{K}}_{\cdot}))\big\|_{p_{3}\text{-var};[t_{i}\land\tau_{K},t_{i+1}\land\tau_{K}]}\|g_{\cdot}\|_{\infty ;[0,t\land\tau_{K}]}\\ \label{e:partial phi2}
&\quad +\big\|(\partial_{y}\phi(t_{i+1},Y^{\tau_{K}}_{\cdot}) - \partial_{y}\phi(\cdot,Y^{\tau_{K}}_{\cdot}))\big\|_{\infty;[t_{i}\land\tau_{K},t_{i+1}\land\tau_{K}]}\|g_{\cdot}\|_{p_{2}\text{-var} ;[0,t\land\tau_{K}]}\\ 
\nonumber
& \le \big\|(\partial_{y}\phi(t_{i+1},Y^{\tau_{K}}_{\cdot}) - \partial_{y}\phi(\cdot,Y^{\tau_{K}}_{\cdot}))\big\|_{p_{3}\text{-var};[t_{i}\land\tau_{K},t_{i+1}\land\tau_{K}]}\Big(\|g_{\cdot}\|_{\infty ;[0,t\land\tau_{K}]}  + \|g_{\cdot}\|_{p_{2}\text{-var} ;[0,t\land\tau_{K}]} \Big).
\end{align}
We further estimate $	\big\|(\partial_{y}\phi(t_{i+1},Y^{\tau_{K}}_{\cdot}) - \partial_{y}\phi(\cdot,Y^{\tau_{K}}_{\cdot}))\big\|_{p_{3}\text{-var};[t_{i}\land\tau_{K},t_{i+1}\land\tau_{K}]}$ in the above inequality.  For any partition $\tilde \pi$ on $[t_i, t_{i+1}]$, we have
\begin{equation}\label{e:partial phi3}
\begin{aligned}
\bigg\{&\sum_{[s,u]\in \tilde \pi}\Big|\partial_{y}\phi(t_{i+1},Y^{\tau_{K}}_{u}) - \partial_{y}\phi(u,Y^{\tau_{K}}_{u}) - \partial_{y}\phi(t_{i+1},Y^{\tau_{K}}_{s}) + \partial_{y}\phi(s,Y^{\tau_{K}}_{s})\Big|^{p_{3}}\bigg\}^{\frac{1}{p_{3}}}\\
&\le \bigg\{\sum_{[s,u]\in \tilde \pi}\Big|\partial_{y}\phi(t_{i+1},Y^{\tau_{K}}_{u}) - \partial_{y}\phi(t_{i+1},Y^{\tau_{K}}_{s}) - \partial_{y}\phi(u,Y^{\tau_{K}}_{u}) + \partial_{y}\phi(u,Y^{\tau_{K}}_{s})\Big|^{p_{3}}\bigg\}^{\frac{1}{p_{3}}}\\
&\quad+ \bigg\{\sum_{[s,u]\in \tilde \pi}\Big|-\partial_{y}\phi(u,Y^{\tau_{K}}_{s}) + \partial_{y}\phi(s,Y^{\tau_{K}}_{s})\Big|^{p_{3}}\bigg\}^{\frac{1}{p_{3}}}\\
&\le 2\sup_{r\in[0,t],|y|\le K}|\partial^{2}_{yy}\phi(r,y)|\cdot \|Y^{\tau_{K}}_{\cdot}\|_{p_{3}\text{-var};[t_{i},t_{i+1}]} + \|\phi\|^{(K)}_{\gamma,1}\cdot |t_{i+1}-t_{i}|^{\gamma}.
\end{aligned}
\end{equation}
In view of \eqref{e:partial phi1}, \eqref{e:partial phi2}, and \eqref{e:partial phi3}, the first term on the right-hand side of \eqref{e:k2} is bounded by (up to a multiplicative constant that depends on $K$)
\begin{align*}
\left(|\pi|^{\tau} +  \sum_{[t_{i},t_{i+1}]\in\pi}|t_{i+1}-t_{i}|^{\tau}\Big(\|Y^{\tau_{K}}_{\cdot}\|_{p_{3}\text{-var};[t_{i},t_{i+1}]} +|t_{i+1}-t_{i}|^{\gamma}\Big)\right),
\end{align*}
which converges to 0 as $|\pi|\to0$ by Lemma~\ref{lem:control's property}.  
Combining the above analysis on the right-hand side of \eqref{e:k2}, we have
$\lim_{|\pi|\downarrow 0}|K_{3}|  = 0,\  a.s.$ Therefore, we have $\lim_{|\pi|\downarrow 0} J_3 = 0$ a.s.

\textbf{ Step 2.} In this step, we will prove $\lim_{|\pi|\downarrow 0} |J_1|=0$. Taylor's expansion yields that
\begin{equation*}					
\begin{aligned}
J_{1} &= \sum_{[t_{i},t_{i+1}]\in\pi}\Big[ \partial_{y}\phi(t_{i+1},Y^{\tau_{K}}_{t_{i}})\left(Y^{\tau_{K}}_{t_{i+1}}-Y^{\tau_{K}}_{t_{i}}\right) - \int_{t_{i}}^{t_{i+1}}\partial_{y}\phi(t_{i+1},Y^{\tau_{K}}_{r})dY^{\tau_{K}}_{r}\Big]\\
&\qquad + \sum_{[t_{i},t_{i+1}]\in\pi}\Big[ \frac{1}{2} \partial^{2}_{yy}\phi(t_{i+1},\xi_{i})\left(Y^{\tau_{K}}_{t_{i+1}}-Y^{\tau_{K}}_{t_i}\right)^{2}  - \frac{1}{2}\int_{t_{i}}^{t_{i+1}}\partial^{2}_{yy}\phi(t_{i+1},Y^{\tau_{K}}_{r})|Z_{r}|^{2}\mathbf{1}_{[0,\tau_{K}]}(r)dr\Big]\\[1mm]
&=: L_{1} + L_{2},
\end{aligned}
\end{equation*}
where $\xi_{i}$ is a (random) number between $Y_{t_{i}}^{\tau_K}$ and $Y_{t_{i+1}}^{\tau_K}$. 
By a similar argument for $J_3$, we can show that $L_1$ converges to 0 in probability as $|\pi|\to0$. As for $L_{2}$, denote $Y^{\tau_{K}}_{t_{i+1}}-Y^{\tau_{K}}_{t_{i}} =: \alpha_i+\beta_i+\gamma_i$,
where $\alpha_i =  - \int_{t_i}^{t_{i+1}} f_{r} \mathbf{1}_{[0,\tau_{K}]}(r) dr, \beta_i= - \int_{t_i}^{t_{i+1}}g_{r} \mathbf{1}_{[0,\tau_{K}]}(r)\eta(dr,X_{r})$, and $\gamma_i= \int_{t_i}^{t_{i+1}} \mathbf{1}_{[0,\tau_{K}]}(r)Z_{r}dW_{r}.$ Then we have	
\begin{equation}\label{e:estimate-Ito}	
\begin{aligned}
|L_{2}| &\lesssim \sum_{[t_{i},t_{i+1}]\in\pi} \frac{1}{2} |\partial^{2}_{yy}\phi(t_{i+1},\xi_{i})|\left( \alpha_i^{2} +\left|\alpha_i\gamma_i\right| +\beta_i^{2} + \left|\beta_i\gamma_i\right| + |\alpha_i \beta_i|\right)\\
&\quad +\sum_{[t_{i},t_{i+1}]\in\pi}\left| \frac{1}{2} \partial^{2}_{yy}\phi(t_{i+1},\xi_{i}) \gamma_i^{2} - \frac{1}{2} \partial^{2}_{yy}\phi(t_{i+1},Y^{\tau_{K}}_{t_{i}}) \gamma_i^{2}\right|\\						
&\quad + \sum_{[t_{i},t_{i+1}]\in\pi}\left|\frac{1}{2} \partial^{2}_{yy}\phi(t_{i+1},Y^{\tau_{K}}_{t_{i}})  \gamma_i^{2}  - \frac{1}{2}\int_{t_{i}}^{t_{i+1}}\partial^{2}_{yy}\phi(t_{i+1},Y^{\tau_{K}}_{r})|Z_{r}|^{2}\mathbf{1}_{[0,\tau_{K}]}(r)dr\right|.
\end{aligned}
\end{equation}
We claim that $\sum_{[t_{i},t_{i+1}]\in\pi} \left( \alpha_i^{2} +\left|\alpha_i\gamma_i\right| +\beta_i^{2} + \left|\beta_i\gamma_i\right| + |\alpha_i \beta_i|\right)$ converges to 0 in probability as $|\pi|\to0$. Indeed, by Proposition~\ref{prop:bounds}, we have
\begin{align*}
\sum_{i}\left|\beta_i \gamma_i\right|		& \le \sum\limits_{i}\left\|\int_{\cdot}^{t_{i+1}}g_{r}\mathbf{1}_{[0,\tau_{K}]}(r)\eta(dr,X_{r})\right\|_{\frac{1}{\tau}\text{-var};[t_{i},t_{i+1}]} \left\|\int_{\cdot}^{t_{i+1}}Z_{r}\mathbf{1}_{[0,\tau_{K}]}(r)dW_{r}\right\|_{p_{3}\text{-var};[t_{i},t_{i+1}]}\\
&\lesssim_{\tau,\lambda,p_{1},p_{2}}\|\eta\|^{(K)}_{\tau,\lambda}(1+K^{\lambda})\cdot K\cdot \sum_{i}|t_{i+1}-t_{i}|^{\tau}\left\|\int_{\cdot}^{t_{i+1}}Z_{r}\mathbf{1}_{[0,\tau_{K}]}(r)dW_{r}\right\|_{p_{3}\text{-var};[t_{i},t_{i+1}]},
\end{align*}
which converges to 0 a.s. as $|\pi|\to 0$ by Lemma~\ref{lem:control's property} and the fact that $\tau + \frac{1}{p_{3}} > 1$. Similar results hold for the other terms. For the second summation in \eqref{e:estimate-Ito}, we have	
\begin{equation*}
\sum_{i}\frac{1}{2}\left|\partial^{2}_{yy}\phi(t_{i+1},\xi_{i}) - \partial^{2}_{yy}\phi(t_{i+1},Y^{\tau_K}_{t_{i}})\right|\gamma_i^{2}\le \frac{1}{2}\sup_{i}\left|\partial^{2}_{yy}\phi(t_{i+1},\xi_{i})-\partial^{2}_{yy}\phi(t_{i+1},Y^{\tau_K}_{t_{i}})\right| \sum_{i}\gamma_i^{2},
\end{equation*}
where the right-hand side converges to 0 in probability due to the uniform continuity of $\partial^2_{yy}\phi$ on compact sets, 
and the fact that $\sum_{i}\gamma_i^{2}$ converges in probability  to $\int_{0}^{t} |Z_{r}|^{2}\mathbf{1}_{[0,\tau_{K}]}(r) dr$ as $|\pi|\to0$.
Denote by $\frac12H$ the third summation in \eqref{e:estimate-Ito}, we have 
\begin{align*}
H =& \sum_{[t_{i},t_{i+1}]\in\pi}\left(\partial^{2}_{yy}\phi(t_{i+1},Y^{\tau_{K}}_{t_{i}}) \int_{t_{i}}^{t_{i+1}}|Z_{r}|^{2}\mathbf{1}_{[0,\tau_{K}]}(r)dr - \int_{t_{i}}^{t_{i+1}}\partial^{2}_{yy}\phi(t_{i+1},Y^{\tau_{K}}_{r})|Z_{r}|^{2}\mathbf{1}_{[0,\tau_{K}]}(r)dr\right)	\\
&\quad + \sum\limits_{[t_{i},t_{i+1}]\in\pi}\frac{1}{2} \partial^{2}_{yy}\phi(t_{i+1},Y^{\tau_{K}}_{t_{i}}) \left[\left(\int_{t_{i}}^{t_{i+1}}Z_{r}\mathbf{1}_{[0,\tau_{K}]}(r)dW_{r}\right)^{2} - \int_{t_{i}}^{t_{i+1}}|Z_{r}|^{2}\mathbf{1}_{[0,\tau_{K}]}(r)dr\right]\\[1mm]
=: & \ H_{1} + H_{2}.
\end{align*}
Similar to the second summation on the right-hand side of \eqref{e:estimate-Ito}, we have $H_1$ and $H_2$ converging to zero in probability, and thus $\lim_{|\pi|\downarrow 0}|J_1| = 0$ in probability.

\ \ \ \ \

In view of \eqref{decomI}, \eqref{e:I2}, and \eqref{e:I1}, letting $|\pi|\rightarrow 0$, we have
\begin{equation*}
\begin{aligned}
\phi(t,Y_{t\land \tau_{K}}) =   \phi(0,Y_{0}) &+ \int_{0}^{t}  \phi(dr,Y^{\tau_{K}}_{r}) -\int_{0}^{t\land\tau_{K}}\partial_{y}  \phi(r,Y^{\tau_{K}}_{r}) f_{r} dr - \int_{0}^{t\land\tau_{K}} \partial_{y}  \phi(r,Y^{\tau_{K}}_{r})g_{r}\eta(dr,X_{r})\\
&+ \int_{0}^{t\land\tau_{K}}  \partial_{y} \phi(r,Y^{\tau_{K}}_{r})Z_{r} dW_{r} + \frac{1}{2}\int_{0}^{t\land\tau_{K}}\partial^{2}_{yy}  \phi(r,Y^{\tau_{K}}_{r})|Z_{r}|^{2} dr.
\end{aligned}
\end{equation*}
Finally, our desired formula follows if we let $K$ tend to infinity. 
\end{proof}

The following product rule is a direct consequence of Proposition~\ref{prop:Ito's formula}. 
\begin{corollary}[product rule] \label{cor:product rule}
Let $Y^{i},i=1,2$, be stochastic processes that satisfy, for $0\le t\le T$,
\[Y^{i}_{t} = Y_{0} - \int_{0}^{t}f^{i}_{r}dr - \sum_{j=1}^{M}\int_{0}^{t}g^{i}_{j}(r)\eta_{j}(dr,X_{r}) + \int_{0}^{t}Z^{i}_{r}dW_{r},\ i=1,2,\]
where $(f^i,g^{i}_{j},Z^i,\eta_{j},X),$ $i=1,2,$ $j=1,2,...,M,$ satisfy the same assumptions as in Proposition~\ref{prop:Ito's formula}. Then we have
\begin{equation*}
\begin{split}
Y^{1}_{t}Y^{2}_{t}  = & \  Y^{1}_{0}Y^{2}_{0}  - \int_{0}^{t}\left[Y^{1}_{r}f^{2}_{r} + Y^{2}_{r}f^{1}_{r}\right]dr - \sum_{j=1}^{M}\int_{0}^{t}\left[Y^{1}_{r}g^{2}_{j}(r) + Y^{2}_{r}g^{1}_{j}(r)\right] \eta_{j}(dr,X_{r})  \\
&+ \int_{0}^{t} \left[Y^{1}_{r}Z^{2}_{r} + Y^{2}_{r}Z^{1}_{r}\right]dW_{r}  + \int_{0}^{t}Z^{1}_{r}(Z^{2}_{r})^{\top}dr.
\end{split}
\end{equation*}
\end{corollary}

\section{BSDEs with bounded drivers}\label{sec:BSDE with bounded drift}

Throughout this section, we assume that $\eta(\cdot, \cdot)\in C([0,T]\times\mathbb{R}^{d};\mathbb{R}^{M})$ with $\eta(0,\cdot)\equiv 0$. We make the following assumption to guarantee the existence of the nonlinear Young integral $\int \eta(dr, x_r)$ (see Proposition~\ref{prop:bounds}), where $x_{\cdot}\in C^{p\text{-var}}([0,T];\R^d)$ for some $p>2$.

\begin{assumptionp}{(H0)}\label{(H0)}
Let $p, \tau, \lambda$ be parameters satisfying  
\[p>2, \ \tau\in(1/2,1],\ \lambda\in(0,1], \text{ and }\tau+\frac{\lambda}{p}>1.\]
\end{assumptionp}
In many situations, one can choose $p>2$ to be arbitrarily close to $2$ (e.g., $x_t$ is a realization of a semi-martingale), and thus \ref{(H0)} can be replaced by the following weaker assumption:
\begin{assumptionp}{(H0')}\label{(H0')}
Let $\tau,\lambda$ be parameters satisfying  
\[\tau\in(1/2,1],\ \lambda\in(0,1], \text{ and }\tau+\frac{\lambda}{2}>1.\]
\end{assumptionp}

Let $f:\Omega\times[0,T]\times\mathbb{R}^{N}\times\mathbb{R}^{N\times d}\rightarrow\mathbb{R}^{N}$ be a progressively measurable process, $g:\mathbb{R}^{N}\rightarrow\mathbb{R}^{N\times M}$ be a Borel measurable function, $X$ be a continuous adapted process,  and the terminal value $\xi=(\xi_1, \dots, \xi_N)^{\top}$ be an $\mathcal{F}_{T}$-measurable random vector.  

\begin{assumptionp}{(H1)}\label{(H1)} Let $\tau,\lambda,p$ be the parameters satisfying Assumption~\ref{(H0)}.

\begin{itemize}
\item[$(1)$]  $\eta \in C^{\tau,\lambda}([0,T]\times\mathbb{R}^{d};\mathbb{R}^{M})$ and $m_{p,k}(X;[0,T])<\infty$ for some $k>1$, where $m_{p,k}$ is given by \eqref{e:def-norms}.  

\item[$(2)$]  For $(y_{i},z_{i})\in\mathbb{R}^{N}\times\mathbb{R}^{N\times d},\ i=1,2$,
\[|f(\cdot,r,y_{1},z_{1}) - f(\cdot, r,y_{2},z_{2})|\le C_{1}\left(|y_{1} - y_{2}| + |z_{1} - z_{2}|\right)		\text{ and } |f(\cdot,r,0,0)|\le C_{1},\]
and $g\in C^{2}(\mathbb{R}^{N};\mathbb{R}^{N\times M})$ with 
\[|g(0)|\vee\|\nabla g\|_{\infty;\mathbb{R}^{N}}\vee \|\nabla^2 g\|_{\infty;\mathbb{R}^{N}} \le C_{1}, \ \text{for some } C_{1}>0.
\]

\item[$(3)$]  $\xi\in \mathcal F_T$ is an $\mathbb{R}^{N}$-valued random vector  with $\|\xi\|_{L^{\infty}(\Omega)}<\infty$. 
\end{itemize}
\end{assumptionp}

\begin{remark}\label{rem:C}
Clearly $C^{\tau,\lambda}_{\mathrm{loc}}\subsetneq C^{\tau,\lambda}$. However, if we assume further $\|X_T\|_{L^{\infty}(\Omega)}<\infty$, the condition $\eta\in C^{\tau,\lambda}_{\mathrm{loc}}$ is equivalent to $\eta\in C^{\tau,\lambda}$ when studying the nonlinear Young integral involving $\eta(dr,X_r)$. Therefore, the requirement $\eta\in C^{\tau,\lambda}([0,T]\times\R^{d};\R^M)$ in Theorem~\ref{thm:picard} and Theorem~\ref{thm:comparison} below can be replaced by $\eta\in C^{\tau,\lambda}_{\mathrm{loc}}([0,T]\times\R^{d};\R^M)$ if $\|X_T\|_{L^{\infty}(\Omega)}<\infty$.
\end{remark}

\subsection{The well-posedness of equation}\label{sec:picard}

Consider the following BSDE on $[u,v]\subset[0,T]$:
\begin{equation}\label{e:BSDE-Phi'}
Y_{t} =\xi+ \int_{t}^{v}f(r,Y_{r},Z_{r}) dr + \sum_{i=1}^{M}\int_{t}^{v} g_{i}(Y_{r})\eta_{i}(dr,X_{r}) - \int_{t}^{v}Z_{r}dW_{r}, ~ t\in  [u,v].
\end{equation}
In the following, we refer to $(\xi,f,g,\eta,X)$ as parameters of this equation. In addition, we denote by $\Theta:=(\tau,\lambda,p,k)$  the parameters in Assumptions {\bf(H0)} and {\bf(H1)}.

\begin{lemma}\label{lem:BSDE-Phi}
Suppose Assumption~\ref{(H1)} holds. Assume further that $\xi \in\mathcal{F}_{v}$. Then for any $(Y,Z)\in \mathfrak B_{p,k}(u,v)$ (recall $k>1$ in Assumption~\ref{(H1)} and $\mathfrak{B}_{p,k}(u,v)$ in \eqref{e:def of B_p,k}), there exists a unique pair of processes $(\tilde{Y}, \tilde{Z})\in\mathfrak{B}([u,v])$ satisfying \[\sup_{t\in[u,v]}\mathbb{E}\Big[\left|\tilde{Y}_{t}\right|^{k}\Big]+\mathbb{E}\Big[\Big|\int_{u}^{v}|\tilde{Z}_{r}|^{2}dr\Big|^{\frac{k}{2}}\Big]<\infty,\] 
and
\begin{equation}\label{e:BSDE-Phi}
\tilde{Y}_{t} =\xi+ \int_{t}^{v}f(r,Y_{r},Z_{r}) dr + \sum_{i=1}^{M}\int_{t}^{v} g_{i}(Y_{r})\eta_{i}(dr,X_{r}) - \int_{t}^{v}\tilde{Z}_{r}dW_{r}, ~ \text{for all } t\in  [u,v].
\end{equation}
We write $\Phi^{\xi}_{[u,v]}(Y,Z):=(\tilde{Y},\tilde{Z})$ to denote the mapping defined by \eqref{e:BSDE-Phi}. 
\end{lemma}
\begin{proof}
Assume $M=1$ for simplicity. By conditions $\|(Y,Z)\|_{p,k;[u,v]}<\infty$ and $\|\xi\|_{L^{\infty}(\Omega)}<\infty$,  
\begin{equation*}
\begin{aligned}
&\mathbb{E} \left[ \left|\xi\right|^{k} + \Big|\int_{u}^{v} |f(r,Y_{r},Z_{r})| dr\Big|^{k} + \Big\|\int_{\cdot}^{v}g(Y_{r})\eta(dr,X_{r})\Big\|^k_{p\text{-var};[u,v]} \right]\\
&\lesssim_{\Theta} \|\xi\|^{k}_{L^{\infty}(\Omega)} + C^{k}_{1}|v-u|^{k} + C^{k}_{1}(|v-u|^{k} + |v-u|^{\frac{k}{2}}) \|(Y,Z)\|^{k}_{p,k;[u,v]} \\[1mm]
&\quad\quad + C^{k}_{1}\|\eta\|^{k}_{\tau,\lambda}|v-u|^{k\tau}(1+m_{p,k}(X;[u,v])^{k})\|(Y,Z)\|^{k}_{p,k;[u,v]},\\
\end{aligned}
\end{equation*}
where we apply the following inequalities,
\begin{equation*}
\begin{aligned}							
\Big|\int_{u}^{v}|f(r,Y_{r},Z_{r})|dr\Big|^{k}&\le C^{k}_{1}\Big|\int_{u}^{v}1+|Y_{r}|+|Z_{r}|dr\Big|^{k}\\
&\le C^{k}_{1}|v-u|^{k}\Big(1+\esssup_{t\in[u,v],\omega\in\Omega}|Y_{t}|^{k}\Big) + C^{k}_{1}|v-u|^{\frac{k}{2}} \Big|\int_{u}^{v}|Z_{r}|^{2}dr\Big|^{\frac{k}{2}},
\end{aligned}
\end{equation*}
and by \eqref{e:p-var1}
\begin{equation*}
\begin{aligned}
&\Big\|\int_{\cdot}^{v}g(Y_{r})\eta(dr,X_{r})\Big\|_{p\text{-var};[u,v]}^{k}\\
&\lesssim_{\Theta} C^{k}_{1}\|\eta\|^{k}_{\tau,\lambda}|v-u|^{k\tau}\Big((1+\|X\|^{k}_{p\text{-var};[u,v]})\esssup_{t\in[u,v], \omega\in\Omega}|Y_{t}|^{k} + \|Y\|^{k}_{p\text{-var};[u,v]}\Big).
\end{aligned}
\end{equation*}
Hence, we have
\begin{equation}\label{e:k moments of Y}
\sup_{t\in[u,v]}\E\left[\left(\xi + \int_{t}^{v}f(r,Y_{r},Z_{r})dr + \int_{t}^{v}g(Y_{r})\eta(dr,X_{r})\right)^k\right]<\infty.
\end{equation}
Then by  the martingale representation theorem for $L^1$-integrable random variable (see \cite[Theorem~3]{clark1970representation}), there exists a unique  progressively measurable
process $\tilde{Z}$ such that 
\begin{align*}
\int_0^t \tilde{Z}_{r} dW_r =& \mathbb{E}_{t}\left[\xi + \int_{u}^{v}f(r,Y_{r},Z_{r})dr + \int_{u}^{v}g(Y_{r})\eta(dr,X_{r})\right]\\
&- \mathbb{E}\left[\xi + \int_{u}^{v}f(r,Y_{r},Z_{r})dr + \int_{u}^{v}g(Y_{r})\eta(dr,X_{r})\right]\text{ for all }t\in[0, v].
\end{align*}
Applying the lower bound of the Burkholder-Davis-Gundy inequality (BDG inequality for short) and Doob's maximal inequality to the submartingale $M_t:=|\int_{0}^{t}\tilde{Z}_r dW_r|$, we have
\[\mathbb{E}\bigg[\Big|\int_{0}^{v}|\tilde{Z}_{r}|^{2}dr\Big|^{\frac{k}{2}}\bigg]\lesssim_{k}\mathbb{E}\bigg[\Big\|\int_{0}^{\cdot}\tilde{Z}_r dW_{r}\Big\|_{\infty;[0,v]}^{k}\bigg]\lesssim_k \mathbb{E}\bigg[\Big|\int_{0}^{v}\tilde{Z}_r dW_{r}\Big|^{k}\bigg]<\infty.\]
Thus, we have for $t\in[u,v]$,
\begin{align*}
&\mathbb{E}_{t}\left[\xi + \int_{t}^{v}f(r,Y_{r},Z_{r}) dr + \int_{t}^{v} g(Y_{r})\eta(dr,X_{r})\right]\\
& = \xi + \int_{t}^{v}f(r,Y_{r},Z_{r}) dr + \int_{t}^{v} g(Y_{r})\eta(dr,X_{r}) -\int_{t}^{v}\tilde{Z}_{r}dW_{r}.
\end{align*}
Let 
\begin{equation}\label{e:tilde-Y}
\tilde{Y}_{t} = \mathbb{E}_{t}\left[\xi + \int_{t}^{v}f(r,Y_{r},Z_{r}) dr + \int_{t}^{v} g(Y_{r})\eta(dr,X_{r})\right],~ t\in[u,v].
\end{equation}
Then $(\tilde Y, \tilde Z)$ satisfies \eqref{e:BSDE-Phi}, and $\sup_{t\in[u,v]}\mathbb{E}[|\tilde{Y}_{t}|^{k}]<\infty$ by \eqref{e:k moments of Y}. On the other hand, Eq.~\eqref{e:BSDE-Phi} yields that $\tilde Y_t$ must be given by \eqref{e:tilde-Y}, and the uniqueness follows. 
\end{proof}

The following proposition shows that the solution map $\Phi^{\xi}_{[u,v]}$ of \eqref{e:BSDE-Phi} is invariant within the ball $\mathfrak B_{p,k}(R;  [u,v])$ when $v-u$ is sufficiently small and $R$ is sufficiently large.

\begin{proposition}\label{prop:ball-invariance} 
Assume the same conditions as in Lemma~\ref{lem:BSDE-Phi}. Then
there exist $\delta >0$ and $R_0>0$ such that for any $u\in [(v-\delta)\vee 0, v]$, if $(Y,Z)\in \mathfrak B_{p,k}(R;  [u,v])$ for some $R\ge R_0$,   we have
\[\Phi^{\xi}_{[u,v]}(Y,Z) \in\mathfrak B_{p,k}(R;[u,v]),\]
where $\Phi$ is defined by \eqref{e:BSDE-Phi}. Furthermore, if $(\widehat Y,\widehat Z)\in \mathfrak{B}_{p,k}(u,v)$ is a solution of \eqref{e:BSDE-Phi'}, we have 
\[(\widehat Y, \widehat Z)\in \mathfrak B_{p, k}(R_{0};[u,v]).\]

In particular, we can choose $R_0=C(1+\|\xi\|_{L^{\infty}(\Omega)})$, where $C$ is a constant that depends only on $
\Theta$; the choice of $\delta$ is independent of $\|\xi\|_{L^{\infty}(\Omega)}$.  		
\end{proposition}

\begin{proof}
Assume $M=1$ for simplicity. Let $(\tilde{Y},\tilde{Z}) := \Phi^{\xi}_{[u,v]}(Y,Z)$. Then by \eqref{e:tilde-Y}, for $t\in[u,v]$,
\begin{equation}\label{e:|tilde Y|}
\begin{aligned}
|\tilde{Y}_{t}| &\le \|\xi\|_{L^{\infty}(\Omega)} + \mathbb{E}_{t}\left[\int_{t}^{v}|f(r,Y_{r},Z_{r})|dr + \Big\|\int_{\cdot}^{v}g(Y_{r})\eta(dr,X_{r})\Big\|_{p\text{-var};[t,v]}\right]\\
&\lesssim_{\tau,\lambda,p} \|\xi\|_{L^{\infty}(\Omega)} + C_{1}(v-t)\left(1+\mathbb{E}_{t}\left[\|Y\|_{\infty;[t,v]}\right]\right) + C_{1}(v-t)^{\frac{1}{2}}\mathbb{E}_{t}\Big[\Big(\int_{t}^{v}|Z_{r}|^{2}dr \Big)^{\frac{1}{2}}\Big]\\
&\quad\quad + \|\eta\|_{\tau,\lambda}|v - t|^{\tau} \mathbb{E}_{t}\Big[\left(1+\|X\|_{p\text{-var};[t,v]}\right)\|g(Y)\|_{\infty;[t,v]} +\|g(Y)\|_{p\text{-var};[t,v]}\Big] \\
&\lesssim \|\xi\|_{L^{\infty}(\Omega)} + C_{1}(v-t)\left(1+\mathbb{E}_{t}\left[\|Y\|_{\infty;[t,v]}\right]\right) + C_{1}(v-t)^{\frac{1}{2}}\mathbb{E}_{t}\Big[\Big(\int_{t}^{v}|Z_{r}|^{2}dr \Big)^{\frac{1}{2}}\Big]\\
&\quad\quad + C_{1}\|\eta\|_{\tau,\lambda}|v - t|^{\tau} \mathbb{E}_{t}\Big[\left(1+\|X\|_{p\text{-var};[t,v]}\right)\left(1+\|Y\|_{\infty;[t,v]}\right)+\|Y\|_{p\text{-var};[t,v]}\Big],
\end{aligned}
\end{equation}
where the second inequality follows from Proposition~\ref{prop:bounds}. Thus,  we have
\begin{align}\nonumber
\esssup_{t\in[u,v];\omega\in\Omega}|\tilde{Y}_{t}|
&  \lesssim_{\tau,\lambda,p} \|\xi\|_{L^{\infty}(\Omega)} + C_{1}(v-u) (1+\|\xi\|_{L^{\infty}(\Omega)}+m_{p,k}(Y;[u,v]) ) \\ \label{e:tilde Y}     
&\quad\quad\quad + C_{1}(v-u)^{\frac{1}{2}}\|Z\|_{k\text{-BMO};[u,v]} \\
\nonumber&\quad\quad\quad + C_{1}\|\eta\|_{\tau,\lambda}|v-u|^{\tau}\big(1+m_{p,k}(X;[u,v])\big)   \big(1+\|\xi\|_{L^{\infty}(\Omega)}+m_{p, k}(Y;[u,v])\big).
\end{align}
Noting that 			
\begin{equation*}
\int_{t}^{v}\tilde{Z}_{r}dW_{r}= - \tilde{Y}_{t} + \xi + \int_{t}^{v}f(r,Y_{r},Z_{r}) dr + \int_{t}^{v} g(Y_{r})\eta(dr,X_{r}),
\end{equation*}
we have, for $t\in[u,v]$, 
\begin{equation}\label{e:estimate-Z-BMO}
\begin{aligned}
&\mathbb{E}_{t}\left[\Big|\int_{t}^{v}|\tilde{Z}_{r}|^{2}dr\Big|^{\frac{k}{2}}\right]\lesssim_{k} \mathbb{E}_{t}\left[\sup_{s\in{[t,v]}}\Big|\int_{s}^{v}\tilde{Z}_{r}dW_{r}\Big|^{k}\right]\\
&\lesssim_{k} \mathbb{E}_{t}\left[\|\tilde{Y}\|^{k}_{\infty;[t,v]}\right] + \mathbb{E}_{t}\left[\Big|\int_{t}^{v}|f(r,Y_{r},Z_{r})|dr\Big|^{k}\right] + \mathbb{E}_{t}\left[\Big\|\int_{\cdot}^{v}g(Y_{r})\eta(dr,X_{r})\Big\|^{k}_{p\text{-var};[t,v]}\right]\\
&\lesssim_{\Theta} \|\xi\|^{k}_{L^{\infty}(\Omega)} + C_{1}^{k}(v-u)^{k}\left(1+\|\xi\|^{k}_{L^{\infty}(\Omega)}+m_{p,k}(Y;[u,v])^{k}\right) \\
&\quad + C_{1}^{k}(v-u)^{\frac{k}{2}} \|Z\|^{k}_{k\text{-BMO};[u,v]} \\
&\quad   + C_{1}^{k}\|\eta\|^{k}_{\tau,\lambda}|v-u|^{k\tau}\left(1+m_{p,k}(X;[u,v])^{k}\right)  \left(1+\|\xi\|^{k}_{L^{\infty}(\Omega)}+m_{p,k}\left(Y;[u,v]\right)^{k}\right).
\end{aligned}
\end{equation}
Here, the first inequality is due to the BDG inequality, and the third inequality is due to \eqref{e:tilde Y} and 
\begin{equation}\label{3.10}
\begin{aligned}
\mathbb{E}_{t}&\left[\Big\|\int_{\cdot}^{v}g(Y_{r})\eta(dr,X_{r})\Big\|_{p\text{-var};[t,v]}^{k}\right]\\
&\lesssim_{\Theta} C_{1}^{k}\|\eta\|^{k}_{\tau,\lambda}|v-u|^{k\tau}\left(1+m_{p,k}(X;[u,v])^{k}\right)\left(1+\|\xi\|^{k}_{L^{\infty}(\Omega)}+m_{p,k}(Y;[u,v])^{k}\right),
\end{aligned}
\end{equation}
which follows from \eqref{e:p-var1} and the assumption on $g$. Therefore, by \eqref{e:estimate-Z-BMO}, we have
\begin{equation}\label{e:estimate-theorem1-Z}
\begin{aligned}
\|\tilde{Z}\|_{k\text{-BMO};[u,v]}
&\lesssim_{\Theta} \|\xi\|_{L^{\infty}(\Omega)} + C_{1}(v-u)\big(1+\|\xi\|_{L^{\infty}(\Omega)}+m_{p,k}(Y;[u,v])\big)  \\
&\quad\quad + C_{1}(v-u)^{\frac{1}{2}}\|Z\|_{k\text{-BMO};[u,v]} + C_{1}\|\eta\|_{\tau,\lambda}|v-u|^{\tau}\big(1+m_{p,k}(X;[u,v])\big)\\
&\quad\quad + C_{1}\|\eta\|_{\tau,\lambda}|v-u|^{\tau}\big(1+m_{p,k}(X;[u,v])\big) \big(\|\xi\|_{L^{\infty}(\Omega)}+m_{p,k}(Y;[u,v])\big).
\end{aligned}
\end{equation}

As for $\|\tilde{Y}\|_{p\text{-var};[t,v]}$, in view of \eqref{e:BSDE-Phi}, \eqref{3.10}, and \eqref{e:estimate-theorem1-Z}, we have 	
\begin{equation*}
\begin{aligned}
&\mathbb{E}_{t}\left[\|\tilde{Y}\|_{p\text{-var};[t,v]}^{k}\right]\\
&\lesssim_{p,k} \mathbb{E}_{t}\left[\Big|\int_{t}^{v}|f(r,Y_{r},Z_{r})|dr\Big|^{k}\right]+\mathbb{E}_{t}\left[\Big\|\int_{\cdot}^{v}g(Y_{r})\eta(dr,X_{r})\Big\|^{k}_{p\text{-var};[t,v]}\right] + \mathbb{E}_{t}\left[\Big|\int_{t}^{v}|\tilde{Z}_{r}|^{2}dr\Big|^{\frac{k}{2}}\right]\\
&\lesssim_{\Theta}\|\xi\|^k_{L^{\infty}(\Omega)} + C^k_{1}(v-u)^k\big(1+\|\xi\|^k_{L^{\infty}(\Omega)}+m_{p,k}(Y;[u,v])^k\big) + C^k_{1}(v-u)^{\frac{k}{2}}\|Z\|^k_{k\text{-BMO};[u,v]} \\
&\quad\quad + C^k_{1}\|\eta\|^{k}_{\tau,\lambda}|v-u|^{k\tau}\big(1+m_{p,k}(X;[u,v])^k\big)\big(1+\|\xi\|^k_{L^{\infty}(\Omega)}+m_{p,k}(Y;[u,v])^k\big),
\end{aligned}
\end{equation*}
where we apply the  BDG inequality for  $p$-variation in the first inequality (see Lemma~\ref{lem:BDG}).	Thus, 
\begin{equation}\label{e:estimate-theorem1-Y}
\begin{aligned}
m_{p,k}(\tilde{Y};[u,v])
& \lesssim_{\Theta} \big(1 + (v-u)\tilde{C} + (v-u)^{\tau}\tilde{C}\big)\|\xi\|_{L^{\infty}(\Omega)} + \big((v-u) + (v-u)^{\tau}\big)\tilde{C}\\
&\hspace{10mm} + \big((v-u) + (v-u)^{\tau} + (v-u)^{\frac{1}{2}}\big)\tilde{C}|(Y,Z)|_{p,k;[u,v]},
\end{aligned}
\end{equation}
where $\tilde{C} := C_{1}\vee (C_{1}\|\eta\|_{\tau,\lambda})\big(1+m_{p,k}(X;[0,T])\big)$. Combining \eqref{e:estimate-theorem1-Z} and \eqref{e:estimate-theorem1-Y},  we can find a positive constant $C_{\Theta}$ that depends only on $\Theta$ such that, if $|v-u|\le \delta$, the following inequality holds,
\begin{equation}\label{e:prop:ball-invariance-bound of solution}
\begin{aligned}
|(\tilde{Y},\tilde{Z})|_{p,k;[u,v]}&\le C_{\Theta}\|\xi\|_{L^{\infty}(\Omega)} + C_{\Theta}\tilde{C}(\delta + \delta^{\tau} + \delta^{\frac{1}{2}})\left(1 + \|\xi\|_{L^{\infty}(\Omega)}+|Y,Z|_{p,k;[u,v]}\right).
\end{aligned}
\end{equation}
Let $R_{0} := 3(C_{\Theta}\vee 1)(1+\|\xi\|_{L^{\infty}(\Omega)})$ and let $\delta >0$ be a real number that satisfies $C_{\Theta}\tilde{C}\left(\delta+\delta^{\frac{1}{2}}+\delta^{\tau}\right) \le \frac{1}{2}$.  
We then have 
\begin{equation}\label{e:bd-(Y,Z)}
|(\tilde{Y},\tilde{Z})|_{p,k;[u,v]}\le \left(\frac{1}{3}+\frac{1}{6}\right)R_0 + \frac{1}{2} |(Y,Z)|_{p,k;[u,v]}.
\end{equation}
Hence, for every $R\ge R_{0},$ $|(\tilde{Y},\tilde{Z})|_{p,k;[u,v]}\le R$ holds whenever $|(Y,Z)|_{p,k;[u,v]}\le R$.

Finally, if $(\widehat{Y},\widehat{Z})\in \mathfrak{B}_{p,k}(u,v)$ is a solution of \eqref{e:BSDE-Phi'} on $[u,v]$, then \eqref{e:bd-(Y,Z)} holds with $(\tilde{Y}, \tilde Z)$ and $(Y, Z)$ replaced by $(\widehat{Y}, \widehat Z)$, which implies that $|(\widehat{Y},\widehat{Z})|_{p,k;[u,v]}\le R_{0}$.
\end{proof}

In order to obtain a uniform bound for the solution, we establish the following \emph{a priori} estimate.
\begin{lemma}\label{lem:prior estimate}
Assume \ref{(H1)} holds. Let $\delta$ be the same as in Proposition~\ref{prop:ball-invariance}. For any fixed $S\in[0,T]$, if $(Y,Z)$ is a solution of \eqref{e:ourBSDE} on $[S,T]$, then we have
\begin{align}\label{e:bound1}
&\left|(Y,Z)\right|_{p,k;[T_{n+1},T_{n}]}\le C\left(1+\|\xi\|_{L^{\infty}(\Omega)}\right)\exp\left\{C\left\lfloor \frac{T-S}{\delta}\right\rfloor\right\},\quad \text{ for any } n\ge 0,\\ \label{e:bound2}
&\esssup_{t\in[S,T];\omega\in\Omega}|Y_{t}|\le C\left(1+\|\xi\|_{L^{\infty}(\Omega)}\right)\exp\left\{C\left\lfloor \frac{T-S}{\delta}\right\rfloor\right\},
\end{align} 
where $T_{n} := (T-n\delta)\vee S$, $n\ge 0$; $\lfloor a\rfloor$ is the biggest integer not greater than $a$; $C$ is a positive constant that  only depends on $\Theta$. 
\end{lemma}

\begin{proof} 
We prove these estimates by induction. Firstly, note that for any nonnegative sequences $\{r_n, n\ge 0\}$ and $\{R_n, n\ge 1\}$ that satisfy the following equalities with a positive constant $A,$
\begin{equation*}
\left\{\begin{aligned}
r_{n}&=R_{n} + r_{n-1},\\
R_{n}&=A(1+r_{n-1}),
\end{aligned}
\right.
\end{equation*}
we have 
\begin{equation}\label{e:rR}
\left\{\begin{aligned}
R_n & \le A(1+r_0) e^{A(n-1)},\\
r_n &\le (1+r_0) e^{An}.
\end{aligned}
\right.
\end{equation}
Denote $T_{n} := T-n\delta$ for $0 \le n \le \lfloor\frac{T-S}{\delta} \rfloor$.  To prove  the desired inequalities \eqref{e:bound1} and \eqref{e:bound2}, it suffices to show, for $0 \le n \le \lfloor\frac{T-S}{\delta} \rfloor$,			
\begin{equation}\label{e:claim}
\left|(Y,Z)\right|_{p,k;[T_{n+1},T_{n}]} \le R_{n+1}, ~ \esssup_{t\in[T_{n+1},T_{n}];\omega\in\Omega}|Y_{t}|\le r_{n+1},
\end{equation}
where $r_n, R_n$ satisfies \eqref{e:rR} with  $r_{0} = \|\xi\|_{L^{\infty}(\Omega)}$ and $A = 3(C_{\Theta}\vee 1)$ with $C_{\Theta}$ appearing in \eqref{e:prop:ball-invariance-bound of solution}. When $n=0$, on $[u,v]=[T_1, T]$, by \eqref{e:bd-(Y,Z)} we have
\begin{equation*}
\left|(Y,Z)\right|_{p,k;[T_{1},T_{0}]}\le A(1+\|\xi\|_{L^{\infty}(\Omega)}) = A (1 + r_{0}) = R_{1}.
\end{equation*}
Moreover,
\begin{align*}
\esssup_{t\in[T_{1},T_{0}];\omega\in\Omega}\left|Y_{t}\right|&\le \esssup_{t\in[T_{1},T_{0}];\omega\in\Omega}\mathbb{E}_{t}\left[|Y_{T} - Y_{t}|\right] + \|Y_{T}\|_{L^{\infty}(\Omega)}\\
& \le \left|(Y,Z)\right|_{p,k;[T_{1},T_{0}]} + \|\xi\|_{L^{\infty}(\Omega)} \notag\\
& \le R_{1} + r_{0} = r_{1}.\notag
\end{align*}
Now we assume that  \eqref{e:claim} holds  on $[T_{j},T_{j-1}]$ for $1\le j\le n$. By \eqref{e:bd-(Y,Z)} in Proposition~\ref{prop:ball-invariance}, we have
\begin{equation*}
\left|(Y,Z)\right|_{p,k;[T_{n+1},T_{n}]}\le A(1+\|Y_{T_n}\|_{L^{\infty}(\Omega)})\le A\Big(1+\esssup_{t\in[T_{n},T_{n-1}];\omega\in\Omega}|Y_{t}|\Big) \le A (1 + r_{n}) = R_{n+1}.
\end{equation*}
Furthermore, 
\begin{equation*}
\begin{aligned}
\esssup_{t\in[T_{n+1},T_{n}]}\left|Y_{t}\right|&\le \esssup_{t\in[T_{n+1},T_{n}]}\mathbb{E}_{t}\left[|Y_{T_{n}} - Y_{t}|\right] + \|Y_{T_{n}}\|_{L^{\infty}(\Omega)}\\
& \le \left|(Y,Z)\right|_{p,k;[T_{n+1},T_{n}]} + \|Y_{T_{n}}\|_{L^{\infty}(\Omega)} \\
& \le R_{n+1} + r_{n} = r_{n+1}.
\end{aligned}
\end{equation*}
Thus, \eqref{e:claim} is proved by induction. 	
\end{proof}

\begin{remark}\label{rem:differ from DZ}
The approach used in the proof of  Lemma~\ref{lem:prior estimate} for establishing the \emph{a priori} estimate for  $(Y,Z)$  differs from \cite[p.~10]{DiehlZhang}. On one hand, assuming the driver is deterministic in \cite{DiehlZhang}, the comparison theorem allows their  solution $|Y_t|$ to be bounded by solutions of \emph{deterministic} Young ODEs with the terminal value $\xi$ replaced by its essential extremum, which does not hold in our situation in general since  the integral $\int \dots \eta(dr,X_r)$ is stochastic due to the presence of $X$. On the other hand, the estimate of the nonlinear Young integral \eqref{e:|tilde Y|} in Proposition~\ref{prop:ball-invariance} is more accurate than the estimate for Young integral in \cite[Theorem 5]{DiehlZhang} with $\int_{t}^{T} g(Y_r) d\tilde{\eta}^{i}_r = \int_{t}^{T} g(Y_r) \eta_{i}(dr,X_r)$, where $\tilde{\eta}^{i}_{t} := \int_{0}^{t} \eta_i(dr,X_r)$. More precisely,
the latter approach would lead to an extra term of $\|g(Y)\|_{p\text{-}\mathrm{var}}\|X\|^{\lambda}_{p\text{-}\mathrm{var}}$ on the right-hand side of  \eqref{e:|tilde Y|} in the estimation of $Y$, and thus we cannot obtain an inequality for any $k$-moment of $\|Y\|_{p\text{-var}}$ due to the randomness of $\|X\|^{\lambda}_{p\text{-}\mathrm{var}}$.
\end{remark}

Consider the following two BSDEs on $[u,v]\subseteq[0,T]$ with a common terminal value,
\begin{equation*}
\begin{aligned}
\tilde{Y}^{(1)}_{t} &= \xi + \int_{t}^{v}f(r,Y^{(1)}_{r},Z^{(1)}_{r})dr + \sum_{i=1}^{M}\int_{t}^{v}g_{i}(Y^{(1)}_{r})\eta_{i}(dr,X_{r}) -  \int_{t}^{v}\tilde{Z}^{(1)}_{r} dW_{r},\ t\in[u,v],\\
\tilde{Y}^{(2)}_{t} &= \xi + \int_{t}^{v}f(r,Y^{(2)}_{r},Z^{(2)}_{r})dr + \sum_{i=1}^{M}\int_{t}^{v}g_{i}(Y^{(2)}_{r})\eta_{i}(dr,X_{r}) - \int_{t}^{v} \tilde{Z}^{(2)}_{r} dW_{r},\ t\in[u,v].
\end{aligned}
\end{equation*}
The following proposition yields the  contraction property  of the solution map $\Phi$ on small intervals.
\begin{proposition}\label{prop:contraction map}
Assume \ref{(H1)} holds. Suppose further that $\xi \in\mathcal{F}_{v},$ and that $\left(Y^{(j)},Z^{(j)}\right)\in \mathfrak{B}_{p,k}(R;[u,v])$ for some $R>0$, with $Y^{(j)}_{v} = \xi$ for $j=1,2$.  Then, for any $\rho>0$, there exists a $\theta=\theta(C_1, \Theta, R, \rho, \|\eta\|_{\tau,\lambda},m_{p,k}(X;[0,T]))>0$ such that for all $u\in[(v-\theta)\vee 0,v]$, we have 
\begin{equation}\label{e:map-rho}
\left|\Phi^{\xi}_{[u,v]}\left(Y^{(1)}, Z^{(1)}\right) - \Phi^{\xi}_{[u,v]}\left( Y^{(2)}, Z^{(2)}\right)\right|_{p,k;[u,v]}\le \rho \left|\left(Y^{(1)} - Y^{(2)},Z^{(1)} - Z^{(2)}\right)\right|_{p,k;[u,v]}.
\end{equation}
In particular, the choice of $\theta$  is independent of $\|\xi\|_{L^{\infty}(\Omega)}$.
\end{proposition}

\begin{proof}
Assume $M=1$ without loss of generality. Let $(\tilde{Y}^{(j)},\tilde{Z}^{(j)}) := \Phi^{\xi}_{[u,v]}(Y^{(j)},Z^{(j)}), j=1, 2.$ Denote $\delta\tilde{Y} := \tilde{Y}^{(1)}-\tilde{Y}^{(2)},\ \delta  Y := Y^{(1)}-Y^{(2)}$, and $\delta f_{r} := f(r,Y^{(1)}_{r},Z^{(1)}_{r}) - f(r,Y^{(2)}_{r},Z^{(2)}_{r})$. Similarly, denote $\delta \tilde{Z}:= \tilde{Z}^{(1)}-\tilde{Z}^{(2)},\ \delta Z := Z^{(1)}-Z^{(2)}$, and $\delta g_{r}:= g(Y^{(1)}_{r}) - g(Y^{(2)}_{r}).$ For $t\in[u,v]$, we have
\begin{equation*}
\begin{aligned}
\Big|\delta \tilde{Y}_{t}\Big| 
&\le \mathbb{E}_{t}\left[\int_{t}^{v}|\delta f_{r}|dr + \Big\|\int_{\cdot}^{v} \delta g_{r} \eta(dr,X_{r})\Big\|_{p\text{-var};[t,v]}\right]\\
&\lesssim_{\tau,\lambda,p} C_{1}(v-t)\mathbb{E}_{t}\left[\|\delta Y\|_{\infty;[t,v]}\right] + C_{1}(v-t)^{\frac{1}{2}}\mathbb{E}_{t}\Big[\Big\{\int_{t}^{v}|\delta Z_{r}|^{2}dr \Big\}^{\frac{1}{2}}\Big]\\
&\quad\quad+ \|\eta\|_{\tau,\lambda}|v - t|^{\tau} \mathbb{E}_{t}\left[\left(1+\|X\|_{p\text{-var};[t,v]}\right)\|\delta g\|_{\infty;[t,v]}\right] + \|\eta\|_{\tau,\lambda}|v - t|^{\tau}\mathbb{E}_{t}\left[\|\delta g\|_{p\text{-var};[t,v]}\right]\\
&\lesssim C_{1}(v-t)\mathbb{E}_{t}\left[\|\delta Y\|_{\infty;[t,v]}\right] + C_{1}(v-t)^{\frac{1}{2}}\mathbb{E}_{t}\Big[\Big\{\int_{t}^{v}|\delta Z_{r}|^{2}dr \Big\}^{\frac{1}{2}}\Big]\\
&\quad\quad + C_{1}\|\eta\|_{\tau,\lambda}|v - t|^{\tau} \mathbb{E}_{t}\left[\left(1+\|X\|_{p\text{-var};[t,v]}\right)\|\delta Y\|_{\infty;[t,v]}\right] \\
&\quad\quad + C_{1}\|\eta\|_{\tau,\lambda}|v - t|^{\tau}\Big(\mathbb{E}_{t}\left[\|\delta Y\|_{p\text{-var};[t,v]}\right] + (\|Y^{(1)}\|_{p\text{-var};[t,v]} + \|Y^{(2)}\|_{p\text{-var};[t,v]})\|\delta Y\|_{\infty;[t,v]}\Big)\\
&\lesssim C_{1}(v-u)m_{p,k}(\delta Y;[u,v]) + C_{1}(v-u)^{\frac{1}{2}}\|\delta Z\|_{k\text{-BMO};[u,v]} \\
&\quad\quad + C_{1}\|\eta\|_{\tau,\lambda}|v-u|^{\tau}\big(1+m_{p,k}(X;[u,v]) + R\big)m_{p,k}(\delta Y;[u,v]),
\end{aligned}
\end{equation*}
where the second inequality is derived from Proposition~\ref{prop:bounds}, and the third inequality is a consequence of Lemma~\ref{lem:delta-g}. Therefore, we have
\begin{equation}\label{e:delta-tilde-Y}		
\begin{aligned}
\esssup_{t\in[u,v];\omega\in\Omega}|\delta \tilde{Y}_{t}|&\lesssim_{\tau,\lambda,p} C_{1}(v-u)m_{p,k}(\delta Y;[u,v]) + C_{1}(v-u)^{\frac{1}{2}}\|\delta Z\|_{k\text{-BMO};[u,v]} \\
&\quad\quad + C_{1}\|\eta\|_{\tau,\lambda}|v-u|^{\tau}(1+m_{p,k}(X;[u,v])+R)m_{p,k}(\delta Y;[u,v]).
\end{aligned}
\end{equation}
Noting that
\begin{equation}\label{e:equality delta tilde Z}
\int_{t}^{v}\delta \tilde{Z}_{r}dW_{r} = - \delta\tilde{Y}_{t} + \int_{t}^{v}\delta f_{r} dr + \int_{t}^{v} \delta g_{r}\eta(dr,X_{r}),
\end{equation} 
we can proceed similarly to the estimation of $|\delta \tilde Y_t |$ and obtain
\begin{equation}
\begin{aligned}
&\mathbb{E}_{t}\left[\Big|\int_{t}^{v}|\delta\tilde{Z}_{r}|^{2}dr\Big|^{\frac{k}{2}}\right] \lesssim_{k} \mathbb{E}_{t}\bigg[\sup_{s\in[t,v]}\Big|\int_{s}^{v}\delta\tilde{Z}_{r}dW_{r}\Big|^{k}\bigg]\\ 
&\lesssim_{k} \mathbb{E}_{t}\left[\|\delta\tilde{Y}\|^{k}_{\infty;[t,v]}\right] + \mathbb{E}_{t}\left[\Big|\int_{t}^{v}|\delta f_{r}|dr\Big|^{k}\right] + \mathbb{E}_{t}\left[\Big\|\int_{\cdot}^{v}\delta g_{r}\eta(dr,X_{r})\Big\|^{k}_{p\text{-var};[t,v]}\right]\\  \label{z2}
&\lesssim_{\Theta} C_{1}^{k}(v-u)^{k}m_{p,k}(\delta Y;[u,v])^{k} + C_{1}^{k}(v-u)^{\frac{k}{2}}\|\delta Z\|^{k}_{k\text{-BMO};[u,v]} \\ 
&\quad\quad + C_{1}^{k}\|\eta\|^{k}_{\tau,\lambda}|v-u|^{k\tau}\big(1+m_{p,k}(X;[u,v])+R\big)^{k}m_{p,k}(\delta Y;[u,v])^{k},
\end{aligned}
\end{equation}
where the last step follows from \eqref{e:delta-tilde-Y}. Then \eqref{z2} yields that	
\begin{equation}\label{e:estimate-theorem2-Z}
\begin{aligned}
\|\delta \tilde{Z}\|_{k\text{-BMO};[u,v]}&\lesssim_{\Theta} C_{1}(v-u)m_{p,k}(\delta Y;[u,v]) + C_{1}(v-u)^{\frac{1}{2}}\|\delta Z\|_{k\text{-BMO};[u,v]} \\
&\quad\quad + C_{1}\|\eta\|_{\tau,\lambda}|v-u|^{\tau}\big(1+m_{p,k}(X;[u,v])+R\big)m_{p,k}(\delta Y;[u,v]).
\end{aligned}
\end{equation}			
Now we estimate $\|\delta \tilde{Y}\|_{p\text{-var};[t,v]}$. By the equality~\eqref{e:equality delta tilde Z}, the BDG inequality proved in Lemma~\ref{lem:BDG}, and the estimate~\eqref{z2}, it follows that
\begin{equation*}
\begin{aligned}
&\mathbb{E}_{t}\left[\|\delta \tilde{Y}\|^{k}_{p\text{-var};[t,v]}\right]\\
&\lesssim_{k} \mathbb{E}_{t}\left[\Big|\int_{t}^{v}|\delta f_{r}|dr\Big|^{k}\right]+\mathbb{E}_{t}\left[\Big\|\int_{\cdot}^{v}\delta g_{r}\eta(dr,X_{r})\Big\|^{k}_{p\text{-var};[t,v]}\right] + \mathbb{E}_{t}\left[\Big\|\int_{\cdot}^{v}\delta\tilde{Z}_{r}dW_{r}\Big\|^{k}_{p\text{-var};[t,v]}\right]\\
&\lesssim_{\Theta} C_{1}^{k}(v-u)^{k}m_{p,k}(\delta Y;[u,v])^{k} + C_{1}^{k}(v-u)^{\frac{k}{2}}\|\delta Z\|^{k}_{k\text{-BMO};[u,v]} \\
&\quad\quad + C_{1}^{k}\|\eta\|^{k}_{\tau,\lambda}|v-u|^{k\tau}\big(1+m_{p,k}(X;[u,v])+R\big)^{k}m_{p,k}(\delta Y;[u,v])^{k},
\end{aligned}
\end{equation*}
which implies that 
\begin{equation}\label{e:estimate-theorem2-Y}
\begin{aligned}
m_{p,k}(\delta\tilde{Y};[u,v])&\lesssim_{\Theta} C_{1}(v-u)m_{p,k}(\delta Y;[u,v]) + C_{1}(v-u)^{\frac{1}{2}}\|\delta Z\|_{k\text{-BMO};[u,v]} \\
&\quad\quad + C_{1}\|\eta\|_{\tau,\lambda}|v-u|^{\tau}\big(1+m_{p,k}(X;[u,v])+R\big)m_{p,k}(\delta Y;[u,v]).
\end{aligned}
\end{equation}
Assume $|v-u|\le \theta$. Combining \eqref{e:estimate-theorem2-Z} and \eqref{e:estimate-theorem2-Y}, we have
\begin{equation*}
\begin{aligned}|(\delta\tilde{Y},\delta\tilde{Z})|_{p,k;[u,v]}
&\le  C_{\Theta}C_{1}\left(\theta + \theta^{\frac{1}{2}} + \|\eta\|_{\tau,\lambda}\theta^{\tau}\big(1+m_{p,k}(X;[u,v]) + R\big)\right)|(\delta Y,\delta Z)|_{p,k;[u,v]}\\
&\le C_{\Theta}\tilde{C}(\theta + \theta^{\tau} + \theta^{\frac{1}{2}})|(\delta Y,\delta Z)|_{p,k;[u,v]},
\end{aligned}
\end{equation*}
where $\tilde{C} := C_{1}\vee\Big(C_{1}\|\eta\|_{\tau,\lambda}\big(1 + m_{p,k}(X;[0,T]) + R\big)\Big)$ and $C_{\Theta}>0$ is a constant that depends on $\Theta$. If we choose $\theta $ to be sufficiently small such that $ C_{\Theta}\tilde{C}(\theta + \theta^{\tau} + \theta^{\frac{1}{2}}) \le \rho$, then \eqref{e:map-rho} follows. 	
\end{proof}

We are ready to prove the main result in this subsection.

\begin{theorem}\label{thm:picard}
Assume \ref{(H1)} holds.  Then, BSDE~\eqref{e:ourBSDE} has a unique solution $(Y,Z)\in \mathfrak B_{p,k}(0,T)$, where we recall that the space $\mathfrak B_{p,k}$ is defined in \eqref{e:def of B_p,k}. 
\end{theorem}

\begin{proof}[Proof of Theorem~\ref{thm:picard}]
\makeatletter\def\@currentlabelname{the proof}\makeatother
\label{proof:1}
We first prove the existence. Let $\delta$ and $C$ be the constants in Lemma~\ref{lem:prior estimate}, and let $v = T$. By the proof of Lemma~\ref{lem:prior estimate}, the constant $C$ here is the same as the constant $C$ in Proposition~\ref{prop:ball-invariance}. Fix $\rho\in(0,1),$ and set
\[R := C\left(1 + C\left(1+\|\xi\|_{L^{\infty}(\Omega)}\right) \exp\left\{C\left\lfloor\frac{T}{\delta}\right\rfloor\right\} \right),\]
where $R$ will play the same role in the proof as in Proposition~\ref{prop:ball-invariance}. Let $\theta$ be the same as in Proposition~\ref{prop:contraction map}.  We assume $\theta\le \delta$. 
Denote $T_{m}:= (T-m\theta)\vee 0$. By Proposition~\ref{prop:ball-invariance}, we have 
\begin{equation*}
\left\{\Phi^{\xi}_{[T_{1},T]}\right\}^{i}\left(0,0\right)\in \mathfrak{B}_{p,k}\left(R;[T_{1},T]\right),~ \text{ for } i\ge 1,
\end{equation*}
where $\left\{\Phi^{\xi}_{[T_{1},T]}\right\}^{1}(\cdot) := \Phi^{\xi}_{[T_{1},T]}(\cdot)$ and $\left\{\Phi^{\xi}_{[T_{1},T]}\right\}^{i}(\cdot) := \Phi^{\xi}_{[T_{1},T]}\left(\left\{\Phi^{\xi}_{[T_{1},T]}\right\}^{i-1}(\cdot)\right)$ for $i \ge 2$. By Proposition~\ref{prop:contraction map}, $\left\{\Phi^{\xi}_{[T_{1},T]}\right\}^{i}\left(0,0\right) $ is a Cauchy sequence under the norm $\|\cdot\|_{p,k;[T_1, T]}$. Thus, the limit process on $[T_1, T]$ denoted by $(Y,Z)\big|_{[T_{1},T]}$, is a fixed point of the mapping  $\Phi^{\xi}_{[T_{1},T]}$, i.e.,  $(Y,Z)\big|_{[T_{1},T]}$ is a solution of \eqref{e:ourBSDE} on $[T_{1},T]$. 

We repeat this procedure to obtain the existence of the solution on each subinterval $[T_{m+1},T_m]$ for $m=1, \dots, \left\lfloor \frac{T}{\delta}\right\rfloor$. Assume $(Y,Z)\big|_{[T_{m},T]}\in \mathfrak B_{p,k}(T_m,T)$ solves  \eqref{e:ourBSDE}  on $[T_{m}, T]$, by Lemma~\ref{lem:prior estimate},
\begin{equation*}
\|Y_{T_m}\|_{L^{\infty}(\Omega)}\le \esssup_{t\in [T_m, T]; \omega\in \Omega} |Y_t| \le C\left(1+\|\xi\|_{L^{\infty}(\Omega)}\right) \exp\left\{C\left\lfloor\frac{T}{\delta}\right\rfloor\right\}.
\end{equation*}
Denote $\xi^{(m)}:=Y_{T_m}$. Then, since $R > C(1 + \|\xi^{(m)}\|_{L^{\infty}(\Omega)})$, by Proposition~\ref{prop:ball-invariance} with $\xi$ replaced by $ \xi^{(m)},$  we have 
\begin{equation*}
\left\{\Phi^{\xi^{(m)}}_{[T_{m+1},T_{m}]}\right\}^{i}\left(0,0\right)\in \mathfrak B_{p,k}\left(R;[T_{m+1},T_{m}]\right), ~ \text{ for } i\ge 1.
\end{equation*}
By Proposition~\ref{prop:contraction map} again,  $\left\{\Phi^{\xi^{(m)}}_{[T_{m+1},T_m]}\right\}^{i}\left(0,0\right) $ is a Cauchy sequence under norm $\|\cdot\|_{p,k;[T_{m+1}, T_m]}$, and the limit, denoted by $(Y,Z)\big|_{[T_{m+1},T_m]}$, solves \eqref{e:ourBSDE} on $[T_{m+1},T_m]$. 		

Now we prove the uniqueness of the solution. 	Assume  $Y^{(1)},Y^{(2)}\in \mathfrak B_{p,k}(0,T)$ solve \eqref{e:ourBSDE} on $[0,T]$. 
By Lemma~\ref{lem:prior estimate} $(Y^{(1)},Z^{(1)}),(Y^{(2)},Z^{(2)})\in \mathfrak{B}_{p,k}(R;[T_{1},T])$. Thus, by \eqref{e:map-rho}, we have
\[ \left|\left(Y^{(1)} - Y^{(2)},Z^{(1)} - Z^{(2)}\right)\right|_{p,k;[T_{1},T]} \le \rho \left|\left(Y^{(1)} - Y^{(2)},Z^{(1)} - Z^{(2)}\right)\right|_{p,k;[T_{1},T]}.\]
Hence, $(Y^{(1)},Z^{(1)})$ and $(Y^{(2)},Z^{(2)})$ must coincide on $[T_1, T]$  since $\rho<1$, and similarly, we can prove $(Y^{(1)},Z^{(1)}) = (Y^{(2)},Z^{(2)})$ on $[0,T]$. 				
\end{proof}

\begin{remark}\label{re:eu}
The proof of the existence and uniqueness of the solution on the entire interval $[0,T]$ in Theorem~\ref{thm:picard} is more involved than in the classical situation.  This is mainly because the estimate for $\left\|g(Y)-g(Y')\right\|_{p\text{-}\mathrm{var}}$ in Lemma~\ref{lem:delta-g} involves $\big(\|Y\|_{p\text{-}\mathrm{var}}+\left\|Y'\right\|_{p\text{-}\mathrm{var}}\big)$, given that the function $g$ is nonlinear and, as a consequence, the parameter $\theta$ (the length of the small intervals) in Proposition~\ref{prop:contraction map} depends on $\big(\|Y^{(1)}\|_{p\text{-}\mathrm{var}}\|+\|Y^{(2)}\|_{p\text{-}\mathrm{var}}\big)$. This necessitates invoking the \emph{a priori} estimate in Proposition~\ref{prop:ball-invariance} and Lemma~\ref{lem:prior estimate} in the proof of Theorem~\ref{thm:picard} when we extend the existence of the solution from small intervals to the entire interval $[0,T]$. 
\end{remark}

\subsection{Linear BSDEs and  comparison theorem}\label{sec:Comparison}
In this subsection,  we prove a comparison result for BSDE~\eqref{e:ourBSDE} when $N=1$ (see Theorem~\ref{thm:comparison}) by studying linear BSDEs in the following form:
\begin{equation}\label{e:BSDE-linear}
Y_{t} = \xi + \sum_{i=1}^{M}\int_{t}^{T} \alpha^{i}_{r}Y_{r}\eta_{i}(dr,X_{r}) - \int_{t}^{T}Z_{r}dW_{r},\ t\in[0,T].
\end{equation}

\begin{proposition}\label{prop:exp}
Assume \ref{(H0)} holds and $(\xi,\eta,X)$ satisfies conditions in \ref{(H1)}. Consider the  linear BSDE~\eqref{e:BSDE-linear} where $\alpha^{i}:[0,T]\times\Omega\rightarrow\mathbb{R}^{N\times N}$, $ i=1,2,...,M$, are adapted \emph{bounded} continuous processes such that $m_{p,k}(\alpha;[0,T])<\infty$.  Then \eqref{e:BSDE-linear} has a unique solution $(Y,Z)\in\mathfrak{B}_{p,k}(0,T)$.  Moreover,
we have that $Y$ is bounded, i.e.,
$$
\esssup\limits_{t\in[0,T],\omega\in\Omega}|Y_{t}|\le C<\infty,
$$ 
where $C$ depends on $\|\xi\|_{L^{\infty}(\Omega)}$, $\esssup\limits_{t\in[0,T],\omega\in\Omega}|\alpha|$, $m_{p,k}(\alpha;[0,T])$, $m_{p,k}(X;[0,T])$, $\|\eta\|_{\tau, \lambda}$, and $\Theta.$
\end{proposition}

\begin{remark}\label{rem:linear BSDE & BSDE}
We point out that \eqref{e:BSDE-linear} is not covered by \eqref{e:ourBSDE}, since $\alpha$ is a random process. However, the well-posedness of \eqref{e:BSDE-linear} can be proved in a way similar to that for \eqref{e:ourBSDE} in Section~\ref{sec:picard}. 
\end{remark}

\begin{proof}
Firstly, we will show that the map $(Y,Z)\mapsto  \Phi^\xi(Y, Z):= (\tilde Y, \tilde Z)$ from $\mathfrak B_{p,k}(0,T) $ to itself by
\begin{equation}\label{e:solution-map}
\tilde{Y}_{t} = \xi + \sum_{i=1}^{M}\int_{t}^{T}\alpha^{i}_{r}Y_{r}\eta_{i}(dr,X_{r}) - \int_{t}^{T}\tilde{Z}_{r}dW_{r},\quad t\in[0,T],
\end{equation}		 	 
is well defined. 
Denote $A:=\esssup\limits_{t\in[0,T],\omega\in\Omega}|\alpha_t|$.
Given $(Y,Z)\in \mathfrak B_{p,k}(0,T)$, i.e., $\|(Y,Z)\|_{p,k;[0,T]}<\infty$, we have for $t\in[0,T]$, 
\[\|\alpha^{i} Y\|_{p\text{-var};[t,T]}\lesssim \|\alpha^{i}\|_{\infty;[t,T]}\|Y\|_{p\text{-var};[t,T]} + \|\alpha^{i}\|_{p\text{-var};[t,T]}\|Y\|_{\infty;[t,T]},\] 
and (see \eqref{eq:ess-mpk}) 
\begin{equation*}
\esssup\limits_{s\in[t,T],\omega\in\Omega}|Y_s|\le \|Y_{T}\|_{L^{\infty}(\Omega)} + m_{p,k}(Y;[t,T]).
\end{equation*}
Then, by \eqref{e:p-var1}, we have
\begin{equation}\label{e:linear-BSDE-main estimate}
\begin{aligned}
&\mathbb{E}_{t}\left[\Big\|\int_{\cdot}^{T}\alpha^{i}_{r}Y_{r}\eta_{i}(dr,X_{r})\Big\|^{k}_{p\text{-var};[t,T]}\right]\\
&\lesssim_{\Theta}|T-t|^{k\tau}\|\eta_{i}\|^{k}_{\tau,\lambda}\left(\mathbb{E}_{t}\left[\|\alpha^{i} Y\|^{k}_{p\text{-var};[t,T]} + \left(1+\|X\|^{k\lambda}_{p\text{-var};[t,T]}\right)\|\alpha^{i} Y\|^{k}_{\infty;[t,T]}\right]\right)\\
&\lesssim_k |T-t|^{k\tau}\|\eta\|^{k}_{\tau,\lambda}\left(1 + m_{p,k}(X;[0,T])^{k\lambda}\right)\left(A^{k} + m_{p,k}(\alpha;[0,T])^{k}\right)\\
&\quad\cdot\left(\|Y_{T}\|^{k}_{L^{\infty}(\Omega)} + m_{p,k}(Y;[t,T])^{k}\right).
\end{aligned}
\end{equation} 

As a consequence of \eqref{e:linear-BSDE-main estimate}, we have 
\begin{equation*}
\esssup_{t\in[0,T],\omega\in\Omega}\mathbb{E}_{t}\left[\Big\|\int_{\cdot}^{T}\alpha^{i}_{r}Y_{r}\eta_{i}(dr,X_{r})\Big\|^{k}_{p\text{-var};[t,T]}\right]<\infty.
\end{equation*}
Thus, by the martingale representation theorem utilized in the proof of Lemma~\ref{lem:BSDE-Phi}, there exists a unique 
progressively measurable process, denoted by $\tilde{Z}$, such that $\int_{0}^{T}|\tilde{Z}_{r}|^{2}dr<\infty$ a.s. and 
\begin{equation}\label{e:tildeZ}
\int_0^t \tilde{Z}_{r}dW_{r} =  \mathbb{E}_{t}\left[\xi + \sum_{i=1}^{M}\int_{0}^{T}\alpha^{i}_{r}Y_{r}\eta_{i}(dr,X_{r})\right] - \mathbb{E}\left[\xi + \sum_{i=1}^{M}\int_{0}^{T}\alpha^{i}_{r}Y_{r}\eta_{i}(dr,X_{r})\right].
\end{equation} 
Denote
\begin{equation}\label{e:tildeY}
\tilde{Y}_{t} := \mathbb{E}_{t}\left[\xi + \sum_{i=1}^{M}\int_{t}^{T}\alpha^{i}_{r}Y_{r}\eta_{i}(dr,X_{r})\right],\quad t\in[0,T].
\end{equation}
Then \eqref{e:solution-map}holds. In addition, similarly to
\eqref{e:tilde Y}--\eqref{e:estimate-theorem1-Y}, we have $\|(\tilde{Y},\tilde{Z})\|_{p,k;[0,T]}<\infty$, and hence $(\tilde Y, \tilde Z)\in \mathfrak B_{p,k}(0,T)$. 

Let $\Phi^{\xi}_{[S,T]}$ denote the restriction of the map $\Phi^{\xi}$ to the interval $[S,T]$ for $0<S<T$. We will demonstrate that $\Phi^{\xi}_{[S,T]}$ is a contraction mapping on the space $(\mathfrak{B}_{p,k}(S,T),\|\cdot\|_{p,k;[S,T]})$ when $T-S$ is sufficiently small.  Assume $Y^{(j)}_{T} = \xi$ and $\|(Y^{(j)},Z^{(j)})\|_{p,k}<\infty$ for $j=1,2$. Denote $\delta \tilde{Y}:=\tilde{Y}^{(1)} - \tilde{Y}^{(2)},\ \delta \tilde{Z}:=\tilde{Z}^{(1)} - \tilde{Z}^{(2)}$, and denote $\delta Y $ and $\delta Z$ similarly. Using the same calculation as in the estimate \eqref{e:linear-BSDE-main estimate}, we obtain the following result,
\begin{equation}\label{e:Y1-Y2}
\begin{aligned}
&\mathbb{E}_{t}\left[\Big\|\int_{\cdot}^{T}\alpha^{i}_{r}\delta Y_r\eta_{i}(dr,X_{r})\Big\|^{k}_{p\text{-var};[t,T]}\right]\\
&\lesssim_{\Theta} |T-t|^{k\tau}\|\eta\|^{k}_{\tau,\lambda}\left(1 + m_{p,k}(X;[0,T])^{k}\right)\left(A^{k} + m_{p,k}(\alpha;[0,T])^{k}\right)    m_{p,k}(\delta Y;[t,T])^{k}.
\end{aligned}
\end{equation}
Now we start to estimate $m_{p,k}(\delta\tilde{Y};[S,T]) + \|\delta\tilde{Z}\|_{k\text{-BMO};[S,T]}$. By \eqref{e:solution-map} and Lemma~\ref{lem:BDG}, we have
\begin{equation}\label{e:deltatildeY}
\begin{aligned}
m_{p,k}(\delta \tilde{Y};[S,T])&\le \sum_{i=1}^{M}\esssup_{\omega\in\Omega,t\in[S,T]}\left\{\mathbb{E}_{t}\left[\Big\|\int_{\cdot}^{T}\alpha^{i}_{r}\delta Y_r\eta_{i}(dr,X_{r})\Big\|^{k}_{p\text{-var};[t,T]}\right]\right\}^{\frac{1}{k}}\\ 
&\quad + \esssup_{\omega\in\Omega,t\in[S,T]}\left\{\mathbb{E}_{t}\left[\Big\|\int_{\cdot}^{T}\delta\tilde{Z}_{r} dW_{r}\Big\|^{k}_{p\text{-var};[t,T]}\right]\right\}^{\frac{1}{k}}\\
&\lesssim_{p,k} \sum_{i=1}^{M}\esssup_{\omega\in\Omega,t\in[S,T]}\left\{\mathbb{E}_{t}\left[\Big\|\int_{\cdot}^{T}\alpha^{i}_{r}\delta Y_r\eta_{i}(dr,X_{r})\Big\|^{k}_{p\text{-var};[t,T]}\right]\right\}^{\frac{1}{k}}\\ 
& \quad + \|\delta\tilde{Z}\|_{k\text{-BMO};[S,T]}.
\end{aligned}
\end{equation}
By the BDG inequality, \eqref{e:tildeZ}, and \eqref{e:tildeY}, we have
\begin{equation}\label{e:deltatildeZ}
\begin{aligned}
\|\delta\tilde{Z}\|_{k\text{-BMO};[S,T]}
&\lesssim_k \left\{\esssup_{\omega\in\Omega,t\in[S,T]} \mathbb{E}_{t}\left[\Big\|\int_{\cdot}^{T}\delta\tilde{Z}_r dW_r\Big\|^{k}_{\infty;[t,T]}\right]\right\}^{\frac{1}{k}}\\
&\le \esssup_{\omega\in\Omega,t\in[S,T]}|\delta\tilde{Y}_{t}| + \sum_{i=1}^{M}\left\{\esssup_{\omega\in\Omega,t\in[S,T]} \mathbb{E}_{t}\left[\Big\|\int_{\cdot}^{T}\alpha^{i}_{r}\delta Y_{r}\eta_{i}(dr,X_{r})\Big\|^{k}_{\infty;[t,T]}\right]\right\}^{\frac{1}{k}}\\
&\le 2 \sum_{i=1}^{M}\esssup_{\omega\in\Omega,t\in[S,T]}\left\{\mathbb{E}_{t}\left[\Big\|\int_{\cdot}^{T}\alpha^{i}_{r}\delta Y_{r}\eta_{i}(dr,X_{r})\Big\|^{k}_{p\text{-var};[t,T]}\right]\right\}^{\frac{1}{k}}.
\end{aligned}
\end{equation}
Since $\tilde{Y}^{(1)}_T = \tilde{Y}^{(2)}_T = \xi$, we have
\[\left\|\Phi^{\xi}_{[S,T]}(Y^{(1)},Z^{(1)}) - \Phi^{\xi}_{[S,T]}(Y^{(2)},Z^{(2)})\right\|_{p,k;[S,T]} = m_{p,k}(\delta\tilde{Y};[S,T]) + \|\delta\tilde{Z}\|_{k\text{-BMO};[S,T]}.\] 
In addition, the following estimation is obtained by combining \eqref{e:Y1-Y2}, \eqref{e:deltatildeY}, and \eqref{e:deltatildeZ},
\begin{equation}\label{e:linear-BSDE-contraction}
\begin{aligned}
& m_{p,k}(\delta\tilde{Y};[S,T]) + \|\delta\tilde{Z}\|_{k\text{-BMO};[S,T]}\\
&\lesssim_{p,k} \sum_{i=1}^{M}\esssup_{\omega\in\Omega,t\in[S,T]}\left\{\mathbb{E}_{t}\left[\Big\|\int_{\cdot}^{T}\alpha^{i}_{r}\delta Y_{r}\eta_{i}(dr,X_{r})\Big\|^{k}_{p\text{-var};[t,T]}\right]\right\}^{\frac{1}{k}}\\
&\lesssim_{\Theta}|T-S|^{\tau}\|\eta\|_{\tau,\lambda}\left(1 + m_{p,k}(X;[0,T])\right)\left(A + m_{p,k}(\alpha;[0,T])\right) m_{p,k}(\delta Y;[S,T])\\
&\lesssim_{B} |T-S|^{\tau}\|(Y^{(1)},Z^{(1)}) - (Y^{(2)},Z^{(2)})\|_{p,k;[S,T]},
\end{aligned}
\end{equation}
where $B:= A\vee m_{p,k}(\alpha;[0,T])\vee m_{p,k}(X;[0,T]) \vee \|\eta\|_{\tau,\lambda}.$
Therefore, \eqref{e:linear-BSDE-contraction} yields that $\Phi^\xi_{[S,T]}$ is a contraction mapping on $(\mathfrak{B}_{p,k}(S,T),\|\cdot\|_{p,k;[S,T]})$ if $|T-S|$ is sufficiently small. 
Since the choice of the length of $|T-S|$ is independent of $\|\xi\|_{L^{\infty}(\Omega)}$, the well-posedness of Eq.~\eqref{e:BSDE-linear} on the entire interval $[0,T]$ follows from
the contraction of solution map and a telescopic argument. Moreover, by the same proof as in Lemma~\ref{lem:prior estimate}, it can be proved that there exists a constant $C$ depending on $\|\xi\|_{L^{\infty}(\Omega)}$, $B$, and $\Theta$ such that $\esssup_{t,\omega}|Y_t|\le C.$ This completes the proof.
\end{proof}

The following result establishes the continuity of the solution to  Eq.~\eqref{e:BSDE-linear} in terms of $(\eta, \xi, \alpha)$.

\begin{proposition}\label{prop:continuity of solution map}
Assume the same conditions as in Proposition~\ref{prop:exp}. Consider $\left\{\eta^{n}\right\}^{\infty}_{n = 1}$, $\left\{\xi^{n}\right\}^{\infty}_{n = 1}$ and $\{\alpha^{i,n}\}_{n=1}^\infty$ such that $\|\alpha^{i,n}_{T} - \alpha^{i}_{T}\|_{L^{\infty}(\Omega)} + m_{p,k}(\alpha^{i,n}-\alpha^{i};[0,T])\rightarrow 0,$ $\|\eta^{n}_{i} - \eta_{i}\|_{\tau,\lambda}\rightarrow 0 $ and $\|\xi^{n} - \xi\|_{L^{\infty}(\Omega)}\rightarrow 0$ as $n\rightarrow\infty$ for each $i=1,2,...,M$. Then, we have $$\lim_{n\to \infty}\|(Y^{n},Z^{n}) - (Y,Z)\|_{p,k;[0,T]}=0,$$ 
where $(Y^{n},Z^{n})$ (resp. $(Y,Z)$) is the unique solution of \eqref{e:BSDE-linear} with parameters $(\xi^{n},\alpha^{(n)},\eta^{n})$ (resp. with $(\xi,\alpha,\eta)$), where $\alpha^{(n)}:=(\alpha^{1,n},\alpha^{2,n},...,\alpha^{M,n})^{\top}$.
Moreover, we have
\begin{equation*}
\begin{aligned}
&\|(Y^{n},Z^{n}) - (Y,Z)\|_{p,k;[0,T]}\\
& \lesssim_{\Theta,C} \| \xi^{n} - \xi\|_{L^{\infty}(\Omega)} + \|\eta^{n} - \eta\|_{\tau,\lambda} + \sum_{i=1}^{M}\left[\|\alpha^{i,n}_{T} - \alpha^{i}_{T}\|_{L^{\infty}(\Omega)} + m_{p,k}(\alpha^{i,n} - \alpha^{i};[0,T])\right],
\end{aligned}
\end{equation*}
where $C$ is any constant greater than $T,$ $\sup\limits_{i, n}\|\eta^{n}_{i}\|_{\tau,\lambda},\  \sup\limits_{n}\|\xi^{n}\|_{L^{\infty}(\Omega)},\  m_{p,k}(X;[0,T]),$ $\sup\limits_{i,n}\|\alpha^{i,n}_{T}\|_{L^{\infty}(\Omega)}$, and $\sup\limits_{i,n}m_{p,k}(\alpha^{i,n};[0,T])$.
\end{proposition}

\begin{proof}
Denoting  $\delta \xi^{n} := \xi^{n} - \xi$, $\delta Y^{n} := Y^{n} - Y$, $\delta Z^{n} := Z^{n} - Z$, $\delta \eta^{n}_{i} := \eta^{n}_{i} - \eta_{i}$, $\delta \alpha^{i,n} := \alpha^{i,n} - \alpha^{i}$, we have
\begin{equation*}
\begin{aligned}
\delta Y^{n}_{t} = \delta\xi^{n} + \sum_{i=1}^{M}&\Big[\int_{t}^{T} \alpha^{i,n}_{r}\delta Y^{n}_{r}\eta^{n}_{i}(dr,X_{r}) + \int_{t}^{T}\delta \alpha^{i,n}_{r}Y_r\eta^{n}_{i}(dr,X_{r})\\
& + \int_{t}^{T}\alpha^{i}_{r}Y_{r}\delta \eta^{n}_{i}(dr,X_{r}) \Big] - \int_{t}^{T}\delta Z^{n}_{r} dW_{r}.
\end{aligned}
\end{equation*}
Let $C$ be a constant as stated in the proposition. Then by \eqref{e:linear-BSDE-main estimate} we have
\begin{align}\nonumber
&\mathbb{E}_{t} \left[ \Big\| \int_{\cdot}^{T} \alpha^{i,n}_{r}\delta Y^{n}_{r}\eta^{n}_{i}(dr,X_{r}) \Big\|^{k}_{p\text{-var;}[t,T]} \right] + \mathbb{E}_{t} \left[ \Big\| \int_{\cdot}^{T} \delta\alpha^{i,n}_{r} Y_{r}\eta^{n}_{i}(dr,X_{r}) \Big\|^{k}_{p\text{-var;}[t,T]} \right] \\ \nonumber
&\quad  + \mathbb{E}_{t} \left[ \Big\| \int_{\cdot}^{T}\alpha^{i}_{r}Y_{r}\delta \eta^{n}_{i}(dr,X_{r}) \Big\|^{k}_{p\text{-var;}[t,T]} \right]\\ \nonumber
&\lesssim_{\Theta} |T-t|^{k\tau}\left(1 + m_{p,k}(X;[0,T])^{k\lambda}\right)\bigg(\|\eta^{n}_{i}\|^{k}_{\tau,\lambda}\left(\|\alpha^{i,n}_{T}\|^k_{L^{\infty}(\Omega)} + m_{p,k}(\alpha^{i,n};[0,T])^k\right) \\ 
\label{e:continuity map-estimate 1} &\quad \quad \cdot \left( \|\delta \xi^{n}\|^{k}_{L^{\infty}(\Omega)} + m_{p,k}(\delta Y^{n};[t,T])^{k}\right)\\  
\nonumber 
& \quad\quad  + \|\eta^{n}_{i}\|^{k}_{\tau,\lambda}\left(\|\delta\alpha^{i,n}_{T}\|^k_{L^{\infty}(\Omega)} +  m_{p,k}(\delta\alpha^{i,n};[0,T])^k\right) \left(\|\xi\|^{k}_{L^{\infty}(\Omega)} + m_{p,k}(Y;[t,T])^{k}\right)\\ \nonumber
&\quad\quad     + \|\delta\eta^{n}_{i}\|^{k}_{\tau,\lambda}\left(\|\alpha^{i}_{T}\|^k_{L^{\infty}(\Omega)} + m_{p,k}(\alpha^{i};[0,T])^k\right)\left(\|\xi\|^{k}_{L^{\infty}(\Omega)} + m_{p,k}(Y;[t,T])^{k}\right)\bigg) \\ \nonumber
&\lesssim_{\Theta,C} \|\delta\xi^{n}\|^{k}_{L^{\infty}(\Omega)} + |T-t|^{k\tau}m_{p,k}(\delta Y^{n};[t,T])^{k} + \|\delta\alpha^{i,n}_{T}\|^k_{L^{\infty}(\Omega)}  +  m_{p,k}(\delta\alpha^{i,n};[0,T])^k + \|\delta \eta^{n}_{i}\|^{k}_{\tau,\lambda},
\end{align}
where in the last inequality we use the fact that $\|(Y,Z)\|_{p,k;[0,T]}\le C'$ for some constant $C'>0$ depending on $C$ and $\Theta$, in view of Proposition~\ref{prop:exp}. Then, by similar calculations as in the proof of Proposition~\ref{prop:ball-invariance}, we have
\begin{equation*}
\begin{aligned}
\|(\delta Y^{n},\delta Z^{n})\|_{p,k;[t,T]}\le C'' & \bigg(\|\delta\xi^{n}\|_{L^{\infty}(\Omega)} + |T-t|^{\tau}m_{p,k}(\delta Y^{n};[t,T]) + \sum_{i=1}^{M}\|\delta\alpha^{i,n}_{T}\|_{L^{\infty}(\Omega)}\\
&  + \sum_{i=1}^{M}m_{p,k}(\delta\alpha^{i,n};[0,T]) +  \sum_{i=1}^{M}\|\delta \eta^{n}_{i}\|_{\tau,\lambda}\bigg),
\end{aligned}
\end{equation*}
for some constant $C''>0$.
Denoting $\theta := \left(\frac{1}{2 C''}\right)^{\frac{1}{\tau}}$ and $\tilde C := 2(C''\vee 1)$, we have
\begin{equation*}
\|(\delta Y^{n},\delta Z^{n})\|_{p,k;[T-\theta,T]}  \le \tilde{C} \left(\|\delta\xi^{n}\|_{L^{\infty}(\Omega)} + \sum_{i=1}^{M}\left[\|\delta\alpha^{i,n}_{T}\|_{L^{\infty}(\Omega)} + m_{p,k}(\delta\alpha^{i,n};[0,T])\right] + \sum_{i=1}^{M}\|\delta \eta^{n}_{i}\|_{\tau,\lambda}\right).
\end{equation*}
Let $B^{n} := \sum_{i=1}^{M}\big[\|\delta\alpha^{i,n}_{T}\|_{L^{\infty}(\Omega)} + m_{p,k}(\delta\alpha^{i,n};[0,T])\big] + \sum_{i=1}^{M}\|\delta\eta^{n}_{i}\|_{\tau,\lambda}$. Similarly, for any $m\ge 1$ with $T-m\theta > 0$, we  have 
\begin{equation*}
\begin{aligned}
&\|(\delta Y^{n},\delta Z^{n})\|_{p,k;[0\vee (T-(m+1)\theta), T-m\theta]}\\ 
&\le \tilde{C} \left(\|\delta Y^{n}_{T-m\theta}\|_{L^{\infty}(\Omega)} + B^{n}\right)\\
&\le \tilde{C}\left( \|(\delta Y^{n},\delta Z^{n})\|_{p,k;[T-m\theta, T-(m-1)\theta]} + B^n\right)\\
&\le \tilde{C}\left( \tilde{C} \left(\|\delta Y^{n}_{T-(m-1)\theta}\|_{L^{\infty}(\Omega)} + B^{n}\right) + B^n\right)\\
&\le ...\\
&\le \tilde{C}^{m+1}\|\delta \xi^{n}\|_{L^{\infty}(\Omega)} + \frac{\tilde{C}^{m+2} - \tilde{C}}{\tilde{C} - 1} B^{n}.
\end{aligned}
\end{equation*}
Then, by the following fact,
\[\|(\delta Y^{n},\delta Z^{n})\|_{p,k;[0,T]} \le \sum_{m=0}^{\lfloor\frac{T}{ \theta}\rfloor}\|(\delta Y^{n},\delta Z^{n})\|_{p,k;[0\vee(T-(m+1)\theta), T-m\theta]},\] we obtain that for some constant $C'''>0$,
\[\|(\delta Y^{n},\delta Z^{n})\|_{p,k;[0,T]}\le C'''\left(\|\delta\xi^{n}\|_{L^{\infty}(\Omega)} + B^{n}\right).\]
\end{proof}

Similar to the above proposition, we have
the continuity of the solution to  Eq.~\eqref{e:BSDE-Phi'} in terms of the coefficient function $g$. 

\begin{corollary}\label{cor:continuity solution map under H1}			
Assume \ref{(H0)} holds. Suppose that $(\xi,f,g,\eta,X)$ satisfies \ref{(H1)}, and  $(\xi^n,f^n,g^n,\eta^n),$ $n=1,2,...$ satisfy \ref{(H1)} with a uniform constant $C_1$, 
where $g^{n} = (g^{n}_{i},...,g^{n}_{M})^{\top}$ and $\eta^{n} = (\eta^{n}_{1},...,\eta^{n}_{M})^{\top}$. Further, assume $\|\eta^{n} - \eta\|_{\tau,\lambda}\rightarrow 0, $ $\|\xi^{n} - \xi\|_{L^{\infty}(\Omega)}\rightarrow 0,$ 
$$\esssup_{t\in[0,T],\omega\in\Omega}\left\{\sup_{y\in\mathbb{R}^{N},z\in\mathbb{R}^{N\times d}}|f^{n}(t,y,z) - f(t,y,z)|\right\}\rightarrow 0,$$ and $g^n(0)\rightarrow g(0)$ with $\sup_{y\in\R^N}\left|\nabla\big(g^{n}(y)-g(y)\big)\right|\rightarrow 0$, as $n\to \infty$. Then, we have
\begin{equation}\label{e:Y^n Z^n - Y Z}
\begin{aligned}
&\|(Y^{n},Z^{n}) - (Y,Z)\|_{p,k;[0,T]} \lesssim_{\Theta,C^{*}}\|\xi^n - \xi\|_{L^{\infty}(\Omega)} + \|\eta^n - \eta\|_{\tau,\lambda} + |g^n(0) - g(0)| \\
&\quad + \esssup_{t\in[0,T],\omega\in\Omega}\left\{\sup_{y\in\mathbb{R}^{N},z\in\mathbb{R}^{N\times d}}|f^{n}(t,y,z) - f(t,y,z)|\right\} + \sup_{y\in\R^N}\left|\nabla\big(g^{n}(y)-g(y)\big)\right|,
\end{aligned}
\end{equation}
where $(Y^{n},Z^{n})$ and $(Y,Z)$ are the solutions of
\eqref{e:BSDE-Phi'} with the parameters $(\xi^n, f^n, g^n, \eta^n, X)$ and $(\xi, f, g, \eta, X)$ respectively, and $C^*$ is a constant greater than  $C_1,\ T,\ \sup_{i,n}\|\eta^{n}_{i}\|_{\tau,\lambda},\  \sup_{n}\|\xi^{n}\|_{L^{\infty}(\Omega)}$, and $m_{p,k}(X;[0,T])$. In particular, we have
$$\lim_{n\to\infty}\|(Y^{n},Z^{n}) - (Y,Z)\|_{p,k;[0,T]}=0.$$
\end{corollary}

\begin{proof} 
The proof is similar to that of  Proposition~\ref{prop:continuity of solution map}. Assume $M=1$ for simplicity of notation, and we will only verify the estimation in the form of \eqref{e:continuity map-estimate 1}. Let $\delta \xi^n,\ \delta Y^n,\  \delta Z^n$, and $\delta \eta^n$ be the same notations as in the proof of Proposition~\ref{prop:continuity of solution map}, and denote $\delta g^n_r := g(Y^n_r) - g(Y_r)$. Firstly, by \eqref{e:p-var1} and Lemma~\ref{lem:delta-g}, we have
\begin{equation*}
\begin{aligned}
&\mathbb{E}_{t}\left[\Big\| \int_{\cdot}^{T} \delta g^n_r\eta^{n}(dr,X_{r})\Big\|^{k}_{p\text{-var;}[t,T]}\right] + \mathbb{E}_{t}\left[\Big\| \int_{\cdot}^{T}g(Y_r)\delta \eta^{n}(dr,X_{r})\Big\|^{k}_{p\text{-var;}[t,T]}\right]\\
&\lesssim_{\Theta} |T-t|^{k\tau}\|\eta^{n}\|^{k}_{\tau,\lambda}\left(1 + m_{p,k}(X;[0,T])^{k\lambda}\right)\left(\|\delta \xi^n\|^k_{L^{\infty}(\Omega)} + m_{p,k}(\delta g^n;[t,T])^k\right)\\
&\qquad + |T-t|^{k\tau}\|\delta\eta^{n}\|^{k}_{\tau,\lambda}\left(1 + m_{p,k}(X;[0,T])^{k\lambda}\right) \left(\|\xi\|^{k}_{L^{\infty}(\Omega)} + m_{p,k}(g(Y_{\cdot});[t,T])^{k}\right)\\
&\lesssim_{k,\Theta,C} \|\delta\xi^{n}\|^{k}_{L^{\infty}(\Omega)} + |T-t|^{k\tau}m_{p,k}(\delta Y^{n};[t,T])^{k} + \|\delta \eta^{n}\|^{k}_{\tau,\lambda},
\end{aligned}
\end{equation*}
where $C$ is a positive constant greater than $C_1,\ T,\ \sup_{n}\|\eta^{n}\|_{\tau,\lambda},\  \sup_{n}\|\xi^{n}\|_{L^{\infty}(\Omega)}$ and $m_{p,k}(X;[0,T])$. 
Note that in the last inequality, we use the fact that $\|(Y,Z)\|_{p,k;[0,T]}\vee\|(Y^n,Z^n)\|_{p,k;[0,T]}\le C'$ for some constant $C'>0$ depending on $C$ and $\Theta$ in view of \eqref{e:bound1}.
Secondly, by \eqref{e:p-var1} again, we have
\begin{equation*}
\begin{aligned}
&\mathbb{E}_t\left[\Big\|\int_{\cdot}^{T}\left(g^{n}(Y^{n}_{r}) - g(Y^{n}_{r})\right)\eta^{n}(dr,X_r)\Big\|_{p\text{-var};[t,T]}^{k}\right]\\
&\lesssim_{\Theta} |T-t|^{k\tau}\|\eta^{n}\|^{k}_{\tau,\lambda}\left(1 + m_{p,k}(X;[0,T])^{k\lambda}\right)\Big(\|g(Y^n_{T}) - g^n(Y^n_{T})\|^k_{L^{\infty}(\Omega)}  \\
&\qquad + m_{p,k}(g^n(Y^n_{r}) - g(Y^n_{r});[t,T])^k\Big)\\
&\lesssim_{k} |T-t|^{k\tau}\|\eta^{n}\|^{k}_{\tau,\lambda}\left(1 + m_{p,k}(X;[0,T])^{k\lambda}\right)\Big(\sup_{y\in\R^N}|\nabla g^n(y) - \nabla g(y)|^{k} \|Y^{n}_{T}\|^k_{L^{\infty}(\Omega)}\\
&\qquad + |g^{n}(0) - g(0)|^k + \sup_{y\in\R^N}|\nabla g^n(y) - \nabla g(y)|^k m_{p,k}(Y^{n};[t,T])^k\Big)\\
&\lesssim_{C,\Theta,C'} |T-t|^{k\tau}\left(\sup_{y\in\R^N}|\nabla g^n(y) - \nabla g(y)|^{k} + |g^{n}(0) - g(0)|^k\right),
\end{aligned}
\end{equation*}
which converges to $0$ since $g^n(0)\rightarrow g(0)$ and $\nabla g^n(y) \rightarrow \nabla g(y)$ uniformly in $y\in\mathbb{R}^N$ as $n\rightarrow \infty$. 
Finally, we have
\begin{equation*}
\begin{aligned}
&\mathbb{E}_t\left[\Big|\int_{t}^{T}|f^{n}(r,Y^n_r,Z^n_r) - f(r,Y_r,Z_r)|dr\Big|^{k}\right]\\
&\lesssim_{k,C_1}  T^k\esssup_{t\in[0,T],\omega\in\Omega}\Big\{\sup_{y\in\mathbb{R}^{N},z\in\mathbb{R}^{N\times d}}|f^{n}(t,y,z) - f(t,y,z)|^k\Big\} \\
&\qquad  + (|T-t|^{k} + |T-t|^{\frac{k}{2}})\|(\delta Y^n,\delta Z^n)\|^{k}_{p,k;[t,T]}.
\end{aligned}
\end{equation*}
The rest of the proof is the same as the proof of Proposition~\ref{prop:continuity of solution map} and is therefore omitted.
\end{proof}

Now we provide the Feynman-Kac formula for the solution  $Y$  to the linear BSDE~\eqref{e:BSDE-linear} 
via the following Young differential equation, 
\begin{equation}\label{e:rough ODE}
A^{t}_{s} = I_{N} + \sum_{i=1}^{M}\int_{t}^{s}(\alpha^{i}_r)^{\top} A^{t}_r \eta_{i}(dr,X_r),\quad s\in[t,T],
\end{equation}
where $I_{N}$ is the unit $N\times N$ matrix, and $\alpha$ satisfies the same conditions in Proposition~\ref{prop:exp}. Fixing an $\omega\in\Omega,$ \cite[Proposition~7]{lejay2010controlled} shows that Eq.~\eqref{e:rough ODE} admits a unique solution in $C^{\frac{1}{\tau}\text{-var}}([0,T];\R^{N\times N}).$

\begin{proposition}[Feynman-Kac formula]\label{prop:exp'}
Assume the same conditions as in Proposition~\ref{prop:exp}. Let $(Y,Z)$ be the unique solution of Eq.~\eqref{e:BSDE-linear}. Then, we have 
\begin{equation}\label{e:Y-FK}
Y_{t} = \mathbb{E}_{t}\left[(A^{t}_{T})^{\top}\xi\right] \text{  a.s. for all }  t\in[0,T],
\end{equation} 
where $  A^{t}_{\cdot}$ is the unique solution of Eq.~\eqref{e:rough ODE}.

\end{proposition}
\begin{proof}
For simplicity of notation, we assume $d=1$ and $M=1$ without loss of generality.
Applying the product rule (Corollary~\ref{cor:product rule}) to $(A^{t}_s)^{\top}Y_{s}$, we have
\[Y_t =(A_T^t)^\top \xi -\int_t^T (A_r^t)^\top Z_r dW_r.\]
Thus, to prove \eqref{e:Y-FK}, it suffices to show that $\int_t^\cdot (A_r^t)^\top Z_r dW_r$ is a martingale, which will be proven by an approximation argument based on Proposition~\ref{prop:continuity of solution map} as follows.

Firstly, we will show that there exists a constant $C>0$ such that $\|\mathbb{E}_{t}\left[|A^{t}_{T}|^2\right]\|_{L^{\infty}(\Omega)}\le C$ for all $t\in[0,T]$.
By Lemma~\ref{lem:smooth approximation of eta}, there exists $\left\{\eta^{n}(t,x)\right\}_{n\ge 1}$  such that for each $n$, $\eta^{n}(t,x)$ is smooth in time, $\partial_{t}\eta^{n}(t,x)$ is continuous in $(t,x)$, $\partial_{t}\eta^{n}(t,x)$ is bounded on $[0,T]\times\R^d$, and that $\lim_{n\to\infty}\|\eta^{n} - \eta\|_{\tau',\lambda}=0$, where $\tau'<\tau$. Assume that $(\tau',\lambda,p)$ satisfies Assumption~\ref{(H0)}. Let $A^{t,n}_{\cdot}$ be the solution of \eqref{e:rough ODE} with $\eta=\eta^{n}$ and denote $C^{t,n}_{s} := A^{t,n}_{s} (A^{t,n}_{s})^{\top}$. 
By  Corollary~\ref{cor:product rule} we have 
\begin{equation}\label{e:C tn}
C^{t,n}_{s} = I_{N} + \int_{t}^{s} C^{t,n}_r \alpha_r \eta^n(dr,X_r) + \int_{t}^{s}(\alpha_r)^{\top}C^{t,n}_r\eta^{n}(dr,X_r),~s\in[t,T].
\end{equation}
For $i=1,2,...,N$, denoting the coordinate system of $\mathbb{R}^{N}$ by $e_i = (\overbrace{0,\cdots,0,1}^{i},0,\cdots,0)^{\top}$, it follows from \eqref{e:C tn} that 
\[C^{t,n}_{s} e_i = e_i + \int_{t}^{s} C^{t,n}_r \alpha_r e_i \eta^n(dr,X_r) + \int_{t}^{s}(\alpha_r)^{\top}C^{t,n}_re_i\eta^{n}(dr,X_r).\]
To estimate $\E_{t}[|C^{t,n}_{T}|]$, we define a vector-valued process $\tilde{C}^{t,n}$ and estimate its dual BSDE. Denote
\begin{equation*}
\tilde{C}^{t,n}_{s} := \begin{pmatrix}
C^{t,n}_s e_{1}\\
C^{t,n}_s e_{2}\\
\cdots\\
C^{t,n}_s e_{N}
\end{pmatrix}_{N^{2}\times 1},\ 
\tilde{I}_{N^{2}\times 1} := \begin{pmatrix}
e_{1}\\
e_{2}\\
\cdots\\
e_{N}
\end{pmatrix}_{N^{2}\times 1}.
\end{equation*}
Then, $\tilde{C}^{t,n}$ is the unique solution of the following linear ordinary differential equation, 
\begin{equation*}
\tilde{C}^{t,n}_{s} = \tilde{I}_{N^{2}\times 1} + \int_{t}^{s} \tilde{\alpha}_r \tilde{C}^{t,n}_r  \partial_{r}\eta^n(r,X_r)dr,\quad s\in[t,T].
\end{equation*}
Here, $\tilde{\alpha}_{r}\in \mathbb{R}^{N^{2}\times N^{2}}$ is the process such that
each entry of $\tilde{\alpha}_r$ is a linear combination of the elements of $\alpha_r,$ and
\begin{equation*}
(\underbrace{\overbrace{ 0,\ 0,\ \cdots,\ 0}^{(j-1)N},\ e^{\top}_{i},}_{jN}\ 0,\ \cdots,\ 0
)_{1\times N^{2}}\cdot\tilde{\alpha}_r\tilde{C}^{t,n}_{r} = \left(C^{t,n}_{r}\alpha_r + (\alpha_r)^{\top}C^{t,n}_{r}\right)_{ij},\quad i,j=1,2,...,N.
\end{equation*}
Let $(B^{n},D^{n})\in\mathfrak{B}_{p,k}(0,T)$ be the unique solution of the following linear BSDE,
\begin{equation*}
B^{n}_{s} = \tilde{I}_{N^{2}\times 1} + \int_{s}^{T}(\tilde{\alpha}_r)^{\top} B^{n}_{r} \eta^{n}(dr,X_r) - \int_{s}^{T} D^{n}_r d W_r.
\end{equation*}
Applying the product rule to $(\tilde{C}^{t,n}_{s})^{\top}B^{n}_{s}$,  we have 
\begin{equation*}
(\tilde{I}_{N^{2}\times 1})^{\top}B^{n}_{t} = (\tilde{C}^{t,n}_{T})^{\top}\tilde{I}_{N^{2}\times 1} - \int_{t}^{T} (\tilde{C}^{t,n}_{r})^{\top} D^{n}_r  d W_r.
\end{equation*}
Noting the boundedness of $(s,\omega)\mapsto\tilde{C}^{t,n}_{s}$ by the boundedness of $\partial_{t}\eta^{n}(t,x)$ and $\tilde{\alpha}_{t}$, the It\^o integral on the right-hand side of the above equation is a martingale, and hence we have
\begin{equation}\label{e:E[IB]=E[CI]}
\mathbb{E}_{t}\left[(\tilde{I}_{N^{2}\times 1})^{\top}B^{n}_{t}\right] = \mathbb{E}_{t}\left[(\tilde{C}^{t,n}_{T})^{\top}\tilde{I}_{N^{2}\times 1}\right] = \mathbb{E}_t\left[\text{tr}\left\{C^{t,n}_{T}\right\}\right].
\end{equation}
Noting that $C^{t,n}_{T}$ is a positive  semi-definite matrix, we have $\text{tr}\{C^{t,n}_{T}\}\gtrsim |C^{t,n}_{T}|$ due to the equivalence of norms in finite dimensional vector spaces. Thus, since 
$\sup_{n}\esssup_{(s,\omega)}|B^{n}_{s}|<\infty$ by Proposition~\ref{prop:exp}, the equality~\eqref{e:E[IB]=E[CI]} implies that for some $C'>0$,
\[\sup_{n\ge 1}\sup_{t\in[0,T]}\big\|\mathbb{E}_{t}\big[|C^{t,n}_{T}|\big]\big\|_{L^{\infty}(\Omega)}\le C'.\]
Furthermore, let $C^{t}_{s}:=A^{t}_{s}(A^{t}_{s})^{\top}$. Then $C^{t}_{s}$ satisfies \eqref{e:C tn} with $\eta^n$ replaced by $\eta$,  and $C^{t,n}_{T}$ converges to $C^{t}_{T}$ a.s. by Proposition~\ref{prop:continuity of solution map}, since $\lim_{n}\|\eta^{n} - \eta\|_{\tau',\lambda}= 0$. Hence, by Fatou's lemma, we have 
\begin{equation*}
\sup_{t\in[0,T]}\|\mathbb{E}_{t}\left[|C^{t}_{T}|\right]\|_{L^{\infty}(\Omega)}<\infty,
\end{equation*}
and thus
\begin{equation}\label{e:A^2k}
\sup_{t\in[0,T]}\|\mathbb{E}_{t}\left[|A^{t}_{T}|^{2}\right]\|_{L^{\infty}(\Omega)}<\infty.
\end{equation}

Secondly, for the inverse of $A^{t}_{s}$, noting that $(A^{t}_{s})^{-1}$ satisfies that 
\[(A^{t}_{s})^{-1} = I_{N} - \int_{t}^{s}(A^{t}_r)^{-1}\alpha_r\eta(dr,X_r),\quad s\in[t,T],\]
we have  
\[(C^{t}_{s})^{-1} = ((A^{t}_s)^{\top})^{-1}(A^{t}_{s})^{-1} = I_{N} - \int_{t}^{s}(\alpha_r)^{\top}(C^{t}_{r})^{-1} \eta(dr,X_r) - \int_{t}^{s}(C^{t}_{r})^{-1}\alpha_r \eta(dr,X_r). \]
Then, by the same argument leading to \eqref{e:A^2k}, it follows that
\begin{equation}\label{e:A^2k'}
\sup_{t\in[0,T]}\left\|\mathbb{E}_{t}\left[|(A^{t}_{T})^{-1}|^{2}\right]\right\|_{L^{\infty}(\Omega)}<\infty.
\end{equation}

Finally, we will prove that $\int_{t}^{\cdot}(A^{t}_{r})^{\top} Z_{r} dW_r$ is a martingale. Indeed, noting that $A^{r}_{T}A^{t}_{r} = A^{t}_{T}$ by the uniqueness of the solution to Eq.~\eqref{e:rough ODE}, we have \begin{equation}\label{e:martingale}
\begin{aligned}
\mathbb{E}\left[\Big|\int_{t}^{T}|(A^{t}_{r})^{\top}Z_{r}|^{2}dr\right|^{\frac{1}{2}}\Big] &\le\mathbb{E}\left[\Big|\int_{t}^{T}\left|(A^{r}_{T})^{-1}A^{t}_{T}\right|^{2}|Z_{r}|^{2}dr\Big|^{\frac{1}{2}}\right]\\
&\le \mathbb{E}\left[|A_T^t|^2+ \int_{t}^{T}\mathbb{E}_{r}\left[|(A_T^r)^{-1}|^2\right] |Z_{r}|^{2}dr\right],
\end{aligned}
\end{equation}
which is finite due to \eqref{e:A^2k} and \eqref{e:A^2k'}, and the fact that $\|Z\|_{k\text{-BMO};[0,T]}<\infty$ (hence $\|Z\|_{2\text{-BMO};[0,T]}<\infty$ by \cite[Corollary~2.1]{kazamaki1994continuous}). Then $\int_{t}^{\cdot}(A^{t}_{r})^{\top}Z_{r}dW_{r}$ is a  martingale and the proof is concluded. 
\end{proof}

When $N=1,$ by Proposition~\ref{prop:continuity of solution map} and the ODE theory, Eq.~\eqref{e:rough ODE} has an explicit solution:
\[A_s^{t} = \exp\left\{\sum_{i=1}^{M}\int_{t}^{s}\alpha^{i}_{r}\eta_{i}(dr,X_{r})\right\} \text{ a.s. for }  s\in[t,T].
\]
We have the following corollary of Proposition~\ref{prop:exp'}. 

\begin{corollary}\label{cor:exp'}
Assume the same conditions as in Proposition~\ref{prop:exp}, but with $N=1$. Let $(Y,Z)$ be the solution of \eqref{e:BSDE-linear}. Then, we have 
\begin{equation*}
Y_{t} = \mathbb{E}_{t}\left[\xi\exp\left\{\sum_{i=1}^{M}\int_{t}^{T}\alpha^{i}_{r}\eta_{i}(dr,X_{r})\right\}\right] \text{ a.s. for all }  t\in[0,T],
\end{equation*} 
and for some finite constant $C$, 
\begin{equation*}    \esssup_{t\in[0,T],\omega\in\Omega}\mathbb{E}_t\left[\xi\exp\left\{\sum_{i=1}^{M}\int_{t}^{T}\alpha^{i}_{r}\eta_{i}(dr,X_{r})\right\}\right]\le C.
\end{equation*}
\end{corollary}

A direct extension of Proposition~\ref{prop:exp} and Proposition~\ref{prop:exp'} is the following.
\begin{corollary}\label{cor:exp''}
Assume $\zeta:[0,T]\rightarrow \mathbb{R}^{N\times N}$ and $\gamma:[0,T]\rightarrow \mathscr{L}(\mathbb{R}^{N\times d};\mathbb{R}^{N})$ are two bounded progressively measurable processes, where $\mathscr{L}(E;V)$ denotes the set of linear mappings  from $E$ to $V$. Then the following linear BSDE has a unique solution $(Y,Z)\in\mathfrak{B}_{p,k}(0,T)$,
\begin{equation*}
Y_{t} = \xi + \sum_{i=1}^{M}\int_{t}^{T}\alpha^{i}_{r}Y_{r}\eta_{i}(dr,X_{r}) + \int_{t}^{T} [\zeta_{r} Y_{r} + \gamma_{r} (Z_{r})] dr - \int_{t}^{T}Z_{r}dW_{r},\quad t\in[0,T].
\end{equation*}
In addition, when $N=1$, we have for $t\in[0,T]$,
\begin{equation*}
Y_{t} = \mathbb{E}_{t}\left[\xi\exp\left\{\sum_{i=1}^{M}\int_{t}^{T}\alpha^{i}_{r}\eta_{i}(dr,X_{r}) + \int_{t}^{T}[\zeta_r - \frac{1}{2}|\gamma_r|^{2}]dr + \int_{t}^{T}\gamma_{r}dW_r\right\}\right]\ \text{a.s.}
\end{equation*}
\end{corollary}

Now we are ready to prove Theorem~\ref{thm:comparison}. Note that one can also prove it by Corollary~\ref{cor:continuity solution map under H1} and the classical comparison theorem for BSDEs. Nevertheless, we prove it directly by a duality argument, of which the idea leads to necessary estimations for solving Eq.~\eqref{e:ourBSDE} with unbounded driver $\eta_t$ in~\cite{BSDEYoung-II}. Moreover, the duality argument also yields the strict comparison theorem (Corollary~\ref{cor:strict compar}).

When $N=1$, we have the following comparison theorem. 
\begin{theorem}[comparison theorem]\label{thm:comparison}
Assume \ref{(H1)} holds with $N=1$. Suppose $\xi'\in\mathcal{F}_{T}$ is bounded and $f'$ satisfies \ref{(H1)}, with $\xi\ge \xi'$ a.s., and $f(\cdot,y,z)\ge f'(\cdot,y,z)$ $dt\otimes d\mathbb{P}$-a.e. for every $(y,z)\in \mathbb{R}^{1+d}$. Let $(Y,Z),(Y',Z')\in\mathfrak{B}_{p,k}(0,T)$ be the unique solutions to the following two equations, respectively:
\begin{equation*}
Y_{t} = \xi + \int_{t}^{T}f(r,Y_{r},Z_{r})dr + \sum_{i=1}^{M}\int_{t}^{T}g_{i}(Y_{r})\eta_{i}(dr,X_{r}) - \int_{t}^{T}Z_{r}dW_{r},\ t\in[0,T];
\end{equation*}
\begin{equation*}
Y'_{t} = \xi' + \int_{t}^{T}f'(r,Y'_{r},Z'_{r})dr + \sum_{i=1}^{M}\int_{t}^{T}g_{i}(Y'_{r})\eta_{i}(dr,X_{r}) - \int_{t}^{T}Z'_{r}dW_{r},\ t\in[0,T].
\end{equation*}
Then we have $Y_{t}\ge Y'_{t}$ a.s. for all $t\in[0,T]$.
\end{theorem}

\begin{proof}[Proof of Theorem~\ref{thm:comparison}]
\makeatletter\def\@currentlabelname{the proof}\makeatother\label{proof:2}
Set $M=1$ without loss of generality. Assume $Z_t = (Z^{(1)}_t,Z^{(2)}_t,...,Z^{(d)}_t)$ and $Z'_{t} = (Z'^{(1)}_t,Z'^{(2)}_t,...,Z'^{(d)}_t)$. 
For $i=1,2,...,d$, let
\[\delta f_{t}(y)^{(i)}:= f(t,y,\overbrace{Z'^{(1)}_t,\cdots,Z'^{(i-1)}_t,}^{i-1}Z^{(i)}_t,\cdots, Z^{(d)}_t) - f(t,y,\overbrace{Z'^{(1)}_t,\cdots,Z'^{(i)}_t,}^{i}Z^{(i+1)}_t,\cdots, Z^{(d)}_t).\]
Denote
\[\alpha_{t} := \frac{g(Y_{t})-g(Y'_{t})}{Y_{t}-Y'_{t}}\mathbf 1_{\{Y_t\neq Y'_t\}},\quad b_{t}: = \frac{f(t,Y_t,Z_t) - f(t,Y'_t,Z_t)}{Y_{t} - Y'_{t}} \mathbf 1_{\{Y_t\neq Y'_t\}},\] 
and
\begin{equation*}
c_{t} := \begin{pmatrix}
\frac{\delta f_{t}(Y'_t)^{(1)}}{Z^{(1)}_{t} - Z'^{(1)}_t}\mathbf 1_{\{Z^{(1)}_t\neq Z'^{(1)}_t\}},\ \frac{\delta f_{t}(Y'_t)^{(2)}}{Z^{(2)}_{t} - Z'^{(2)}_t}\mathbf 1_{\{Z^{(2)}_t\neq Z'^{(2)}_t\}},\ \cdots,\frac{\delta f_{t}(Y'_t)^{(d)}}{Z^{(d)}_{t} - Z'^{(d)}_t}\mathbf 1_{\{Z^{(d)}_t\neq Z'^{(d)}_t\}}
\end{pmatrix}\in \mathbb{R}^{1\times d}.
\end{equation*}
Note that by the assumptions on $f$ and $g$,  $\alpha_t$, $b_t$, and $c_t$ are bounded and progressively measurable.

We first prove $m_{p,k}(\alpha;[0,T])<\infty$. By Assumption~\ref{(H1)}, $\nabla g$ is a continuous function and hence 
\begin{equation*}
\alpha_{r} = \int_{0}^{1}\nabla g\left(Y'_{r}+\theta(Y_{r}-Y'_{r})\right)d\theta, ~~r\in[0,T].
\end{equation*}
Then, Minkowski's inequality yields that, for $t\in [0,T]$, 
\begin{equation}\label{e:estimate-alpha-p-var}
\begin{aligned}
\|\alpha\|_{p\text{-var};[t,T]}&\le \int_{0}^{1}\left\|\nabla g\left(Y'_{\cdot}+\theta(Y_{\cdot}-Y'_{\cdot})\right)\right\|_{p\text{-var};[t,T]}d\theta\\
&\le C_{1}\left(\|Y\|_{p\text{-var};[t,T]}+\|Y'\|_{p\text{-var};[t,T]}\right).
\end{aligned}
\end{equation}
Thus,  noting that $(Y,Z),(Y',Z')\in \mathfrak B_{p,k}(0,T)$ by Theorem~\ref{thm:picard},  we have 
\begin{equation}\label{e:mpk-finite}
m_{p,k}(\alpha;[0,T])<\infty.
\end{equation} 			

Now, we show that $Y_{t} - Y'_{t}\ge 0$ a.s. Note that $(Y-Y',Z-Z')\in \mathfrak{B}_{p,k}(0,T)$ is a solution of the following linear BSDE, with a non-decreasing drift $V_{t} := \int_{0}^{t} f(r,Y'_r,Z'_r) - f'(r,Y'_r,Z'_r) dr$,
\begin{equation*}
\begin{aligned}
(Y-Y')_{t} = (\xi - \xi') &+ \int_{t}^{T}\alpha_r (Y-Y')_{r}\eta(dr,X_{r}) + \int_{t}^{T}b_{r}(Y-Y')_{r}dr \\
& + \int_{t}^{T}c_{r}(Z - Z')^{\top}_{r}dr + V_{T} - V_{t} - \int_{t}^{T} (Z - Z')_r dW_r.
\end{aligned}
\end{equation*}
Denote 
\[A^{t}_s := \exp\left\{\int_{t}^{s}\alpha_{r}\eta(dr,X_r)\right\} \text{ for } 0\le t\le s\le T.\]
In view of the finiteness shown in \eqref{e:mpk-finite}, by \eqref{e:A^2k} and \eqref{e:A^2k'} we have that $\mathbb{E}_{t}[|A^{t}_{T}|^{2}]$ and $\mathbb{E}_{t}[|(A^{t}_{T})^{-1}|^{2}]$ are bounded processes. In addition, by applying Corollary~\ref{cor:product rule} to $(Y_s - Y'_s)A^{t}_{s}$, we have 
\begin{equation*}
\begin{aligned}
Y_t - Y'_t = (Y_s - Y'_s)A^{t}_{s} &+ \int_{t}^{s}\left[b_r(Y_r - Y'_r)A^{t}_{r} + c_{r} (Z_r - Z'_r)^{\top}A^{t}_r\right]dr \\
& + \int_{t}^{s}A^{t}_r dV_r - \int_{t}^{s} A^{t}_r(Z_r - Z'_r) dW_r.
\end{aligned}
\end{equation*}

Denote 
\[B^{t}_{s}:=\exp\left\{\int_{t}^{s}b_{r} - \frac{1}{2}|c_{r}|^{2}dr + \int_{t}^{s}c_{r}dW_r\right\} \text{ for } 0\le t\le s\le T.\] Thanks to the boundedness of $b_{t}$ and $c_t$, we have $\mathbb{E}\left[|B^{t}_{T}|^{4}\right]\vee\esssup\limits_{t\in[0,T],\omega\in\Omega}\mathbb{E}_{t}\left[|(B^{t}_{T})^{-1}|^{2}\right]<\infty.$ 
Therefore, by the fact that $\|Z-Z'\|_{k\text{-BMO};[0,T]}<\infty$ and \cite[Corollary~2.1]{kazamaki1994continuous}, we have
\begin{equation}\label{e:B(Z-Z')}
\begin{aligned}
\mathbb{E}\left[\int_{t}^{T}|B^{t}_{r}|^{2}|Z_{r} - Z'_{r}|^{2}dr\right] 
&= \mathbb{E}\left[\int_{t}^{T}\left|(B^{r}_{T})^{-1}B^{t}_{T}\right|^{2}|Z_{r} - Z'_{r}|^{2}dr\right]\\
&\le \mathbb{E}\left[|B_T^t|^4+ \Big|\int_{t}^{T}\mathbb{E}_{r}\left[|(B_T^r)^{-1}|^2\right] |Z_{r} - Z'_{r}|^{2}dr\Big|^{2}\right]<\infty.
\end{aligned}
\end{equation}
Applying Corollary~\ref{cor:product rule} again to $[(Y_s - Y'_s)A^{t}_{s}]B^{t}_s$, we obtain
\begin{equation*}
Y_t - Y'_t = (\xi - \xi')A^{t}_{T}B^{t}_{T} + \int_{t}^{T}A^{t}_r B^{t}_r dV_r - \int_{t}^{T} A^{t}_rB^{t}_r (Z_r - Z'_r) dW_r  - \int_{t}^{T} (Y_r - Y'_r) A^{t}_{r}B^{t}_r c_r dW_r.
\end{equation*}
Then, in view of \eqref{e:B(Z-Z')}, by the same calculation as in \eqref{e:martingale}, the last two It\^o integrals on the right-hand side are martingales due to the finiteness of 
\[\esssup\limits_{t\in[0,T],\omega\in\Omega}\mathbb{E}_{t}\left[|A^{t}_{T}|^{2}\right]\vee\mathbb{E}_{t}\left[|(A^{t}_{T})^{-1}|^{2}\right],\]
the boundedness of $|Y_t-Y'_t|$ and $|c_t|,$ and the fact that $\E[\int_{t}^{T}|B^{t}_{r}|^{2}dr]<\infty$.  Thus, we have 
\begin{equation}\label{e:strict comparison}
Y_t-Y'_t = \mathbb{E}_{t}\left[(\xi - \xi')A^{t}_{T}B^{t}_{T} + \int_{t}^{T}A^{t}_r B^{t}_r dV_r\right],
\end{equation}
and hence $Y_{t} \ge  Y'_{t}$  a.s. for all $t\in[0,T]$.
\end{proof}

By the equality~\eqref{e:strict comparison}, we have the following strict comparison result.

\begin{corollary}[strict comparison theorem]\label{cor:strict compar}
Assume the same conditions in Theorem~\ref{thm:comparison}. Assume further that $\xi > \xi'$ a.s. Then we have $Y_{t}>Y'_{t}$ a.s.
\end{corollary}

\section{BSDEs with random terminal times and  regularity of solutions}\label{sec:continuity map}

In Section~\ref{sec:picard}, we have studied BSDEs under \ref{(H1)} which assumes $\|\eta\|_{\tau,\lambda}<\infty$. For the more general case $\|\eta\|_{\tau,\lambda;\beta}<\infty$ (e.g., the trajectories of a fractional Brownian sheet), we will apply a localization argument to study the BSDEs in Part II \cite{BSDEYoung-II}. For this purpose, in this section we will establish the well-posedness of BSDE \eqref{e:ourBSDE} with  a random terminal time, and then study the regularity property for the solution $Y$, in particular estimating $\mathbb{E}[\|Y\|^{q}_{p\text{-var};[S_{1},S_{0}]}]$ with $S_1$ and $S_0$ being stopping times. 

Consider BSDE~\eqref{e:ourBSDE} with a random terminal time:
\begin{equation}\label{e:BSDE'}
\begin{cases}
dY_{t} = - f(t,Y_{t},Z_{t})dt  - \sum\limits_{i=1}^{M}g_{i}(Y_{t})\eta_{i}(dt,X_{t}) + \ Z_{t}dW_{t},\quad  t\in[0,S_{0}],\\
 Y_{S_{0}}=\xi,
\end{cases}	
\end{equation}
where $S_0\in[0,T]$ is a stopping time, and $\xi$ is an $\mathcal F_{S_0}$-measurable random variable.
Throughout this section, we assume the following conditions:

\begin{assumptionp}{(H2)}\label{(H2)}
Let $\tau,$ $\lambda,$ and $p$ be the parameters satisfying Assumption~\ref{(H0)}, and let $S_0\in[0,T]$ be a stopping time.
\begin{itemize}
\item[$(1)$] Assume $\eta\in C^{\tau,\lambda;\beta}([0,T]\times\mathbb{R}^{d};\mathbb{R}^{M}),$ for some $\beta \ge 0,$ and for some $k >1$ and $x\in\R^d$,
\[X_{0} = x,\ \esssup_{\omega\in\Omega, t\in[0,T]}|X_{t\land S_{0}}|<\infty,\  m_{p,k}(X;[0,S_{0}])<\infty,    \|\xi\|_{L^{\infty}(\Omega)}<\infty.\]

\item[$(2)$] Assume $g_{i}\in C^2(\mathbb{R}^{N};\mathbb{R}^{N})$, and for some $C_1>0$,
\[\|g_{i}\|_{\infty;\R^N}\vee\|\nabla g_{i}\|_{\infty;\mathbb{R}^{N}}\vee \|\nabla^{2} g_{i}\|_{\infty;\mathbb{R}^{N}} \le C_{1}, \text{ for } i=1,2,...,M.\]

\item[$(3)$] Assume $f$ is progressively measurable and satisfies that for a.s. $\omega\in\Omega$ and for all $(t,y_{j},z_{j})\in[0,T]\times\mathbb{R}^{N}\times\mathbb{R}^{N\times d},\ j=1,2,$
\begin{equation}\label{e:(3) in (H2)}
|f(\omega,t\land S_{0}(\omega),y_{1},z_{1}) - f(\omega, t\land S_{0}(\omega),y_{2},z_{2})|\le C_{1}\left(|y_{1} - y_{2}| + |z_{1} - z_{2}|\right).
\end{equation}
In addition, assume 
\[\esssup_{\omega\in\Omega, t\in[0,T]}|f(\omega,t\land S_{0}(\omega),0,0)|<\infty.\]
\end{itemize}
\end{assumptionp}

\begin{assumptionp}{(A)}\label{(A)}
Let $x = (x^{1},...,x^{d})^{\top}\in\R^d$. We assume  $X$ is given by 
\begin{equation}\label{e:X_t = x + ...}
X_{t} = x + \int_{0}^{t}\sigma(r,X_{r})dW_{r} + \int_{0}^{t} b(r,X_{r}) dr,\ t\in[0,T],
\end{equation}
where $\sigma(t,x):[0,T]\times\mathbb{R}^{d}\to\mathbb{R}^{d\times d}$ and $b(t,x):[0,T]\times\mathbb{R}^{d}\to\mathbb{R}^{d}$ are measurable functions that are globally Lipschitz in $x$ and satisfy the following condition for some positive constant $L$:
\[\sup_{t\in[0,T],x\in \mathbb{R}^{d}}\left\{|\sigma(t,x)| \vee |b(t,x)|\right\}\le L.\]
\end{assumptionp}

BSDE with random terminal time was introduced by Peng~\cite{Peng-1991} (see also Darling-Pardoux~\cite{DP97}), by which  the definition of solution to Eq.~\eqref{e:BSDE'} is inspired.

\begin{definition}\label{def:BSDE'}
Under \ref{(H2)}, we call $(Y,Z)$ a solution of \eqref{e:BSDE'}, if $(Y,Z)\in\mathfrak{B}_{p,k}(0,T)$ satisfies 	\begin{equation}\label{e:bsde-3}
\begin{cases}
\displaystyle Y_{t} = \xi + \int_{t\wedge S_{0}}^{S_{0}}f(r,Y_{r},Z_{r})dr + \sum\limits_{i=1}^{M}\int_{t\wedge S_0}^{S_{0}}g_{i}(Y_{r})\eta_{i}(dr,X_{r}) - \int_{t\wedge S_0}^{S_{0}}Z_{r}dW_{r},  ~t\in[0,T],\\
Z_{t}= 0 \text{ on }[S_{0}\le t\le T].
\end{cases}
\end{equation}
\end{definition}

Note that $Y_{t} = \xi$ for $t \in  [S_{0},T].$ The existence and uniqueness of the solution to BSDE~\eqref{e:BSDE'} are guaranteed by the following proposition. 

\begin{proposition}\label{prop:SPDE-BSDE with stopping times}
Assume \ref{(H2)} holds. Then BSDE~\eqref{e:BSDE'} has a unique solution $(Y,Z)$ in $\mathfrak{B}_{p,k}(0,T)$.
\end{proposition}

\begin{proof}
Let $\bar{X}_{t} :=X_{t\land S_{0}}$ for $t\in[0,T]$, and
clearly, $\bar X$ satisfies \ref{(H1)}. Denoting \[\bar{g}(t,y) := g(y)\mathbf 1_{[0,S_{0}]}(t),\ \bar{f}(t,y,z) := f(t,y,z)\mathbf{1}_{[0,S_{0}]}(t),\] 
then Eq.~\eqref{e:BSDE'} can be written as
\begin{equation}\label{e:BSDE-stopping time2}
Y_{t} = \xi +   \int_{t}^{T}\bar{f}(r,Y_{r},Z_{r})dr + \sum_{i=1}^{M}\int_{t}^{T}\bar{g}_{i}(r,Y_{r})\eta_{i}(dr,\bar{X}_{r}) -\int_{t}^{T} Z_{r} dW_{r},\ t\in[0,T].
\end{equation}
As explained below the inequality \eqref{e:tau,lambda;0}, the norms $\|\cdot\|_{\tau, \lambda;\beta}$ and $\|\cdot\|_{\tau, \lambda}$ are equivalent if the spatial variable $x$ in $\eta(t,x)$ is restricted to a bounded domain. In view of the boundedness of $\bar X$ by \ref{(H2)}, the driver $\eta$ of the integral $\int\bar{g}_{i}(r,Y_{r})\eta_{i}(dr,\bar{X}_{r})$ in \eqref{e:BSDE-stopping time2} can be viewed as the driver restricted to the ball with radius $\esssup_{t,\omega}|\bar{X}_{t}|$. Therefore, we can assume without loss of generality that the condition concerning $\eta$ in Assumption~\ref{(H1)} is also satisfied. Although $\bar{g}(t,y)$ does not satisfy \ref{(H1)} directly, we can still follow the approach used in the proof of Theorem~\ref{thm:picard} to prove the well-posedness. 

More precisely, assume $M=1$ without loss of generality, and let $\bar{g}_{r} := \bar{g}(r,Y_r)$. Repeating the arguments in the proofs of Proposition~\ref{prop:ball-invariance}, Lemma~\ref{lem:prior estimate}, and Proposition~\ref{prop:contraction map}, the only difference occurs in the computation of the $k$-th moment of $\|\int_{\cdot}^{v}\bar{g}_{r}\eta(dr,X_{r})\|_{p\text{-var}}$ in \eqref{3.10} and $\|\int_{\cdot}^{v}\delta\bar{g}_{r}\eta(dr,X_{r})\|_{p\text{-var}}$ in \eqref{z2} for some $v\in[0,T]$. These moments can be bounded from above as follows:
\begin{equation*}
\mathbb{E}_{t}\left[\left\|\int_{\cdot}^{v}g(Y_{r})\mathbf{1}_{[0,S_{0}]}(r)\eta(dr,X_{r})\right\|^{k}_{p\text{-var};[t,v]}\right]\le\mathbb{E}_{t}\left[\left\|\int_{\cdot}^{v}g(Y_{r})\eta(dr,\bar{X}_{r})\right\|^{k}_{p\text{-var};[t,v]}\right]
\end{equation*}
and
\begin{equation*}
\mathbb{E}_{t}\left[\left\|\int_{\cdot}^{v}\delta g_{r}\mathbf{1}_{[0,S_{0}]}(r)\eta(dr,X_{r})\right\|^{k}_{p\text{-var};[t,v]}\right]\le \mathbb{E}_{t}\left[\left\|\int_{\cdot}^{v}\delta g_{r}\eta(dr,\bar{X}_{r})\right\|^{k}_{p\text{-var};[t,v]}\right],
\end{equation*}
where the estimations for the terms on the right-hand side were already obtained in the proofs of Proposition~\ref{prop:ball-invariance} and Proposition~\ref{prop:contraction map}. 
\end{proof}

\begin{remark}\label{rem:stopping times}
In view of the proof of Proposition~\ref{prop:SPDE-BSDE with stopping times}, the analogues of Proposition~\ref{prop:exp}, Corollary~\ref{cor:continuity solution map under H1}, Proposition~\ref{prop:exp'}, Corollary~\ref{cor:exp'}, and Corollary~\ref{cor:exp''} still hold under Assumption~\ref{(H2)} which replaces Assumption~\ref{(H1)}.
\end{remark}

The following result provides an upper bound for the solution of \eqref{e:bsde-3}. To emphasize the main items in the upper bound, we use the notation $\Theta_{1}:=(\tau,\lambda,\beta,p,T,C_{1},\|\eta\|_{\tau,\lambda;\beta})$ to simplify the dependence of constants, which will be used throughout the remainder of this section.

\begin{proposition}\label{prop:SPDE-estimate for Gamma} 
Assume \ref{(H2)} holds. Denote by $(Y,Z)$ the unique solution of BSDE~\eqref{e:BSDE'} in the space $\mathfrak{B}_{p,k}(0,T)$.  Let $\varepsilon\in(0,1)$ be a number satisfying $\tau + \frac{1-\varepsilon}{p}>1$.
Then for every $q>1,$ we have 
\begin{equation}\label{e:bound-Y}
\begin{aligned}
&\mathbb{E}\left[\left\|Y\right\|^{q}_{p\text{-}\mathrm{var};[0,S_{0}]}\right] + \E\bigg[ \Big|\int_{0}^{S_0} |Z_r|^2 dr\Big|^{\frac{q}2}\bigg] \\
&\lesssim_{\Theta_{1},\varepsilon,q} 1 + |x|^{\frac{q(\lambda+\beta)}{\varepsilon}} + \mathbb{E}\Big[\|X\|^{\frac{q(\lambda + \beta)}{\varepsilon}}_{p\text{-}\mathrm{var};[0,S_{0}]}\Big] + \E[|\xi|^q] + \mathbb{E}\left[\|f(\cdot,0,0)\|^{q}_{\infty;[0,S_{0}]}\right].
\end{aligned}
\end{equation}
\end{proposition}		

\begin{proof}[Proof of Proposition \ref{prop:SPDE-estimate for Gamma}]
Assume $M=1$ without loss of generality. Denote $f_{r} := f(r,Y_{r},Z_{r})$ for $r\in[0,T],$ and fix a $q>1.$

{\bf Step 1.} Firstly, we will prove the estimation \eqref{e:bound-Y} assuming the following additional condition:
\begin{equation}\label{e:addition-con}
\E\left[\|Y\|^{q}_{p\text{-var};[0,S_{0}]}\right] + \mathbb{E}\bigg[\Big|\int_{0}^{S_{0}}|Z_{r}|^{2}dr\Big|^{\frac{q}{2}}\bigg] <\infty.
\end{equation}
For an integer $v\ge 1,$ split $[0,T]$ into $[t_{i},t_{i+1}],$ $i=1,2,...,v$ evenly with the length $\theta := T/v.$ Let $\tau_{i} := t_{i}\land S_{0	}.$ By \eqref{e:bsde-3}, we claim that for $i=1,2,...,v,$
\begin{align}\label{e:bound-Y-3}
\mathbb{E}\left[\|Y\|^{q}_{\infty;[\tau_{i},\tau_{i+1}]}\right]
&= \mathbb{E}\left[\sup_{s\in[t_{i},t_{i+1}]}\bigg|\mathbb{E}_{s}\left[Y_{\tau_{i+1}} + \int_{s\land S_0}^{\tau_{i+1}}g(Y_{r})\eta(dr,X_{r}) + \int_{s\land S_0}^{\tau_{i+1}}f_{r}dr\right]\bigg|^{q}\right]\\\nonumber
&\lesssim_{p,q} \|Y_{\tau_{i+1}}\|^{q}_{L^{q}(\Omega)} +  \mathbb{E}\left[\Big\|\int_{\cdot}^{\tau_{i+1}}g(Y_{r})\eta(dr,X_{r})\Big\|^{q}_{p\text{-var};[\tau_{i},\tau_{i+1}]} + \Big|\int_{\tau_{i}}^{\tau_{i+1}}|f_{r}|dr\Big|^{q}\right],
\end{align}
and that
\begin{equation}\label{e:h_p,q bound}
\begin{aligned}
&\mathbb{E}\left[\left\|Y\right\|^{q}_{p\text{-var};[\tau_{i},\tau_{i+1}]}\right] + \mathbb{E}\left[\Big|\int_{\tau_{i}}^{\tau_{i+1}}|Z_{r}|^{2}dr\Big|^{\frac{q}{2}}\right]\\
&\lesssim_{\Theta_{1},q} \|Y_{\tau_{i+1}}\|^{q}_{L^{q}(\Omega)} + \mathbb{E}\left[\Big\|\int_{\cdot}^{\tau_{i+1}}g(Y_r)\eta(dr,X_r)\Big\|_{p\text{-var};[\tau_{i},\tau_{i+1}]}^{q}  + \Big|\int_{\tau_{i}}^{\tau_{i+1}}|f_{r}|dr\Big|^{q}\right].
\end{aligned}
\end{equation}
Indeed, for the second step in \eqref{e:bound-Y-3} we use the following equality for $s\in[t_i, t_{i+1}]$,  \begin{equation}\label{e:E[int g d eta]}
\begin{aligned}
&\mathbb{E}_{s}\left[\int_{s\land S_0}^{\tau_{i+1}}g(Y_r)\eta(dr,X_r)\right] = \mathbb{E}_{s}\left[\int_{\tau_{i}}^{\tau_{i+1}}g(Y_r)\eta(dr,X_r)\right]  - \int_{\tau_{i}}^{s\land S_0}g(Y_r)\eta(dr,X_r),
\end{aligned}
\end{equation}
and therefore, by Lemma~\ref{lem:BDG} with $M_t = \mathbb{E}_{t}[\int_{\tau_{i}}^{\tau_{i+1}}g(Y_{r})\eta(dr,X_{r})]$, we have 
\begin{equation}\label{e:BDG+Doob}
\begin{aligned}
&\mathbb{E}\left[\bigg\|\mathbb{E}_{\cdot}\left[\int_{\cdot}^{\tau_{i+1}} g(Y_{r}) \eta(dr,X_{r})\right]\bigg\|^{q}_{p\text{-var};[\tau_{i},\tau_{i+1}]}\right]\\
&\lesssim_{q} \mathbb{E}\left[\bigg\|\mathbb{E}_{\cdot}\left[\int_{\tau_{i}}^{\tau_{i+1}} g(Y_{r}) \eta(dr,X_{r})\right]\bigg\|^{q}_{p\text{-var};[\tau_{i},\tau_{i+1}]} + \Big\| \int_{\tau_{i}}^{\cdot} g(Y_{r}) \eta(dr,X_{r})\Big\|^{q}_{p\text{-var};[\tau_{i},\tau_{i+1}]}\right]\\
&\lesssim_{p,q} \mathbb{E}\left[\Big\| \int_{\cdot}^{\tau_{i+1}} g(Y_{r}) \eta(dr,X_{r})\Big\|^{q}_{p\text{-var};[\tau_{i},\tau_{i+1}]}\right]. 
\end{aligned}
\end{equation}
In addition, by a similar calculation, for the last term on the left-hand side of \eqref{e:bound-Y-3}, we have
\begin{equation*}
\mathbb{E}\left[\bigg\|\mathbb{E}_{\cdot}\Big[\int_{\cdot}^{\tau_{i+1}}f_{r}dr\Big]\bigg\|^{q}_{\infty;[\tau_{i},\tau_{i+1}]}\right] \lesssim_{q} \mathbb{E}\left[\Big|\int_{\tau_{i}}^{\tau_{i+1}}|f_{r}|dr\Big|^{q}\right],
\end{equation*}
which implies \eqref{e:bound-Y-3}.
Moreover, the inequality \eqref{e:h_p,q bound} holds by Eq.~\eqref{e:bsde-3} and the following estimate:
\begin{equation*}
\begin{aligned}
&\mathbb{E}\left[\Big\|\int_{\cdot}^{\tau_{i+1}}Z_{r}dW_{r}\Big\|^{q}_{p\text{-var};[\tau_{i},\tau_{i+1}]}\right] \lesssim_{p,q}\mathbb{E}\left[\Big\|\int_{\cdot}^{\tau_{i+1}}Z_{r}dW_{r}\Big\|^{q}_{\infty;[\tau_{i},\tau_{i+1}]}\right]\\
&\lesssim_{\Theta_{1},q} \mathbb{E}\left[\|Y\|^{q}_{\infty;[\tau_{i},\tau_{i+1}]}\right] + \mathbb{E}\left[\Big\|\int_{\cdot}^{\tau_{i+1}}g(Y_{r})\eta(dr,X_{r})\Big\|^{q}_{p\text{-var};[\tau_{i},\tau_{i+1}]}\right] + \mathbb{E}\left[\Big|\int_{\tau_{i}}^{\tau_{i+1}}|f_{r}|dr\Big|^{q}\right]\\
&\lesssim_{p,q} \|Y_{\tau_{i+1}}\|^{q}_{L^{q}(\Omega)} + \mathbb{E}\left[\Big\|\int_{\cdot}^{\tau_{i+1}}g(Y_{r})\eta(dr,X_{r})\Big\|^{q}_{p\text{-var};[\tau_{i},\tau_{i+1}]}\right] + \mathbb{E}\left[\Big|\int_{\tau_{i}}^{\tau_{i+1}}|f_{r}|dr\Big|^{q}\right],
\end{aligned}
\end{equation*}
where the first step follows from Lemma~\ref{lem:BDG}, and the third step is by \eqref{e:bound-Y-3},
which completes the proof of our claim.
Now we further estimate the last two terms on the right-hand side of \eqref{e:h_p,q bound}.  Noting $\|g(Y)\|_{\infty;[\tau_{i},\tau_{i+1}]}\le C_{1}$ and $\|g(Y)\|_{p\text{-var};[\tau_{i},\tau_{i+1}]}\lesssim \|Y\|_{p\text{-var};[\tau_{i},\tau_{i+1}]}$, by Lemma~\ref{lem:estimate for Young integral-for g is bounded}, we have 
\begin{equation*}
\begin{aligned}
&\Big\|\int_{\cdot}^{\tau_{i+1}}g(Y_{r})\eta(dr,X_{r})\Big\|_{p\text{-var};[\tau_{i},\tau_{i+1}]}\\
&\lesssim_{\Theta_{1},\varepsilon}\left(1 + \|X\|^{\lambda + \beta}_{\infty;[0,S_{0}]}\right)\left(\|g(Y)\|_{\infty;[\tau_{i},\tau_{i+1}]} + \|g(Y)\|^{\varepsilon}_{\infty;[\tau_{i},\tau_{i+1}]}\|g(Y)\|^{1-\varepsilon}_{p\text{-var};[\tau_{i},\tau_{i+1}]} \right)\\
&\quad + \left(1 + \|X\|^{\beta}_{\infty;[0,S_{0}]}\right)\|X\|^{\lambda}_{p\text{-var};[0,S_{0}]}\|g(Y)\|_{\infty;[\tau_{i},\tau_{i+1}]}\\
&\lesssim_{\Theta_{1}} \left(1+\|X\|_{\infty;[0,S_{0}]}^{\lambda + \beta}\right)\left(1 + \|Y\|^{1-\varepsilon}_{p\text{-var};[\tau_{i},\tau_{i+1}]}\right) +  \|X\|^{\lambda+\beta}_{p\text{-var};[0,S_{0}]}\,,
\end{aligned}
\end{equation*}
where the last inequality holds due to Young's inequality $\chi^{\beta}\upsilon^{\lambda}\le \chi^{\beta+\lambda} + \upsilon^{\beta+\lambda}$.  Thus, 
noting that 
\begin{align*}\E \left[\|X\|_{\infty;[0,S_{0}]}^{q(\lambda + \beta)}\|Y\|^{q(1-\varepsilon)}_{p\text{-var};[\tau_{i},\tau_{i+1}]}\right]\le &\left\{\E \left[\|X\|_{\infty;[0,S_{0}]}^{\frac{q(\lambda + \beta)}{\varepsilon}}\right]\right\}^{\varepsilon} \E\left\{\left[\|Y\|^{q}_{p\text{-var};[\tau_{i},\tau_{i+1}]}\right]  \right\}^{1-\varepsilon}\\
= & Q^{\varepsilon}\, \left\{\E\left[\|Y\|^{q}_{p\text{-var};[\tau_{i},\tau_{i+1}]}\right]  \right\}^{1-\varepsilon},
\end{align*} 
where we denote $Q := \mathbb{E}\Big[\|X\|^{\frac{ q(\lambda + \beta)}{\varepsilon}}_{p\text{-var};[0,S_{0}]}\Big]$, and that 
$\|X\|_{\infty;[0,S_0]} \le |x| + \|X\|_{p\text{-var};[0,S_0]},$
we have  
\begin{equation}\label{e:integral when g is bounded}
\begin{aligned}
&\mathbb{E}\left[\Big\|\int_{\cdot}^{\tau_{i+1}}g(Y_{r})\eta(dr,X_{r})\Big\|^{q}_{p\text{-var};[\tau_{i},\tau_{i+1}]}\right]\\
&\lesssim_{\Theta_{1},\varepsilon,q}(1+|x|^{q(\lambda + \beta)}+Q^\varepsilon)\left(1+\left\{\mathbb{E}\left[\|Y\|^{q}_{p\text{-var};[\tau_{i},\tau_{i+1}]}\right]\right\}^{1-\varepsilon}\right) + Q^{\varepsilon}\\
&\lesssim (1+|x|^{q(\lambda + \beta)}+Q^\varepsilon) \left(1 + \left\{\mathbb{E}\left[\|Y\|^{q}_{p\text{-var};[\tau_{i},\tau_{i+1}]}\right]\right\}^{1-\varepsilon}\right).
\end{aligned}
\end{equation}
Furthermore, since $\tau_{i+1} - \tau_{i}\le \theta\le T,$ the Lipschitz condition of $f(t,y,z)$ in $(y,z)$ yields that
\begin{equation}\label{e:integral when g is bounded 2}
\begin{aligned}
\Big\|\int_{\tau_{i}}^{\tau_{i+1}}|f_{r}|dr\Big\|_{L^{q}(\Omega)}&\lesssim_{\Theta_{1}}  \theta\left\|\left\|Y\right\|_{\infty;[\tau_{i+1},\tau_{i+1}]}\right\|_{L^{q}(\Omega)} + \Big\|\int_{\tau_{i}}^{\tau_{i+1}}|Z_{r}|dr\Big\|_{L^{q}(\Omega)}\\
&\quad  + \left\|\left\|f(\cdot,0,0)\right\|_{\infty;[\tau_{i},\tau_{i+1}]}\right\|_{L^{q}(\Omega)}\\
&\lesssim_{T} \theta\left\|\left\|Y\right\|_{p\text{-var};[\tau_{i},\tau_{i+1}]}\right\|_{L^{q}(\Omega)} + \theta^{\frac{1}{2}}\left\|\Big|\int_{\tau_{i}}^{\tau_{i+1}}|Z_{r}|^{2}dr\Big|^{\frac{1}{2}}\right\|_{L^{q}(\Omega)}\\
&\quad + \|Y_{\tau_{i+1}}\|_{L^{q}(\Omega)} + \left\|\left\|f(\cdot,0,0)\right\|_{\infty;[0,S_{0}]}\right\|_{L^{q}(\Omega)}.
\end{aligned}
\end{equation}
Combining \eqref{e:h_p,q bound}, \eqref{e:integral when g is bounded}, and \eqref{e:integral when g is bounded 2}, we obtain
\begin{equation}\label{e:SPDE-Y^i}
\begin{aligned}
&\left\|\left\|Y\right\|_{p\text{-var};[\tau_{i},\tau_{i+1}]}\right\|_{L^{q}(\Omega)} + \left\|\Big|\int_{\tau_{i}}^{\tau_{i+1}}|Z_{r}|^{2}dr\Big|^{\frac{1}{2}}\right\|_{L^{q}(\Omega)}\\
&\lesssim_{\Theta_{1},\varepsilon,q} (1+|x|^{\lambda + \beta} + Q^{\frac{\varepsilon}{q}})\left(1 + \left\{\left\|\|Y\|_{p\text{-var};[\tau_{i},\tau_{i+1}]}\right\|_{L^{q}(\Omega)}\right\}^{1-\varepsilon}\right)\\
&\quad + \theta\left\|\left\|Y\right\|_{p\text{-var};[\tau_{i},\tau_{i+1}]}\right\|_{L^{q}(\Omega)} + \theta^{\frac{1}{2}}\left\|\Big|\int_{\tau_{i}}^{\tau_{i+1}}|Z_{r}|^{2}dr\Big|^{\frac{1}{2}}\right\|_{L^{q}(\Omega)}\\
&\quad + \|Y_{\tau_{i+1}}\|_{L^{q}(\Omega)} + \left\|\left\|f(\cdot,0,0)\right\|_{\infty;[0,S_{0}]}\right\|_{L^{q}(\Omega)}.
\end{aligned}
\end{equation}
By applying Young's inequality  $\chi \upsilon\le \varepsilon \chi^{\frac{1}{\varepsilon}} + (1-\varepsilon) \upsilon^{\frac{1}{1-\varepsilon}}$ with $\varepsilon\in(0,1)$ to \eqref{e:SPDE-Y^i}, and noting that $\theta\le T^{\frac{1}{2}}\theta^{\frac{1}{2}},$ we can find a constant $C_{\varepsilon,q}>0$ such that
\begin{equation*}
\begin{aligned}
&\left\|\left\|Y\right\|_{p\text{-var};[\tau_{i},\tau_{i+1}]}\right\|_{L^{q}(\Omega)} + \left\|\Big|\int_{\tau_{i}}^{\tau_{i+1}}|Z_{r}|^{2}dr\Big|^{\frac{1}{2}}\right\|_{L^{q}(\Omega)}\\
&\le C_{\varepsilon,q}\bigg( 1 + |x|^{\frac{\lambda+\beta}{\varepsilon}} + Q^{\frac{1}{q}} + \left\|Y_{\tau_{i+1}}\right\|_{L^{q}(\Omega)} + \left\|\|f(\cdot,0,0)\|_{\infty;[0,S_{0}]}\right\|_{L^{q}(\Omega)}\\
&\qquad \qquad + \theta^{\frac{1}{2}}\left\|\left\|Y\right\|_{p\text{-var};[\tau_{i},\tau_{i+1}]}\right\|_{L^{q}(\Omega)} + \theta^{\frac{1}{2}}\left\|\Big|\int_{\tau_{i}}^{\tau_{i+1}}|Z_{r}|^{2}dr\Big|^{\frac{1}{2}}\right\|_{L^{q}(\Omega)}\bigg).
\end{aligned}
\end{equation*}
By the assumption \eqref{e:addition-con}, both sides of the above inequality are finite.  Letting $v := \lfloor 4 T C_{\varepsilon,q}^{2}\rfloor + 1$ and $\theta = \frac{T}{\lfloor 4 T C_{\varepsilon,q}^{2}\rfloor + 1}$, we have for $i=1,2,...,v$, 
\begin{equation}\label{e:tau_i Y,Z}
\begin{aligned}
&\left\|\left\|Y\right\|_{p\text{-var};[\tau_{i},\tau_{i+1}]}\right\|_{L^{q}(\Omega)} + \left\|\Big|\int_{\tau_{i}}^{\tau_{i+1}}|Z_{r}|^{2}dr\Big|^{\frac{1}{2}}\right\|_{L^{q}(\Omega)}\\
&\le 2C_{\varepsilon,q} \left( 1 + |x|^{\frac{\lambda+\beta}{\varepsilon}} + Q^{\frac{1}{q}} + \left\|Y_{\tau_{i+1}}\right\|_{L^{q}(\Omega)} + \left\|\|f(\cdot,0,0)\|_{\infty;[0,S_{0}]}\right\|_{L^{q}(\Omega)}\right).
\end{aligned}
\end{equation}
Noting that 
\[\left\|Y_{\tau_{i+1}}\right\|_{L^{q}(\Omega)} \le \|\xi\|_{L^{q}(\Omega)} + \sum_{j=i+1}^{v}\left\|\|Y\|_{p\text{-var};[\tau_{j},\tau_{j+1}]}\right\|_{L^{q}(\Omega)},\]
by the inequality \eqref{e:tau_i Y,Z}, we have
\begin{equation*}
\begin{aligned}
&\left\|\left\|Y\right\|_{p\text{-var};[\tau_{i},\tau_{i+1}]}\right\|_{L^{q}(\Omega)} + \left\|\Big|\int_{\tau_{i}}^{\tau_{i+1}}|Z_{r}|^{2}dr\Big|^{\frac{1}{2}}\right\|_{L^{q}(\Omega)}\\
&\le 2C_{\varepsilon,q} \left( 1 + |x|^{\frac{\lambda+\beta}{\varepsilon}} + Q^{\frac{1}{q}} + \left\|\xi\right\|_{L^{q}(\Omega)} + \left\|\|f(\cdot,0,0)\|_{\infty;[0,S_{0}]}\right\|_{L^{q}(\Omega)}\right)\\
&\quad + 2C_{\varepsilon,q} \sum_{j=i+1}^{v}\left[\left\|\|Y\|_{p\text{-var};[\tau_{j},\tau_{j+1}]}\right\|_{L^{q}(\Omega)} + \bigg\|\Big|\int_{\tau_{j}}^{\tau_{j+1}}|Z_{r}|^{2}dr\Big|^{\frac{1}{2}}\bigg\|_{L^{q}(\Omega)}\right].
\end{aligned}
\end{equation*}
Thus, by induction, it follows that, for  $i=v-1,v-2,...,1,$
\begin{equation*}
\begin{aligned}
&\Big\|\left\|Y\right\|_{p\text{-var};[\tau_i,\tau_{i+1}]}\Big\|_{L^{q}(\Omega)} + \bigg\|\Big|\int_{\tau_i}^{\tau_{i+1}}|Z_{r}|^{2}dr\Big|^{\frac{1}{2}}\bigg\|_{L^{q}(\Omega)}\\
&\lesssim_{\Theta_{1},\varepsilon,q}  1 + |x|^{\frac{\lambda+\beta}{\varepsilon}} + Q^{\frac{1}{q}} + \left\|\xi\right\|_{L^{q}(\Omega)} + \left\|\|f(\cdot,0,0)\|_{\infty;[0,S_{0}]}\right\|_{L^{q}(\Omega)}.
\end{aligned}
\end{equation*}
The above estimation implies that
\begin{equation}\label{e:Gronwall 1}
\begin{aligned}
&\Big\|\left\|Y\right\|_{p\text{-var};[0,S_{0}]}\Big\|_{L^{q}(\Omega)} + \bigg\| \Big|\int_{0}^{S_{0}}|Z_{r}|^{2}dr\Big|^{\frac{1}{2}}\bigg\|_{L^{q}(\Omega)}\\
&\lesssim_{\Theta_{1},\varepsilon,q}  1 + |x|^{\frac{\lambda+\beta}{\varepsilon}} + Q^{\frac{1}{q}} + \left\|\xi\right\|_{L^{q}(\Omega)} + \left\|\|f(\cdot,0,0)\|_{\infty;[0,S_{0}]}\right\|_{L^{q}(\Omega)},
\end{aligned}
\end{equation}
which gives \eqref{e:bound-Y}. 

{\bf Step 2.} Now we will prove that if $Q=\mathbb{E}\Big[\|X\|^{\frac{q(\lambda + \beta)}{\varepsilon}}_{p\text{-var};[0,S_{0}]}\Big]<\infty$, the condition \eqref{e:addition-con} holds, i.e., $\|(Y,Z)\|_{\mathfrak{H}_{p,q};[0,T]}<\infty$ (recall $\|\cdot\|_{\mathfrak{H}_{p,q};[0,T]}$ from \eqref{e:Hpk}).  Since $(Y,Z)\in\mathfrak{B}_{p,k}(0,T)$, we have that $R:=\esssup_{t\in[0,T],\omega\in\Omega}|Y_{t}|<\infty.$ For $n\ge 1,$ denote 
\[T_{n} := \inf\Big\{t>0;\Big(\|Y\|_{p\text{-var};[0,t]}\vee \int_{0}^{t}|Z_{r}|^{2}dr\Big) > n\Big\}\land S_{0}.\] 
Repeating the first part of the proof with $S_{0}$ replaced by $T_{n}$, \eqref{e:Gronwall 1} becomes
\[\sup_{n\ge 1}\|\left(Y_{\cdot\land T_{n}},Z_{\cdot} \mathbf{1}_{[0,T_{n}]}(\cdot)\right)\|_{\mathfrak{H}_{p,q};[0,T]}\lesssim_{\Theta_{1},\varepsilon,q} 1 + |x|^{\frac{\lambda+\beta}{\varepsilon}} + Q^{\frac{1}{q}} + R  + \left\|\|f(\cdot,0,0)\|_{\infty;[0,S_{0}]}\right\|_{L^{q}(\Omega)},\]
which has a finite right-hand side. This yields  $\|(Y,Z)\|_{\mathfrak{H}_{p,q};[0,T]}<\infty$ by Fatou's Lemma.		
\end{proof}

\begin{corollary}\label{cor:growth of the solution-unbounded}
Assume \ref{(H0)} and \ref{(A)} are satisfied. Suppose that $g$ satisfies $(2)$ in \ref{(H2)}, $f$ is a progressively measurable process satisfying \eqref{e:(3) in (H2)} in \ref{(H2)} with $S_0$ replaced by $T$,
and that for a.s. $\omega\in\Omega,$ for all $t\in[0,T],$
\[|f(\omega,t,0,0)|\lesssim 1 + |X_{t}|^{\frac{\lambda+\beta}{\varepsilon}\vee 1}.\]
If $(Y,Z)\in\mathfrak{H}_{p,q}(0,T)$, with some $q>1$, solves the following equation,
\begin{equation}\label{e:limit BSDE}
Y_{t} = Y_{T} + \int_{t}^{T}f(r,Y_r,Z_r)dr + \sum_{i=1}^{M}\int_{t}^{T}g_{i}(Y_{r})\eta_{i}(dr,X_{r}) - \int_{t}^{T}Z_{r}dW_{r},\ t\in[0,T].
\end{equation} 
Then
\[\|(Y,Z)\|_{\mathfrak{H}_{p,q};[0,T]}\lesssim_{\Theta_{1},\varepsilon,L,q} 1 + |x|^{\frac{\lambda+\beta}{\varepsilon}\vee 1} + \|Y_{T}\|_{L^{q}(\Omega)}.\]
\end{corollary}
\begin{proof}
The desired result follows by repeating step 1 of the proof of Proposition~\ref{prop:SPDE-estimate for Gamma} with the following estimation of $\|f(\cdot,0,0)\|_{\infty;[0,T]}$ under Assumptions~\ref{(A)}:
\[\left\|\|f(\cdot,0,0)\|_{\infty;[0,T]}\right\|_{L^{q}(\Omega)}\lesssim 1 + |x|^{\frac{\lambda+\beta}{\varepsilon}\vee 1} +  \Big\|\|X\|^{\frac{\lambda+\beta}{\varepsilon}\vee 1}_{p\text{-var};[0,T]}\Big\|_{L^{q}(\Omega)}\lesssim_{\varepsilon,L,q} 1 + |x|^{\frac{\lambda+\beta}{\varepsilon}\vee 1}.\]
\end{proof}

Now we estimate $ \mathbb{E}\big[\|Y\|^{q}_{p\text{-var};[S_{1},S_{0}]}\big]$, where $S_1\in[0,S_0]$ is a stopping time.

\begin{proposition}\label{prop:L^k estimate on stopping intervals}
Consider BSDE~\eqref{e:BSDE'} under Assumptions~\ref{(A)} and \ref{(H2)}. Let $\varepsilon\in(0,1)$ be a number that satisfies $\tau + \frac{1-\varepsilon}{p}>1$, and $S_1$ be a stopping time satisfying $S_1\le S_0$. Then, for each $q>1$, we have 
\begin{equation}\label{e:ineq-y-p-var}
\begin{aligned}
&\mathbb{E}\left[\|Y\|^{q}_{p\text{-}\mathrm{var};[S_{1},S_{0}]}\right] + \mathbb{E}\bigg[\Big|\int_{S_{1}}^{S_{0}}|Z_{r}|^{2}dr\Big|^{\frac{q}{2}}\bigg]\lesssim_{\Theta_{1},\varepsilon,L,q} \mathbb{E}\left[\left\|\mathbb{E}_{\cdot}\left[\xi\right]\right\|^{q}_{p\text{-}\mathrm{var};[S_1,S_0]}\right]\\
&\qquad\quad  + \left\{\mathbb{E}\left[|S_{0} - S_{1}|^{4q\tau}\right]\right\}^{\frac{1}{4}} \left(1 + |x|^{\frac{q(\lambda+\beta)}{\varepsilon}} + \left\|\xi\right\|^{q}_{L^{2q}(\Omega)} + \left\|\|f(\cdot,0,0)\|_{\infty;[0,S_{0}]}\right\|^{q}_{L^{2q}(\Omega)}\right) ,
\end{aligned}
\end{equation}
where $L$ is the constant introduced in Assumption~\ref{(A)}. 

\end{proposition}

\begin{proof} 
Assume $M=1$ for simplicity. To prove \eqref{e:ineq-y-p-var}, we shall estimate the $p$-variation of the last three terms on the right-hand side of Eq.~\eqref{e:bsde-3}, i.e. $\int_\cdot^{S_0} g(Y_r)\eta(dr, X_r),$ $\int_{\cdot}^{S_{0}}f(r,Y_{r},Z_{r})dr,$ and $\int_\cdot^{S_0} Z_r dW_r$. For the integral $\int_{\cdot}^{S_{0}}g(Y_{r})\eta(dr,X_{r}),$ by Lemma~\ref{lem:estimate for Young integral-for g is bounded}, we have  
\begin{equation}\label{e:SPDE-estimate-Stoping 1}
\begin{aligned}
&\mathbb{E}\bigg[\Big\|\int_{\cdot}^{S_0}g(Y_{r})\eta(dr,X_{r})\Big\|^{q}_{p\text{-var};[S_{1},S_{0}]}\bigg]\\
&\lesssim_{\Theta_{1},q} \|\eta\|^{q}_{\tau,\lambda;\beta}\mathbb{E}\bigg[|S_{0} - S_{1}|^{q\tau}\Big\{\left(1 + \|X\|^{q\beta}_{\infty;[S_{1},S_{0}]}\right)\|X\|^{q\lambda}_{p\text{-var};[S_{1},S_{0}]}\|g(Y_{\cdot})\|^{q}_{\infty;[S_{1},S_{0}]}\\
&\qquad  
+\Big(1 + \|X\|^{q(\lambda + \beta)}_{\infty;[S_{1},S_{0}]}\Big)\Big(\|g(Y_{\cdot})\|^{q}_{\infty;[S_{1},S_{0}]} + \|g(Y_{\cdot})\|^{q(1-\varepsilon)}_{p\text{-var};[S_{1},S_{0}]}\Big) \Big\}\bigg]\\
&\lesssim_{\Theta_{1},q}\mathbb{E}\bigg[\Big(1 + \|X\|^{q(\lambda + \beta)}_{\infty;[S_{1},S_{0}]}\Big)|S_{0} - S_{1}|^{q\tau}\Big(1 + \|X\|^{q\lambda}_{p\text{-var};[S_{1},S_{0}]} + \|Y\|^{q(1-\varepsilon)}_{p\text{-var};[S_{1},S_{0}]}\Big)\bigg]\\
&\lesssim_{\Theta_{1},L,q} (1 + |x|^{q(\lambda+\beta)})\left\{\mathbb{E}\left[ |S_{0} - S_{1}|^{4q\tau}\right]\right\}^{\frac{1}{4}}\left(1 + \left\{\mathbb{E}\left[\|Y\|^{2q}_{p\text{-var};[S_{1},S_{0}]}\right]\right\}^{\frac{1-\varepsilon}{2}}\right).
\end{aligned}
\end{equation}
Here, the last inequality of \eqref{e:SPDE-estimate-Stoping 1} follows from the following two facts. First, we have
\[\E\left[|(1+A)B(|D|^{1-\varepsilon}+E)|\right]\le \left\{\E[|B|^{4}]\right\}^{\frac{1}{4}} \Big( \left(1 + \left\{\E[|A|^{4}]\right\}^{\frac{1}{4}}\right) \left\{\E[|D|^{2}]\right\}^{\frac{1-\varepsilon}{2}} + \left\{\E[|(1+A)E|^{\frac{4}{3}}]\right\}^{\frac{3}{4}}\Big),\]
with $A = \|X\|^{q(\lambda + \beta)}_{\infty;[S_{1},S_{0}]},$ $B = |S_{0} - S_{1}|^{4q\tau},$ $D = \|Y\|^{q}_{p\text{-var};[S_{1},S_{0}]},$ $E = 1 + \|X\|^{q\lambda}_{p\text{-var};[S_{1},S_{0}]}$. Second, we have the following estimate (noting by Lemma~\ref{lem:BDG} that the moments of $\|X\|_{p\text{-var};[0,T]}$ are bounded by a constant depending on $L$)
\begin{equation*}
\begin{aligned}
&\left\{\mathbb{E}\left[|A|^{4}\right]\right\}^{\frac{1}{4}} +  \left\{\mathbb{E}\left[\left|\left(1 + A\right)E\right|^{\frac{4}{3}}\right]\right\}^{\frac{3}{4}}\\
&\lesssim_{\Theta_{1},L,q} |x|^{q(\lambda+\beta)} + \left\{\E[\|X\|^{4q(\lambda+\beta)}_{p\text{-var};[0,T]}]\right\}^{\frac{1}{4}} + \left\{\mathbb{E}\Big[\big| (1 + |x|^{q(\lambda+\beta)}  + \|X\|^{q(\lambda+\beta)}_{p\text{-var};[0,T]} )E\big|^{\frac{4}{3}}\Big]\right\}^{\frac{3}{4}}\\
&\lesssim_{\Theta_{1},L,q} 1 + |x|^{q(\lambda+\beta)}.
\end{aligned}
\end{equation*}

For the integral $\int_{\cdot}^{S_{0}}f(r,Y_{r},Z_{r})dr$, we have 
\begin{equation}\label{e:f terms}
\begin{aligned}
&\mathbb{E}\bigg[\Big|\int_{S_1}^{S_{0}}\left|f(r,Y_r,Z_r)\right|dr\Big|^{q}_{\infty;[S_{1},S_{0}]}\bigg]\\
&\lesssim_{\Theta_{1},q} \mathbb{E}\bigg[|S_{0} - S_{1}|^{\frac{q}{2}}\Big(\|f(\cdot,0,0)\|^{q}_{\infty;[0,S_{0}]} + \|Y\|^{q}_{\infty;[S_{1},S_{0}]}  + \Big|\int_{S_{1}}^{S_{0}}|Z_{r}|^{2}dr \Big|^{\frac{q}{2}}\Big)\bigg]\\
&\lesssim_{q} \Big(\big\|\|f(\cdot,0,0)\|_{\infty;[0,S_{0}]}\big\|^{q}_{L^{2q}(\Omega)} + \|\xi\|^{q}_{L^{2q}(\Omega)} + \left\{\mathbb{E}\left[ \|Y\|^{2q}_{p\text{-var};[0,S_{0}]}\right]\right\}^{\frac{1}{2}}\Big)\cdot \Big\{\mathbb{E}\Big[|S_{0} - S_{1}|^{q}\Big]\Big\}^{\frac{1}{2}} \\ 
&\qquad  + \bigg\{\mathbb{E}\bigg[\Big|\int_{0}^{S_{0}}|Z_{r}|^{2}dr\Big|^{q}\bigg]\bigg\}^{\frac{1}{2}} \Big\{\mathbb{E}\left[|S_{0} - S_{1}|^{q}\right]\Big\}^{\frac{1}{2}},
\end{aligned}
\end{equation}
where the second inequality follows from H\"older's inequality and $\|Y\|_{\infty;[0,S_{0}]}\le \|Y\|_{p\text{-var};[0,S_{0}]} + |\xi|$.
To calculate the $q$-th moment of $\|\int_{\cdot}^{S_0}Z_r dW_r\|_{p\text{-var};[S_{1},S_{0}]}$, we notice that  by \eqref{e:bsde-3}, 
\begin{equation}\label{e:SPDE-estimate-Stoping 2}
\begin{aligned}
&\int_{(t\land S_{0})\vee S_{1}}^{S_{0}} Z_{r} dW_{r}\\
&=  \xi - \mathbb{E}_{(t\land S_{0})\vee S_{1}}\left[\xi\right] + \int_{(t\land S_{0})\vee S_{1}}^{S_{0}} g(Y_{r}) \eta(dr,X_{r})
 - \mathbb{E}_{(t\land S_{0})\vee S_{1}}\left[\int_{(t\land S_{0})\vee S_{1}}^{S_{0}} g(Y_{r}) \eta(dr,X_{r})\right] \\
&\ \ \  + \int_{(t\land S_{0})\vee S_{1}}^{S_{0}}f(r,Y_r,Z_r)dr
 - \mathbb{E}_{(t\land S_{0})\vee S_{1}}\left[\int_{(t\land S_{0})\vee S_{1}}^{S_{0}}f(r,Y_r,Z_r)dr\right],\ \ \ \ t\in[0,T].
\end{aligned}
\end{equation}
Here, on the right-hand side of \eqref{e:SPDE-estimate-Stoping 2}, 
the $q$-th moment of $p$-variation of the third and fourth terms are bounded by the right-hand side \eqref{e:SPDE-estimate-Stoping 1}, which can be proved by \eqref{e:E[int g d eta]} and \eqref{e:BDG+Doob}. Similarly, the $q$-th moment of $p$-variation of the last two terms is bounded by the right-hand side of \eqref{e:f terms}. Also note that by
Proposition~\ref{prop:SPDE-estimate for Gamma}, we have
\begin{equation}\label{e:bound of (Y,Z)_p,2q}
\|(Y,Z)\|_{\mathfrak{H}_{p,2q};[0,T]}\lesssim_{\Theta_{1},\varepsilon,L,q} 1 + |x|^{\frac{\lambda+\beta}{\varepsilon}} + \|\xi\|_{L^{2q}(\Omega)} + \left\|\|f(\cdot,0,0)\|_{\infty;[0,S_{0}]}\right\|_{L^{2q}(\Omega)}.
\end{equation} 
Thus, by combining \eqref{e:SPDE-estimate-Stoping 1}, \eqref{e:f terms}, and \eqref{e:bound of (Y,Z)_p,2q}, in view of ~\eqref{e:SPDE-estimate-Stoping 2}, we have
\begin{align}\nonumber
&\mathbb{E}\bigg[\Big\|\int_{\cdot}^{S_{0}}Z_{r}dW_r\Big\|^{q}_{p\text{-var};[S_{1},S_{0}]}\bigg] \lesssim_{\Theta_{1},\varepsilon,L,q}\mathbb{E}\left[\|\mathbb{E}_{\cdot}[\xi]\|^{q}_{p\text{-var};[S_1,S_0]}\right] + (1 + |x|^{\frac{q(\lambda+\beta)}{\varepsilon}})\Big\{\mathbb{E}\left[ |S_{0} - S_{1}|^{4q\tau}\right]\Big\}^{\frac{1}{4}} \\\label{e:SPDE-estimate-Stoping 2.5}
&\qquad \quad + \left(1 + |x|^{\frac{q(\lambda+\beta)}{\varepsilon}} + \|\xi\|^{q}_{L^{2q}(\Omega)} + \left\|\|f(\cdot,0,0)\|_{\infty;[0,S_{0}]}\right\|^{q}_{L^{2q}(\Omega)}\right)\Big\{\mathbb{E}\left[|S_{0} - S_{1}|^{q}\right]\Big\}^{\frac{1}{2}}.
\end{align}
Finally, the desired estimate \eqref{e:ineq-y-p-var} follows by combining \eqref{e:SPDE-estimate-Stoping 1}, \eqref{e:f terms}, \eqref{e:bound of (Y,Z)_p,2q}, and \eqref{e:SPDE-estimate-Stoping 2.5}.
\end{proof}

We have the following result resembling Proposition~\ref{prop:L^k estimate on stopping intervals}, which in particular provides regularity of the solution to \eqref{e:ourBSDE} under the conditions in Section~\ref{sec:BSDE with bounded drift} when $X$ is a diffusion process satisfying Assumption~\ref{(A)}.

\begin{corollary}\label{cor:L^k estimate on stopping intervals'}
Assume $(Y,Z)$ is a solution of Eq.~\eqref{e:limit BSDE} on the entire interval $[0,T]$. Under the same assumptions and notation as in Corollary~\ref{cor:growth of the solution-unbounded}, assume further that $(Y,Z)\in \mathfrak{H}_{p,2q}(0,T)$. Let $S$ be a stopping time bounded by $T$. Then, we have
\begin{equation*}
\begin{aligned}
&\mathbb{E}\left[\left\|Y\right\|_{p\text{-}\mathrm{var};[S,T]}^{q}\right] + \mathbb{E}\left[\Big|\int_{S}^{T}|Z_{r}|^{2}dr\Big|^{\frac{q}{2}}\right]\\
&\lesssim_{\Theta_{1},\varepsilon,L,q}\left\{\mathbb{E}\left[|T-S|^{4q\tau}\right]\right\}^{\frac{1}{4}}\left(1 + |x|^{\frac{q(\lambda+\beta)}{\varepsilon}\vee q} + \left\|Y_{T}\right\|^{q}_{L^{2q}(\Omega)}\right) + \mathbb{E}\left[\|\mathbb{E}_{\cdot}[Y_{T}]\|^{q}_{p\text{-}\mathrm{var};[S,T]}\right].
\end{aligned}
\end{equation*}
\end{corollary}
\begin{proof}
The proof follows from the proof of Proposition~\ref{prop:L^k estimate on stopping intervals} with $\xi = Y_{T}$, $S_1 = S$, and $S_0=T$.
\end{proof}

\section{An application to SPDEs with Neumann boundary conditions}\label{sec:reflecting Feynman-Kac functional}

Let $D\subset\mathbb{R}^{d}$ be a bounded domain (i.e., a bounded, connected open set) with a $C^{\infty}$-boundary $\partial D$. Consider Eq.~\eqref{e:SPDEs} with Neumann boundary conditions:
\begin{equation}\label{e:SPDE Neumann}
\begin{cases}
\displaystyle-\partial_{t}u(t,x) = \frac{1}{2}\Delta u(t,x) + u(t,x)\partial_{t}B(t,x),\ (t,x)\in(0,T)\times D,\\[5pt]
\displaystyle u(T,x) = h(x),\ x\in\bar{D},\\[5pt] \displaystyle \partial_{\mathbf{n}}u(t,x) = 0,\ (t,x)\in (0,T)\times \partial D.
\end{cases}
\end{equation}
Here, $B(t,x)$ is a fractional Brownian sheet defined on a completed probability space $(\Omega',\mathcal{F}',\mathbb{P}')$,
with time Hurst parameter $H_0\in(\frac12, 1)$ and space Hurst parameters $H_i=H\in(0,1)$ for  $i=1, \dots, d$, $\mathbf{n}$ is the inward unit normal vector to the boundary $\partial D$ of  the domain $D$, and $\partial_{\mathbf{n}}$ denotes the inward normal derivative. 

In this subsection, we aim to show that $u(t,x)=\E^W[V_{t,x}]$ solves SPDE \eqref{e:SPDE Neumann},  where $\mathbb{E}^{W}[\cdot]$ is the expectation with respect to $W$, and 
\begin{equation}\label{e:Feynman-Kac functional}
V_{t,x}:= h(X^{t,x}_{T})\exp\left\{\int_{t}^{T}B(dr,X^{t,x}_{r})\right\},
\end{equation}
with $X^{t,x}_{s}$ being a reflected Brownian motion in the domain $D$ starting at $(t,x)$ with $x\in \bar{D}$. 

The reflected Brownian motion $X^{t,x}$ satisfies the so-called \emph{Skorohod equation}:
\begin{equation}\label{e:Skorohod eq}
X^{t,x}_{s}=x + W_s - W_t + \int_t^s \mathbf{n}\left(X_r^{t,x}\right) d L_r^{t,x}, \quad \text { for } s \geq t,
\end{equation}
where $W$ is a $d$-dimensional Brownian motion defined on $(\Omega,\mathcal{F},\mathbb{P})$ with filtration $(\mathcal{F}_{t})_{t\in[0,T]}$ (see  Section~\ref{sec:preliminaries}), and the solution pair $(X^{t,x},L^{t,x})$  satisfies that $L^{t,x}$ is a continuous, increasing, adapted process, and that $X^{t,x}_{s}\in \bar{D}$  and  $L^{t,x}_{s} = \int_{t}^{s} \mathbf{1}_{\{X^{t,x}_{r}\in\partial D\}}d L^{t,x}_{r}$ for all $s\ge t$. The well-posedness of Eq.~\eqref{e:Skorohod eq} was discussed, for instance, in Lions-Sznitman \cite{lions1984stochastic}.
The process $L^{t,x}$  coincides with the local time of $X^{t,x}$ on the boundary $\partial D$ (see Hsu \cite[Theorem~3.6]{Hsu}):
\begin{equation*}
L^{t,x}_{s} = \lim_{\rho\downarrow 0}\frac{1}{2\rho}\int_{t}^{s}\mathbf{1}_{\{\inf_{y\in \partial D}|y - X^{t,x}_{r}|<\rho\}}dr ~\text{ in }L^{2}(\Omega),\text{ for } s\in[t,T].
\end{equation*}

\

Classical Feynman-Kac formulae for Neumann PDEs can be found,  e.g., in \cite[Chapter~II, Theorem~5.1]{freidlin1985functional}.  To obtain Feynman-Kac formula for SPDE \eqref{e:SPDE Neumann},  we first prove the following result to validate the term $\E^W[V_{t,x}]$.

\begin{proposition}\label{prop:app 3}
Let $D\subset \R^d$ be a bounded domain with a $C^{\infty}$-boundary. Assume $h(\cdot)\in C(\bar{D};\R)$, and the Hurst parameters $H_0$ and $H$ satisfy $H_0\in(\frac12,1), H\in(0,1)$ and $H_0+\frac{H}{2}>1$. Let $X^{t,x}$ be given by \eqref{e:Skorohod eq},  and $V_{t,x}$ be given by    \eqref{e:Feynman-Kac functional}. Then, $u(t,x):= \mathbb{E}^{W}\left[V_{t,x}\right]\text{ is continuous in }(t,x)\in[0,T]\times \bar{D},\ \mathbb{P}'\text{-a.s.}$
Moreover, 
$\E^{B}[|u(t,x)|^{k}]<\infty,$ for every $k\ge 1$. 
\end{proposition}

\begin{remark}\label{rem:prob-fk} Assume \ref{(H0)} holds and $B(\cdot,\cdot)\in C^{\tau,\lambda;\beta}([0,T]\times\mathbb{R}^{d})$ $\mathbb{P}'(d \omega')$-a.s.  In view of Proposition \ref{prop:bounds}, to get the integrability of $V_{t,x}$ given by \eqref{e:Feynman-Kac functional} conditional on $B$, one needs to estimate the expectation of $\E^W \exp\left(C \|X^{t,x}_{\cdot}\|^\lambda_{p\text{-}\mathrm{var};[a,b]}\right)$ or $\E^W\exp\left(C \|X^{t,x}_{\cdot}\|^\lambda_{\frac1p\text{-}\mathrm{H\ddot ol};[a,b]}\right)$,  where $C$ is a nonnegative random variable depending only on $B$.

When the domain $D$ is convex, one has (see e.g. Tanaka~\cite[Lemma~2.6]{tanaka1979stochastic}), $\|X^{t,x}_{\cdot}\|_{\frac{1}{p}\text{-}\mathrm{H\ddot ol};[a,b]}\lesssim  \|W\|_{\frac{1}{p}\text{-}\mathrm{H\ddot ol};[a,b]}$, which implies the exponential integrability of  $\|X^{t,x}_{\cdot}\|_{\frac{1}{p}\text{-}\mathrm{H\ddot ol}}$ by  Fernique's Theorem~(\cite{fernique}). When $D$ is not convex, to our best knowledge,  the existing literature provides the following estimate according to Aida-Sasaki~\cite[Lemma~2.3]{aida2013wong}:
\begin{equation}\label{e:X<=exp{W}}
\|X^{t,x}_{\cdot}\|_{\frac{1}{p}\text{-}\mathrm{H\ddot ol};[a,b]}\lesssim \left(\|W\|_{\frac{1}{p}\text{-}\mathrm{H\ddot ol};[a,b]}+ \|W\|^{1+C}_{\frac{1}{p}\text{-}\mathrm{H\ddot ol};[a,b]}\right)\exp\Big\{C\|W_{\cdot} - W_{a}\|_{\infty;[a,b]}\Big\},
\end{equation}
which, however, does not yield the exponential integrability of $\|X^{t,x}_{\cdot}\|_{\frac{1}{p}\text{-}\mathrm{H\ddot ol};[a,b]}$. In the proof of Proposition \ref{prop:app 3}, we use BSDE theory developed in Section~\ref{sec:BSDE with bounded drift} to prove the integrability.
\end{remark}

\begin{proof}
Let $p>2$ be a constant such that $H_0+\frac{H}{p}>1.$ By Lemma~\ref{lem:Hurst parameter}, there exist $\tau<H_0$ and $\lambda<H$ such that $(\tau,\lambda,p)$ satisfies \ref{(H0)}, and $B\in C^{\tau,\lambda}_{\text{loc}}([0,T]\times\R^{d})$, $\mathbb{P}'(d \omega')$-a.s.

Fix a sample path of $B$ in $C^{\tau,\lambda}_{\text{loc}}([0,T]\times\R^{d})$. Let $m^{(s,t)}_{p,2}$ and $\|\cdot\|_{p,2;[s,t]}$ be defined by \eqref{e:def-norms} and \eqref{def:norms'}, respectively, 
with $\mathbb{E}_{t}$ replaced by $\mathbb{E}^{W}_{t}$, where $\mathbb{E}^{W}_{t}[\cdot]$ represents the corresponding conditional expectation with respect to $W$ up to time $t$. Since $W$ has independent increments, we have
(see, e.g., \cite[Lemma~B.1]{HuLe}), for every $c\in\R$,
\begin{equation}\label{e:BMO of exp W}
\esssup_{s\in[0,T],\omega\in\Omega}\mathbb{E}^{W}_{s}\left[\exp\left\{c\|W\|_{\frac{1}{p}\text{-H\"ol};[s,T]}\right\}\right] = \mathbb{E}^{W}\left[\exp\left\{c\|W\|_{\frac{1}{p}\text{-H\"ol};[0,T]}\right\}\right] <\infty.
\end{equation}
Then $\sup_{(t,x)}m_{p,2}(X^{t,x};[t,T])<\infty$ by \eqref{e:X<=exp{W}}. In addition, since $X^{t,x}$ is bounded and $B\in C^{\tau,\lambda}_{\text{loc}}([0,T]\times\R^d)$, due to the reason stated in Remark \ref{rem:C},  Proposition~\ref{prop:exp} and Corollary~\ref{cor:exp'} still apply for $\eta(t,x) = B(t,x).$ Thus, denoting
\begin{equation*}
Y^{t,x}_{s}:=\mathbb{E}^{W}_{s}\left[h(X^{t,x}_{T})\exp\left\{\int_{s}^{T}B(dr,X^{t,x}_{r})\right\}\right],\ s\in[t,T],
\end{equation*}
there exists a progressively measurable process $Z^{t,x}$ such that $(Y^{t,x},Z^{t,x})\in \mathfrak{B}_{p,2}(t,T)$, and $(Y^{t,x},Z^{t,x})$ uniquely solves the following linear BSDE,
\begin{equation}\label{e:Neu linear BSDE}
Y^{t,x}_{s} = h(X^{t,x}_{T}) + \int_{s}^{T}Y^{t,x}_{r}B(dr,X^{t,x}_{r}) - \int_{s}^{T}Z^{t,x}_{r}dW_{r},\ s\in[t,T].
\end{equation}
Therefore, we have 
\begin{equation*}
\left|\mathbb{E}^{W}\left[ V_{t,x}\right]\right| = |Y^{t,x}_{t}| \le  \|(Y^{t,x},Z^{t,x})\|_{p,2;[t,T]} < \infty. 
\end{equation*}

We now prove the continuity of $u$.  Let $F(x)$ be a smooth cutoff function such that $F(x)=1$ for $x\in D$, and $F(x)=0$ for  $x\notin D$ with $d(x,D)>1$. Then, Eq.~\eqref{e:Neu linear BSDE} can be rewritten with $B(t,x)$ replaced by $F(x)B(t,x)$. For every $n\ge 1$, define $B^{n}(t,x) := \int_{\R}\rho^{n}(t-s)B(s,x)ds,$ where $\rho^n\in C^{\infty}_{c}(\R)$ is a standard mollifier (see e.g. the proof of Lemma~\ref{lem:smooth approximation of eta}). By the proof of Lemma~\ref{lem:smooth approximation of eta}, we have that $\lim_{n}\|F(x)B(t,x) - F(x)B^{n}(t,x)\|_{\tau',\lambda} = 0$, for every $\tau'\in(0,\tau)$. Therefore, we can choose $\tau',\lambda\in(0,1]$ so that Assumption~ \ref{(H0)} holds with $(\tau,\lambda)$ replaced by $(\tau',\lambda)$. For $n\ge 1$, let $V^{n}_{t,x} := h(X^{t,x}_{T})\exp\{\int_{t}^{T}B^{n}(ds,X^{t,x}_{s})\}$. By Proposition~\ref{prop:continuity of solution map}, as $n\to \infty$,
\begin{equation}\label{e:V^n -> V}
\left|\mathbb{E}^{W}\left[V^{n}_{t,x}\right] - \mathbb{E}^{W}\left[V_{t,x}\right]\right| \rightarrow 0,\quad \text{uniformly in }(t,x)\in[0,T]\times\bar{D}.
\end{equation}
Furthermore, since $B^{n}$ is smooth in time and $\partial_{t} B^{n}(t,x)$ is continuous, we have $$\int_{t}^{T}B^{n}(ds,X^{t,x}_{s}) = \int_{t}^{T}\partial_{s}B^{n}(s,X^{t,x}_{s})ds.$$ 
Thus, we can apply Pardoux-Zhang~\cite[Proposition~4.1]{PardouxZhanggeneralized} and get that $\mathbb{E}^{W}\left[V^{n}_{t,x}\right]$ is continuous in $(t,x)$. Then, the uniform convergence \eqref{e:V^n -> V} yields the continuity of  $u(t,x)=\E^W[V_{t,x}]$.

For the integrability of $u$, by the \emph{a priori} estimate for the solution $Y^{t,x}$ in Lemma~\ref{lem:prior estimate}, we have 
\begin{equation}\label{e:a priori est for Y}
|Y^{t,x}_{s}| \lesssim \exp\left\{C \left\lfloor (T - t)/\delta \right\rfloor \right\},
\end{equation}
where $C$ is a constant depending on $(\tau,\lambda,p,k)$. By the proof of Proposition~\ref{prop:ball-invariance} (under \eqref{e:prop:ball-invariance-bound of solution}), it suffices to find $\delta$ such that  $(\delta + \delta^{\tau} + \delta^{1/2})\cdot C (1 + \|F\cdot B\|_{\tau,\lambda})<1/2$. Note that the presence of $\delta^{1/2}$ is due to  the  function $f(r,y,z)$, while in Eq.~\eqref{e:Neu linear BSDE} $f\equiv 0$. Hence, we may choose $\delta>0$ such that $\delta^{\tau} \lesssim \|F\cdot B\|^{-1}_{\tau,\lambda}$. Therefore, Eq. \eqref{e:a priori est for Y} yields that $|u(t,x)| =  |Y^{t,x}_{t}| \lesssim \exp\{C  (T-t) \|F\cdot B\|^{1/\tau}_{\tau,\lambda} \}$.
Then, we get, for $k\ge 1$,
\begin{equation*}
\E^B[|u(t,x)|^k] \lesssim \exp\Big\{k C  (T-t) \|F\cdot B\|^{1/\tau}_{\tau,\lambda} \Big\} \lesssim_k \E^B\left[\exp\left\{\rho \|F\cdot B\|^2_{\tau,\lambda}\right\}\right]<\infty,
\end{equation*}
where the second inequality holds since $1/\tau<2$ and the last one  follows from Fernique's Theorem (\cite{fernique}) with  $\rho$ being some positive constant. This completes the proof.
\end{proof}

Now we provide the definition of \emph{weak solution} to Eq.~\eqref{e:SPDE Neumann}. Then we will show that the Feynman-Kac formula $u(t,x)=\mathbb{E}^{W}\left[V_{t,x}\right]$ is a weak solution under some proper conditions.

\begin{definition}\label{def:Neumann}
We call $u(t,x)$ a \emph{weak solution} to Eq.~\eqref{e:SPDE Neumann}, if there exists  $\tilde{\tau}\in(1 - H_0 , 1]$ such that $u(\cdot,x)\in C^{\tilde{\tau}\text{-}\mathrm{H\ddot ol}}([0,T];\R)$ for all $x\in\bar{D}$,  
and $u(t,x)$ satisfies
\begin{equation}\label{e:weak solu}
\begin{aligned}
\int_{D}u(t,x) \varphi(x) dx & =  \int_{D}h(x)\varphi(x) dx + \frac{1}{2}\int_{D}\int_{t}^{T}u(s,x)\Delta \varphi(x) ds dx \\
+ \frac12 & \int_{\partial D} \int_{t}^{T} u(s,x)\partial_{\mathbf{n}}\varphi(x) ds dS   +\int_{D}\int_{t}^{T} u(s,x) \varphi(x)B(ds,x) dx, \ \mathbb{P}'\text{-a.s.}
\end{aligned}
\end{equation}
for all $t\in[0,T]$ and $\varphi\in C^{2}(\bar{D};\R).$ Here,  $dS$ denotes the surface integral (of a scalar field) and $\int_{t}^{T}u(s,x) \varphi(x) B(ds,x)$ is a nonlinear Young integral.
\end{definition}

\begin{proposition}\label{prop:Feynman-Kac for Neumann}
Let $D\subset \R^d$ be a bounded domain with a $C^{\infty}$-boundary $\partial D$, $h\in C^{\mathrm{Lip}}(\bar{D};\R)$, and the Hurst parameters $H_0$ and $H$ satisfy $H_0\in(1/2,1),H\in(0,1)$, and $H_{0}+\frac{H}{4} > 1$. Let $u(t,x):=\mathbb{E}^{W}\left[V_{t,x}\right]$. 
Then, for $\tau^*\in(0,H_{0} + \frac{H}{2} - 1)$, $u(\cdot,x)\in C^{\tau^*\text{-}\mathrm{H\ddot{o}l}}([0,T];\R)$ for all $x\in\bar{D}$, $\mathbb{P}'$-a.s. Moreover, $u(t,x)$ is a weak solution to Eq.~\eqref{e:weak solu}.
\end{proposition}

\begin{remark}
The conditions in Proposition~\ref{prop:app 3} do not involve the space dimension~$d$, since the reflected Brownian motion $X$ is bounded (see Remark~\ref{rem:C}). For the same reason, the conditions imposed in Proposition~\ref{prop:Feynman-Kac for Neumann} are independent of the  space dimension.  In contrast, the dimension $d$ usually plays a critical role in the study of (stochastic) PDEs in the whole space $\R^d$. 
\end{remark}

\begin{proof}
Noting $H_0\in(1/2,1),H\in(0,1)$, and $H + 4 H_{0} > 4$, by Lemma~\ref{lem:Hurst parameter}, there exists $(\tau,\lambda,\beta)\in(\frac{1}{2},1)\times (0,1)\times [0,\infty)$ such that $B(t,x)\in C^{\tau,\lambda;\beta}([0,T]\times \R^d;\R)$, $\mathbb{P}'$-a.s., and that $
\tau\in(1/2,1),\ \lambda\in(0,1)$, and $ \lambda + 4\tau > 4$.  Let $B^{n}(t,x) := \int_{\R}\rho^{n}(t-s)B(s,x)ds,$ where $\rho^n\in C^{\infty}_{c}(\R)$ is a standard mollifier. Then by Lemma~\ref{lem:smooth approximation of eta'}, there exists $\tau'\in (0,\tau)$ such that
\begin{equation}\label{e:Bn -> B}
\tau'\in(1/2,1),\ \lambda + 4\tau' > 4,\text{ and }\lim_{n\rightarrow \infty}\|B^{n} - B\|_{\tau',\lambda;\beta} = 0\quad \mathbb{P}'\text{-a.s.}
\end{equation}
Let $Y^{n;t,x}_{s} := \mathbb{E}^{W}_{s}\left[h(X^{t,x}_{T})\exp\left\{\int_{s}^{T}B^{n}(dr,X^{t,x}_{r})\right\}\right],$ where $X^{t,x}$ is given by \eqref{e:Skorohod eq}, and denote $u^{n}(t,x) := Y^{n;t,x}_{t}$. Then, by \cite[Theorem~3]{Wong-Yang-Zhang-22}, $u^{n}$ belongs to  $L^{2}([0,T];H^{1}(D))$, $\mathbb{P}'$-a.s., and satisfies  
\begin{equation*}
\begin{aligned}
\int_{D}u^{n}(t,x)\varphi(x)dx &= \int_{D}h(x)\varphi(x)dx - \frac{1}{2}\int_{D}\int_{t}^{T}\left(\nabla u^{n}(s,x)\right)^{\top} \nabla \varphi(x) ds dx \\
&\quad + \int_{D}\int_{t}^{T}u^{n}(s,x)\varphi(x)B^{n}(ds,x)dx, ~ \text{ for all } \varphi\in C^{2}(\bar{D}).
\end{aligned}
\end{equation*}
Hence, by Green's formula (see, e.g., \cite[(4.31) in Chapter~4]{taylor1996partial}), for a.e. $s\in[0,T],$
\[- \int_{D} \left( \nabla u^n(s,x)\right)^{\top} \nabla \varphi(x) dx = \int_{D} u^n(s,x) \Delta \varphi(x) dx + \int_{\partial D} u^{n}(s,x) \partial_{\mathbf{n}}\varphi(x) dS,\] 
and then $u^{n}$ is a weak solution to \eqref{e:weak solu} with $B$ replaced by $B^{n}$.

By \eqref{e:V^n -> V} in the proof of Proposition~\ref{prop:app 3}, we have, as $n\to \infty$, 
\begin{equation}\label{e:u^n -> u}
u^{n}(t,x) \rightarrow u(t,x)\text{ uniformly in }(t,x)\in[0,T]\times \bar{D}, \quad \mathbb{P}'\text{-a.s.}
\end{equation}
Therefore, for $\varphi\in C^{2}(\bar{D})$, the following three convergences hold  as $n\rightarrow\infty$: 
\begin{equation*}
    \begin{cases}
\displaystyle\int_{D}u^{n}(t,x)\varphi(x)dx\rightarrow \int_{D}u(t,x)\varphi(x)dx;\\[6pt]
\displaystyle \int_{D}\int_{t}^{T} u^{n}(s,x)\Delta \varphi(x) ds dx \rightarrow \int_{D}\int_{t}^{T} u(s,x)\Delta \varphi(x) ds dx;\\[6pt]
\displaystyle \int_{\partial D} u^{n}(s,x) \partial_{\mathbf{n}}\varphi(x) dS\rightarrow \int_{\partial D} u(s,x) \partial_{\mathbf{n}}\varphi(x) dS.
    \end{cases}
\end{equation*}
Thus, to show that $u$ is a weak solution to Eq.~\eqref{e:SPDE Neumann}, it suffices to prove:
\begin{itemize}
\item[(a)] there exists $\tilde{\tau}\in (1 - H_0, 1]$ such that for every $x\in\R^d$,
\begin{equation}\label{e:weak sol condi 1}
\|u(\cdot,x)\|_{\tilde{\tau}\text{-H\"ol};[0,T]}<\infty;
\end{equation}

\item[(b)] the following convergence holds a.s.:
\begin{equation}\label{e:weak sol condi 2}
\int_{D}\int_{t}^{T}u^{n}(s,x)\varphi(x)B^{n}(ds,x)dx\to \int_{D}\int_{t}^{T}u(s,x)\varphi(x)B(ds,x)dx, ~\text{ as } n\to\infty. 
\end{equation}
\end{itemize}

\

To prove \eqref{e:weak sol condi 1}, we first establish a uniform bound for $\big\{\|u^{n}(\cdot,x)\|_{\tau^*\text{-H\"ol};[0,T]}, n\ge 1\big\}$ with
$\tau^*\in\left( 0,H_0 + \frac{H}{2} - 1 \right).$
Denote $v^{n}(t,x) := \int_{t}^{T}B^{n}(dr,X^{t,x}_{r})$ and $v(t,x) := \int_{t}^{T}B(dr,X^{t,x}_{r})$. Note that 
\begin{equation*}
u^{n}(t,x) = \E^{W}\left[h(X^{t,x}_{T})\exp\left\{v^{n}(t,x)\right\}\right] \text{ and } u(t,x) = \E^{W}\left[h(X^{t,x}_{T})\exp\left\{v(t,x)\right\}\right],
\end{equation*}
and hence for $0\le s< t\le T$,
\begin{equation}\label{e:expression of u^n - u^n}
u^{n}(t,x) - u^n(s,x) = \mathbb{E}^{W}\left[h(X^{t,x}_{T})\exp\left\{v^{n}(t,x)\right\} - h(X^{s,x}_{T})\exp\left\{v^n(s,x)\right\}\right].
\end{equation}
By the boundedness of $X^{s,x}$ and H\"older's inequality, 
\begin{equation}\label{e:Hol to u(t,x)}
\begin{aligned}
&\left|\mathbb{E}^{W}\left[h(X^{s,x}_{T})\left(\exp\left\{v^{n}(t,x)\right\} - \exp\left\{v^n(s,x)\right\}\right)\right]\right|\\
& \lesssim \mathbb{E}^{W}\left[\left(\exp\left\{v^{n}(t,x)\right\} + \exp\left\{v^{n}(s,x)\right\}\right)\left|v^{n}(t,x) - v^{n}(s,x)\right|\right]\\
& \le \left\{\mathbb{E}^{W}\left[\exp\left\{4v^{n}(t,x)\right\}\right]\right\}^{\frac{1}{4}} \left\{\mathbb{E}^{W}\left[\exp\left\{4v^{n}(s,x)\right\}\right]\right\}^{\frac{1}{4}} \left\{\mathbb{E}^{W}\left[\left|v^{n}(t,x) - v^{n}(s,x)\right|^{2}\right]\right\}^{\frac{1}{2}}.
\end{aligned}
\end{equation}
Moreover,  Lipschitzness of $h(\cdot)$ implies
\begin{equation}\label{e:to the Hol con of u^n}
\begin{aligned}
&\left|\E^{W}\left[\left(h(X^{t,x}_{T}) - h(X^{s,x}_{T})\right)\exp\left\{v^{n}(t,x)\right\}\right]\right|\\
&\quad \lesssim \left\{\E^{W}\left[|X^{t,x}_{T} - X^{s,x}_{T}|^2\right]\right\}^{\frac{1}{2}}\left\{\mathbb{E}^{W}\left[\exp\left\{2v^{n}(t,x)\right\}\right]\right\}^{\frac{1}{2}}.
\end{aligned}
\end{equation}

To obtain a uniform upper bound for   $\|u^n(\cdot,x)\|_{\tau^*\text{-H\"ol};[0,T]}$, 
we will estimate the right-hand side of \eqref{e:expression of u^n - u^n}, and hence the right-hand sides of \eqref{e:Hol to u(t,x)} and \eqref{e:to the Hol con of u^n} by triangular inequality. This  will be done in the following  three steps.

\ 

{\bf (i)} In the first step, we will estimate $\mathbb{E}^{W}\left[\left|v^{n}(t,x) - v^{n}(s,x)\right|^2\right]$ for $0\le s < t \le T$. Triangular inequality yields that,  for  $n\ge1$ and $0\le s < t\le T$,
\begin{equation}\label{e:tilde v^n_[s,t]}
\begin{aligned}
|v^{n}(t,x) - v^{n}(s,x)|&\le \left|\int_{s}^{t}B^{n}(dr,X^{s,x}_{r})\right| + \left|\int_{t}^{T}B^{n}(dr,X^{s,x}_{r}) - \int_{t}^{T}B^{n}(dr,X^{t,x}_{r})\right|.
\end{aligned}
\end{equation}
For the first term on the right-hand side, by Proposition~\ref{prop:bounds},
\begin{equation}\label{e:int_[s,t]B^n}
\mathbb{E}^{W}\left[\left|\int_{s}^{t}B^{n}(dr,X^{s,x}_{r})\right|^{2}\right]\lesssim \|B^{n}\|^{2}_{\tau',\lambda;\beta}|t - s|^{2\tau'}.
\end{equation}
For the second term on the right-hand side of \eqref{e:tilde v^n_[s,t]}, by \cite[Proposition~2.11]{HuLe}, we have for every $p>2$ and $\theta\in(0,1)$ satisfying $\tau' + \frac{\theta\lambda}{p} > 1$, 
\begin{equation}\label{e:B(dr,Xs) - B(dr,Xt)}
\begin{aligned}
&\left|\int_{t}^{T}B^{n}(dr,X^{s,x}_{r}) - \int_{t}^{T}B^{n}(dr,X^{t,x}_{r})\right| \\
&	 \lesssim \|B^{n}\|_{\tau',\lambda;\beta}\left(1 + \|X^{s,x}\|^{\beta}_{\infty;[t,T]} + \|X^{t,x}\|^{\beta}_{\infty;[t,T]}\right)\bigg(\|X^{t,x} - X^{s,x}\|^{\lambda}_{\infty;[t,T]}|T|^{\tau'} \\
&\quad\quad\quad \quad +\left(\|X^{s,x}\|^{\lambda}_{\frac{1}{p}\text{-H\"ol};[t,T]} + \|X^{t,x}\|^{\lambda}_{\frac{1}{p}\text{-H\"ol};[t,T]}\right)^{\theta}\|X^{t,x} - X^{s,x}\|^{(1-\theta)\lambda}_{\infty;[t,T]}|T|^{\tau' + \frac{\theta\lambda}{p}}\bigg).
\end{aligned}
\end{equation}

In order to bound the right-hand side of \eqref{e:B(dr,Xs) - B(dr,Xt)}, we will estimate the moments of $\|X^{t,x} - X^{s,x}\|_{\infty;[t,T]}$. By the same argument in the proof of \cite[Theorem~5.1.2]{zhang2017backward}, one can prove the flow property $X^{s,x}_{r} = X^{t,X^{s,x}_{t}}_{r}$ for $s\in[t,T].$ Hence, by the 
fourth moment estimate obtained in \cite[Proposition~3.1]{PardouxZhanggeneralized} and the Markov property, we have
\begin{equation}\label{e:four moments}
\mathbb{E}^{W}\left[\left\|X^{t,x}_{\cdot} - X^{t,X^{s,x}_{t}}_{\cdot}\right\|^{4}_{\infty;[t,T]}\right]\lesssim \mathbb{E}^{W}\left[\left|X^{s,x}_{t} - x\right|^{4}\right].
\end{equation}
Moreover, the estimate \eqref{e:X<=exp{W}} yields
\begin{equation}\label{e:|X^s,x - x|}
\left|X^{s,x}_{t} - x\right|\lesssim \left(\|W\|_{\frac{1}{p}\text{-H\"ol};[s,t]} + \|W\|^{1 + C}_{\frac{1}{p}\text{-H\"ol};[s,t]}\right)\exp\left\{C\|W_{\cdot} - W_{s}\|_{\infty;[s,t]}\right\} |t-s|^{\frac{1}{p}}.
\end{equation}
Then, combining \eqref{e:four moments}, \eqref{e:|X^s,x - x|}, and \eqref{e:BMO of exp W}, we have 
\begin{equation}\label{e:4 moments}
\begin{aligned}
\E^{W}\left[\left\|X^{t,x}_{\cdot} - X^{s,x}_{\cdot}\right\|^{4}_{\infty;[t,T]}\right] = \mathbb{E}^{W}\left[\left\|X^{t,x}_{\cdot} - X^{t,X^{s,x}_{t}}_{\cdot}\right\|^{4}_{\infty;[t,T]}\right]\lesssim \mathbb{E}^{W}\left[\left|X^{s,x}_{t} - x\right|^{4}\right]\lesssim |t-s|^{\frac{4}{p}}.
\end{aligned}
\end{equation}
Thus,  \eqref{e:B(dr,Xs) - B(dr,Xt)} implies that,   noting \eqref{e:4 moments}, 
\begin{equation}\label{e:B(t,X^s) - B(t,X^t)}
\begin{aligned}
&\mathbb{E}^{W}\left[\left|\int_{t}^{T}B^{n}(dr,X^{s,x}_{r}) - \int_{t}^{T}B^{n}(dr,X^{t,x}_{r})\right|^{2}\right]\\
&\lesssim \|B^{n}\|^{2}_{\tau',\lambda;\beta}\Bigg(\mathbb{E}^{W}\left[ \|X^{t,x} - X^{s,x}\|^{2\lambda}_{\infty;[t,T]} \right] + \left\{\E^{W}\left[\|X^{t,x}\|^{4\lambda\theta}_{p\text{-var};[t,T]} + \|X^{s,x}\|^{4\lambda\theta}_{p\text{-var};[t,T]}\right]\right\}^{\frac{1}{2}}\\
&\qquad\qquad\qquad\qquad\qquad\qquad\qquad\qquad\qquad\qquad\qquad\cdot  \left\{ \E^{W}\left[\|X^{t,x} - X^{s,x}\|^{4(1-\theta)\lambda}_{\infty;[t,T]} \right]\right\}^{\frac{1}{2}}\Bigg)\\
&\lesssim \|B^{n}\|^{2}_{\tau',\lambda;\beta}\left(\left\{\E^{W}\left[\|X^{t,x} - X^{s,x}\|^{4}_{\infty;[t,T]}\right]\right\}^{\frac{\lambda}{2}} + \left\{\E^{W}\left[\|X^{t,x} - X^{s,x}\|^{4}_{\infty;[t,T]}\right]\right\}^{\frac{(1-\theta)\lambda}{2}} \right)\\
&\lesssim \|B^{n}\|^{2}_{\tau',\lambda;\beta}|t-s|^{\frac{2(1-\theta)\lambda}{p}}.
\end{aligned}
\end{equation}
Combining  \eqref{e:tilde v^n_[s,t]}, \eqref{e:int_[s,t]B^n}, and \eqref{e:B(t,X^s) - B(t,X^t)}, we have (notice $\frac{(1-\theta)\lambda}{p} < \tau'$ by $ \tau' + \frac{\theta\lambda}{p}>1 > \frac{\lambda}{p}$)
\begin{equation}\label{e:|tilde v^n|^2}
\mathbb{E}^{W}\left[\left|v^{n}(s,x) - v^{n}(t,x)\right|^{2}\right]\lesssim\|B^n\|^2_{\tau',\lambda;\beta} |t-s|^{\frac{2(1-\theta)\lambda}{p}}.
\end{equation}

\ 

{\bf (ii)}
In this step, we will estimate $\mathbb{E}^{W}\left[\exp\left\{4v^{n}(t,x)\right\}\right]$ and $\mathbb{E}^{W}\left[\exp\left\{2v^{n}(t,x)\right\}\right]$ on  the right-hand side of \eqref{e:Hol to u(t,x)} and \eqref{e:to the Hol con of u^n}.
For $(r,x)\in[t,T]\times\bar{D}$, denote
\begin{equation*}
\hat{Y}^{t,x;n}_{r}:=\E^{W}\left[\exp\left\{4\int_{r}^{T}B^{n}(d\nu,X^{t,x}_{\nu})\right\}\right].
\end{equation*}
By Proposition~\ref{prop:exp} and Corollary~\ref{cor:exp'}, there exists a progressively measurable process $\hat{Z}^{t,x;n}$ such that $(\hat{Y}^{t,x;n},\hat{Z}^{t,x;n})\in \mathfrak{B}_{p,2}(t,T)$, and $(\hat{Y}^{t,x;n},\hat{Z}^{t,x;n})$ uniquely solves the following linear BSDE,
\begin{equation*}
\hat{Y}^{t,x;n}_{r} = 1 + \int_{r}^{T}4 \hat{Y}^{t,x;n}_{\nu} B^n(d\nu,X^{t,x}_{\nu}) - \int_{r}^{T}\hat{Z}^{t,x;n}_{\nu} dW_{\nu},\ r\in[t,T].
\end{equation*}
In addition, letting $F:\bar{D}\rightarrow \R$ be the same smooth cutoff function as in the proof of Proposition~\ref{prop:app 3}, we have 
$\sup_{n\ge 1}\|F\cdot B^{n}\|_{\tau',\lambda}<\infty$, $\mathbb{P}'$-a.s., and hence Assumption \ref{(H0)} holds with $(\tau, \lambda)$ 
replaced by $(\tau',\lambda)$. 
Moreover, by \eqref{e:X<=exp{W}} and \eqref{e:BMO of exp W} we have $
\sup_{(t,x)\in[0,T]\times\bar{D}}m_{p,k}(X^{t,x};[t,T])<\infty.$
Then, we can apply Proposition~\ref{prop:exp} to get that there exists a finite positive random variable $K(\cdot)$ independent of $(n,t,x)$ such that  $\hat{Y}^{t,x;n}_{t}(\omega') \le  K(\omega')$, $\mathbb{P}'$-a.s. Hence, we have

\begin{equation}\label{e:4v^n <= K}
\mathbb{E}^{W}\left[\exp\left\{4v^{n}(t,x)\right\}\right](\omega') = \hat{Y}^{t,x;n}_{t}(\omega') \le  K(\omega')<\infty. 
\end{equation}

\ 

{\bf (iii)} In this final step, we will  obtain a uniform upper bound for $\|u^{n}(\cdot,x)\|_{\tau^*\text{-H\"ol};[0,T]}$ using the estimations obtained in the previous two steps.  By \eqref{e:Hol to u(t,x)}, \eqref{e:|tilde v^n|^2}, and \eqref{e:4v^n <= K}, we have for $\mathbb{P}'$-a.s. $\omega'\in\Omega'$, 
\begin{equation}\label{e:Hol to u(t,x)'}
\left| \E^{W}\left[ h(X^{s,x}_{T})\left(\exp\left\{v^{n}(t,x)\right\} - \exp\left\{v^{n}(s,x)\right\}\right) \right] \right|(\omega') \lesssim |K(\omega')|^{\frac{1}{2}} \|B^{n}\|_{\tau',\lambda;\beta}(\omega') |t - s|^{\frac{(1-\theta)\lambda}{p}},
\end{equation}
and by \eqref{e:to the Hol con of u^n}, \eqref{e:4 moments}, and \eqref{e:4v^n <= K}, we have for $\mathbb{P}'$-a.s. $\omega'\in\Omega'$, 
\begin{equation}\label{e:to the Hol con of u^n'}
\left|\E^{W}\left[\left(h(X^{t,x}_{T}) - h(X^{s,x}_{T})\right)\exp\left\{v^{n}(t,x)\right\}\right]\right|(\omega') \lesssim |K(\omega')|^{\frac{1}{4}} |t - s|^{\frac{1}{p}}.
\end{equation}
Recalling \eqref{e:expression of u^n - u^n}, the estimations \eqref{e:Hol to u(t,x)'} and \eqref{e:to the Hol con of u^n'} together with triangular inequality yield that,  for all $(n,t,s,x)$,
\begin{equation*}
\left|u^{n}(t,x) - u^{n}(s,x)\right|(\omega')\lesssim \left(\left|K(\omega')\right|^{\frac{1}{2}}\sup_{n\ge 1}\|B^n\|_{\tau',\lambda;\beta}(\omega') + \left|K(\omega')\right|^{\frac{1}{4}} \right)|t-s|^{\frac{(1-\theta)\lambda}{p}},
\end{equation*}
which further implies that for $\mathbb{P}'$-a.s. $\omega'\in\Omega'$,
\begin{equation}\label{e:uniform Hol of u^n}
\sup_{n\ge 1}\|u^{n}(\cdot,x)\|_{\frac{(1-\theta)\lambda}{p}\text{-H\"ol};[0,T]}(\omega') \lesssim \left|K(\omega')\right|^{\frac{1}{2}}\sup_{n\ge 1}\|B^n\|_{\tau',\lambda;\beta}(\omega') + \left|K(\omega')\right|^{\frac{1}{4}}.
\end{equation}

\ 

Now, we are ready  to prove \eqref{e:weak sol condi 1} and \eqref{e:weak sol condi 2}. In view of the  uniform convergence  \eqref{e:u^n -> u} and the uniform bound  \eqref{e:uniform Hol of u^n}, we can apply  \cite[Corollary 8.18 (ii)]{FrizVictoir} and get that for each $\tau''\in (0,\frac{(1-\theta)\lambda}{p})$,	 
\begin{equation}\label{e:|u^n|_Hol}
\|u^n(\cdot,x) - u(\cdot,x)\|_{\tau''\text{-H\"ol};[0,T]}\to 0 \text{ uniformly in } x\in \bar{D}\text{ as } n\to \infty, \quad  \mathbb{P}'\text{-a.s.} 
\end{equation}
Note that the condition  $\tau' + \frac{\theta\lambda}{p} > 1$ is equivalent to $\frac{(1-\theta)\lambda}{p}< \tau' + \frac{\lambda}{p}  - 1$, and  recall that $\tau^*\in (0,H_0+\frac{H}{2} - 1)$. Then, by choosing $(\lambda, \tau, p)$ sufficiently close to $(H, H_0, 2)$ and  $\tau'$ sufficiently close to   $\tau$, we can find  $\theta\in (\frac{p(1-\tau')}{\lambda},1)$ such that:
\begin{equation*}
\frac{(1-\theta)\lambda}{p}> \tau^*.
\end{equation*}
Similarly, since $\lambda + 4\tau'> 4,$ by choosing $p$ close enough to $2$ and $\theta$ close enough to $\frac{p(1-\tau')}{\lambda}$, we have 
\[\frac{(1-\theta)\lambda}{p} + \tau' > 1.\]
Thus, we can choose $\tau''\in (0,\frac{(1-\theta)\lambda}{p})$ such that
\begin{equation}\label{e:tau'' + tau > 1}
\tau'' > \tau^*\text{ and }\tau'' + \tau' > 1.
\end{equation}
Noting that 
$\tau^{*}\in (0,H_0 + \frac{H}{2} - 1)$ is arbitrary, and that $1 - H_0 < H_0 + \frac{H}{2} - 1$ by the condition $H_{0} + \frac{H}{4} > 1$, we assume without loss of generality that $\tau^* \in (1 - H_{0}, H_0 + \frac{H}{2} - 1)$. Then, \eqref{e:tau'' + tau > 1} together with \eqref{e:|u^n|_Hol} yields \eqref{e:weak sol condi 1} with $\tilde{\tau} = \tau''\in(1-H_0,1]$.

To prove  \eqref{e:weak sol condi 2}, first note that \eqref{e:p-var and Holder} and  \eqref{e:|u^n|_Hol} imply 
\begin{equation}\label{e:|u^n|_Hol'}
\|u^n(\cdot,x) - u(\cdot,x)\|_{\frac{1}{\tau''}\text{-var};[0,T]}\to 0 \text{ uniformly in } x\in \bar{D}\text{ as } n\to \infty,\quad  \mathbb{P}'\text{-a.s.}
\end{equation}
Then, combining the convergence \eqref{e:Bn -> B} of $\{B^n\}_{n\ge 1}$,  the inequality $\tau'' + \tau ' >1$ in \eqref{e:tau'' + tau > 1}, and the convergence \eqref{e:|u^n|_Hol'} of $\{u^n(\cdot,x)\}_{n\ge 1}$, we can apply the estimate on 
nonlinear Young integrals given by Proposition~\ref{prop:bounds} and get that
\begin{equation*}
\lim_{n\rightarrow \infty} \left|\int_{t}^{T}u^{n}(r,x)B^n(dr,x) - \int_{t}^{T}u(r,x)B(dr,x)\right| = 0\text{ uniformly in } x\in \bar{D} , \quad  \mathbb{P}'\text{-a.s.},
\end{equation*}
which implies 
\begin{equation*}
\lim_{n\rightarrow \infty} \left|\int_{D}\int_{t}^{T}u^{n}(r,x)\varphi(x)B^n(dr,x)dx - \int_{D}\int_{t}^{T}u(r,x)\varphi(x)B(dr,x)dx\right| = 0, \quad  \mathbb{P}'\text{-a.s.}
\end{equation*}
This yields \eqref{e:weak sol condi 2} and hence completes the proof.
\end{proof}

\appendix 	

\section{Some facts on processes with space-time regularity}\label{miscellaneous results}

The Burkholder-Davis-Gundy inequality (BDG inequality) for $p$-variation can be found, for example, in Friz-Victoir \cite[Theorem~14.12]{FrizVictoir}. Let  $M$ be a continuous local martingale defined on some filtered probability space $(\Omega,\mathcal{F},\{\mathcal{F}_{t}\}_{t\in[0,T]},\mathbb{P})$ that satisfies the usual conditions. For $p>2$ and $k\ge 1$, there exists a positive number $C$ depending only on $p$ and $k$ such that
\begin{equation}\label{e:BDG}
C^{-1}\mathbb{E}[\langle M \rangle ^{\frac{k}{2}}_{T}]\le \mathbb{E}[\|M\|^{k}_{p\text{-var};[0,T]}] \le C \mathbb{E}[\langle M \rangle ^{\frac{k}{2}}_{T}].
\end{equation}
Denote $\langle M\rangle_{t,s}:= \langle M\rangle_{s} - \langle M\rangle_{t}$. The inequality \eqref{e:BDG} also holds for  conditional expectations (see Diehl-Zhang \cite{DiehlZhang}). 

\begin{lemma}\label{lem:BDG}
Assume $p>2$, $k\ge 1$, and $t\in[0,T]$. Let $M$ be a continuous local martingale such that $\E[\langle M\rangle_{t,T}^{k/2}]<\infty$.  Then, there exists a positive number $C$ depending only on $p$ and $k$ such that
\[C^{-1}\E_{t}[\langle M\rangle_{t,T}^{\frac k2}]\le \mathbb{E}_{t}[\|M\|^{k}_{p\text{-}\mathrm{var};[t,T]}] \le C\E_{t}[\langle M\rangle_{t,T}^{\frac k2}].\]
If we further assume that $\mathbb{E}[\|M\|^{k}_{\infty;[t,T]}]<\infty$ (hence $M$ is a martingale), the above result, combined with the lower bound of the classical BDG inequality, yields 
\[\mathbb{E}_{t}[\|M\|^{k}_{p\text{-}\mathrm{var};[t,T]}] \lesssim_{p,k} \mathbb{E}_{t}[\|M\|^{k}_{\infty;[t,T]}].\]
In addition, Doob's inequality yields, for $k>1$,
\[\mathbb{E}_{t}[\|M\|^{k}_{p\text{-}\mathrm{var};[t,T]}] \lesssim_{p,k} \mathbb{E}_{t}[|M_T|^{k}].\]
\end{lemma}

\begin{corollary}\label{cor:BDG'}
Assume $p>2$ and $k\ge 1,$ and let $S$ be a stopping time such that $S\le T.$ Suppose that $M$ is a continuous local martingale with $\E[\langle M\rangle_{S,T}^{k/2}]<\infty$. Then, there exists a positive constant $C$ depending only on $p$ and $k$ such that
\begin{equation}\label{e:BDG'}
C^{-1}\mathbb{E}_{S}[\langle M\rangle_{S,T}^{\frac{k}{2}}]\le \mathbb{E}_{S}[\|M\|^{k}_{p\text{-}\mathrm{var};[S,T]}] \le C\mathbb{E}_{S}[\langle M\rangle_{S,T}^{\frac{k}{2}}].
\end{equation}
\end{corollary}
\begin{proof}
Denote $\mathcal{G}_{t}:=\mathcal{F}_{t\vee S}.$ Since $(\Omega,\mathcal{F},\{\mathcal{F}_{t}\}_{t\in[0,T]},\mathbb{P})$ satisfies the usual conditions, and referring to the characterization $\mathcal{G}_{t+} = \mathcal{F}_{t\vee S+}$ shown in \cite[IV-56 (d)]{dellacherie2011probabilities}, clearly $(\Omega,\mathcal{F},\{\mathcal{G}_{t}\}_{t\in[0,T]},\mathbb{P})$ also satisfies the usual conditions. Note that $\bar{M}_{t}:=M_{t\vee S}$ is a martingale under the filtration $\{\mathcal{G}_{t}\}_{t\in[0,T]}$, and that $\langle \bar{M} \rangle_{t} := \langle M\rangle_{t\vee S} - \langle M\rangle_{S}$ is the quadratic variation process of $\bar{M}$. Thus, by Lemma~\ref{lem:BDG}, we have (denoting by $\mathbb{E}^{\mathcal{G}}_{t}$ the conditional expectation with respect to $\mathcal{G}_{t}$)
\begin{equation*}
C^{-1}\mathbb{E}^{\mathcal{G}}_{0}[\langle \bar{M}\rangle_{T}^{\frac{k}{2}}]\le \mathbb{E}^{\mathcal{G}}_{0}[\|\bar{M}\|^{k}_{p\text{-var};[0,T]}] \le C\mathbb{E}^{\mathcal{G}}_{0}[\langle \bar{M}\rangle_{T}^{\frac{k}{2}}],
\end{equation*}
which gives \eqref{e:BDG'}.
\end{proof}

The following approximation result is used in the proof of  Proposition~\ref{prop:exp'}. 		

\begin{lemma}\label{lem:smooth approximation of eta}
Suppose $\eta \in C^{\tau,\lambda}([0,T]\times\mathbb{R}^{d})$ for some $\tau,\lambda\in(0,1]$, where we recall that the space $C^{\tau,\lambda}([0,T]\times\mathbb{R}^d)$ is equipped with the seminorm $\|\cdot\|_{\tau, \lambda}$ defined in \eqref{e:norm2}. Then we have
\begin{equation*}
\sup_{t\in[0,T],x\in\mathbb{R}^{d}}|\eta(t,x) - \eta(0,x)|<\infty.
\end{equation*}
Moreover, there exists a sequence  $\{\eta^{n}(t,x)\}_{n\ge 1}$ such that for each $n\ge 1$,   $\|\eta^{n}\|_{\tau,\lambda}<\infty$, $\eta^{n}$ is smooth in time, $\partial_{t}\eta^{n}(t,x)$ is continuous in $(t,x)$, and  $\sup_{t\in[0,T],x\in\mathbb{R}^{d}}|\partial_{t}\eta^{n}(t,x)|<\infty$. Furthermore, for every $0<\tau'<\tau$, 
\begin{equation*}
\lim_{n\to\infty}\|\eta^{n} - \eta\|_{\tau',\lambda}= 0, 
\end{equation*}	
\end{lemma}

\begin{proof}
Clearly, 
\begin{equation*}
\sup_{t\in[0,T],x\in\mathbb{R}^{d}}|\eta(t,x) - \eta(0,x)|\le \|\eta\|_{\tau,\lambda}|T|^{\tau}<\infty.
\end{equation*} 
We extend the time domain of $\eta$ to $\R$, i.e, we define
\begin{equation*}
\eta(t,x) = \eta(0,x) \mathbf 1_{[t<0]} + \eta(t,x) \mathbf 1_{[0\le t\le T]} + \eta(T,x)\mathbf 1_{[t>T]},\  (t,x)\in\R\times \R^d.
\end{equation*}
Consider $\rho\in C^{\infty}_{c}(\mathbb{R})$ (i.e. $\rho$ is a smooth function with  compact support) such that $\rho \ge 0$ and $\int_{\mathbb{R}}\rho(s)ds = 1$. For $n\ge 1$, let $\rho^{n}(t) := n\rho(nt)$ and denote
\begin{equation*}
\eta^{n}(t,x) := \int_{\mathbb{R}}\rho^{n}(t-s)\eta(s,x)ds,\ t\in[0,T],x\in\mathbb{R}^{d}.
\end{equation*}
Clearly, for each 
$n\ge 1,$ $\eta^{n}$ is smooth in time, $\partial_{t}\eta^{n}(t,x)$ is continuous in $(t,x),$ and $\sup_{(t,x)}|\partial_{t}\eta^{n}(t,x)|<\infty$ by the fact that
\begin{equation*}
\sup_{t\in[0,T],x\in\R^{d}}\left|\int_{\R}\partial_{t}\rho^{n}(t-s)\Big(\eta(s,x) - \eta(0,x)\Big) ds \right|\le \|\eta\|_{\tau,\lambda}|T|^{\tau}\int_{\R}|\partial_{t}\rho^{n}(t)|dt < \infty.
\end{equation*}
In addition, by \eqref{e:norm2} and the fact that $\int_\R\rho^n(s) ds=1$, we have $\|\eta^{n}\|_{\tau,\lambda}\le \|\eta \|_{\tau,\lambda}<\infty. $ 
Furthermore, denoting $\delta \eta^{n}:=\eta^{n} - \eta$ and $C := 2\int_{\R}\rho(r)\left|r\right|^{\tau}dr$, we have 
\begin{equation}\label{e:smooth appro}
\begin{aligned}
&\sup _{\substack{0 \leq s<t \leq T \\ x, y \in \mathbb{R}^{d} ; x \neq y}} \frac{|\delta \eta^{n}(s, x)-\delta \eta^{n}(t, x)-\delta\eta^{n}(s, y)+\delta \eta^{n}(t, y)|}{|x - y|^{\lambda}}\\
& \le \sup _{\substack{0 \leq t \leq T \\ x, y \in \mathbb{R}^{d} ; x \neq y}} \frac{\int_{\R}\rho^{n}(r)\left|\eta(t-r,x) - \eta(t,x) - \eta(t-r,y) + \eta(t,y)\right|dr}{|x - y|^{\lambda}}\\
&\quad + \sup _{\substack{0 \leq s \leq T \\ x, y \in \mathbb{R}^{d} ; x \neq y}} \frac{\int_{\R}\rho^{n}(r)\left|\eta(s-r,x) - \eta(s,x) - \eta(s-r,y) + \eta(s,y)\right|dr}{|x - y|^{\lambda}}\\
&\le 2\|\eta\|_{\tau,\lambda}\int_{\R}\rho^{n}(r)|r|^{\tau} dr = C n^{-\tau} \|\eta\|_{\tau,\lambda} \rightarrow
 0,\text{ as }n\rightarrow 0.
\end{aligned}
\end{equation}

For every $0<\tau'<\tau$, we claim that $\|\delta\eta^{n}\|_{\tau',\lambda}\rightarrow 0$ as $n\rightarrow \infty$. More precisely, by \eqref{e:smooth appro}, 
\begin{equation*}
\begin{aligned}	
&\sup _{\substack{0 \leq s<t \leq T \\ x, y \in \mathbb{R}^{d} ; x \neq y}} \frac{|\delta \eta^{n}(s, x)-\delta \eta^{n}(t, x)-\delta\eta^{n}(s, y)+\delta \eta^{n}(t, y)|}{|t-s|^{\tau'}|x-y|^{\lambda}} \notag \\
&\le  \Big(\sup _{\substack{0 \leq s<t \leq T \\ x, y \in \mathbb{R}^{d} ; x \neq y}} \frac{|\delta \eta^{n}(s, x)-\delta \eta^{n}(t, x)-\delta\eta^{n}(s, y) + \delta \eta^{n}(t, y)|}{|t-s|^{\tau}|x-y|^{\lambda}}\Big)^{ \frac{\tau'}{\tau}} \Big(Cn^{-\tau}\|\eta\|_{\tau,\lambda}\Big)^{1 - \frac{\tau'}{\tau}}\\		
&\lesssim  n^{\tau' - \tau} \|\eta\|_{\tau,\lambda} \to 0 \text{ as } n\to \infty.
\end{aligned}
\end{equation*}
Similarly, we have
\begin{align*}
&\sup _{\substack{0 \leq t \leq T \\ x ,y\in \mathbb{R}^{d}}} \frac{|\delta \eta^{n}(t, x)-\delta \eta^{n}(t, y)|}{|x-y|^{\lambda}}\lesssim n^{\tau' - \tau}\|\eta\|_{\tau,\lambda} \to 0 \text{ as } n\to \infty.	
\end{align*}
Moreover, 
\begin{equation}\label{e:supreme norm approxiamtion of eta}
\begin{aligned}
\sup_{t\in[0,T],x\in\mathbb{R}^{d}}|\eta^{n}(t,x) - \eta(t,x)|
&\le \sup_{t\in[0,T],x\in\mathbb{R}^{d}}\int_{\mathbb{R}}\rho^{n}(s)|\eta(t,x) - \eta(t-s,x)|ds\\
&\le \sup_{\substack{s,t\in[0,T],|s-t|<\frac{1}{n}\\ x\in\mathbb{R}^{d}}}|\eta(t,x) - \eta(s,x)| \le \frac{\|\eta\|_{\tau,\lambda}}{n^{\tau}}\to 0, \text{ as } n \to \infty,
\end{aligned}
\end{equation}
where we assume that $\rho^{n}$ is supported in $[-\frac{1}{2}, \frac{1}{2}]$ without loss of generality. 
Hence by \eqref{e:supreme norm approxiamtion of eta}, 
\begin{equation*}
\sup _{\substack{0 \leq s<t \leq T \\ x \in \mathbb{R}^{d}}} \frac{|\delta \eta^{n}(s, x)-\delta \eta^{n}(t, x)|}{|t-s|^{\tau'}} \le 2 \|\eta\|^{\frac{\tau'}{\tau}}_{\tau,\lambda}\cdot\left(\sup_{t\in[0,T],x\in\mathbb{R}^{d}}|\delta\eta^{n}(t,x)|\right)^{1 - \frac{\tau'}{\tau}}\rightarrow0 \text{ as }n\rightarrow\infty.
\end{equation*}
This completes the proof. 
\end{proof}

For $\eta\in C^{\tau,\lambda;\beta}([0,T]\times\mathbb{R}^{d})$, we also have the following result similar to Lemma~\ref{lem:smooth approximation of eta}.

\begin{lemma}\label{lem:smooth approximation of eta'}
Suppose $\eta \in C^{\tau,\lambda;\beta}([0,T]\times\mathbb{R}^{d})$ for some $\tau,\lambda\in(0,1]$ and $\beta\ge 0$, where we recall that the space $C^{\tau,\lambda;\beta}([0,T]\times\mathbb{R}^d)$ is equipped with the seminorm $\|\cdot\|_{\tau, \lambda;\beta}$ defined in \eqref{e:norm1}. Then, there exists a sequence  $\{\eta^{n}(t,x)\}_{n\ge 1}\subset C^{\tau,\lambda;\beta}([0,T]\times\mathbb{R}^{d})$ such that for each $n\ge 1$, $\eta^{n}$ is smooth in time and $\partial_{t}\eta^{n}(t,x)$ is continuous in $(t,x)$,  and that 
\begin{equation*}
\lim_{n\to\infty}\|\eta^{n} - \eta\|_{\tau',\lambda;\beta}= 0, 
\end{equation*} 
for every $0<\tau'<\tau$.
\end{lemma}

The following lemma shows that the fractional Brownian sheet almost surely belongs to $C^{\tau,\lambda;\beta}$.

\begin{lemma}\label{lem:Hurst parameter}
Suppose $B$ is a fractional Brownian sheet on $[0,T]\times\mathbb{R}^{d}$ with Hurst parameter $H :=(H_{0},H_{1},\ldots,H_{d})$, $H_{j}\in(0,1)$ for every $j=0,1,...,d$. For every  $\theta\in(0,\min_{0\le j\le d}H_j)$, set $\tau := H_{0}-\theta$, $\lambda := \min_{1\le j\le d}H_{j} - \theta$ and $\beta := 2\theta + \sum_{j=1}^{d}H_{j} - \min_{1\le j\le d}H_{j}$. Then we have
$
\|B\|_{\tau,\lambda;\beta}<\infty,\  a.s.
$
\end{lemma}

\begin{proof}
First, Hu-Lê \cite[Example~5.9]{HuLe} implies that
\begin{equation}\label{e:growth of continuity modulus}
\sup_{\substack{R\ge 1,0<\delta_{i}\le 2^{H_{i}}R\\ i=0,...,d}}\sup_{\substack{|u_{i}|^{H_{i}},|v_{i}|^{H_{i}}\le R\\|u_{i}-v_{i}|^{H_{i}}\le \delta_{i},i=0,...,d}}\frac{|B(\square_{d+1}[u,v])|}{\delta_{0}\cdots\delta_{d}\sqrt{\log((2R)^{d+1}\delta^{-1}_{0}\cdots\delta^{-1}_{d})}}<\infty\qquad \text{a.s.}
\end{equation}
Here, $u:=(t,x),v:=(s,y)$ and the notation $B(\square_{d+1}[(t,x_1,...,x_d),(s,y_1,...,y_d)])$ represents the $(d+1)$-increment of $B(\cdot)$ over the rectangle $[(t,x),(s,y)]$, as used in \cite[Section~5]{HuLe}. In \eqref{e:growth of continuity modulus}, for every $(t,x),(s,y)\in[0,T]\times \mathbb{R}^{d}$, let $\delta_{0} = |t-s|^{H_{0}}$, $\delta_{j} = |x_{j}-y_{j}|^{H_{j}}$ for $j=1,...,d$, and $R = 1\vee T^{H_{0}}\vee\sup_{j=1,...,d}\left\{|x_{j}|^{H_{j}}\vee|y_{j}|^{H_{j}}\right\}$. Then, by \eqref{e:growth of continuity modulus}, the following fact holds,
\begin{equation*}
\begin{aligned}
M:= \sup_{\substack{(t,x),(s,y)\\ \in[0,T]\times \mathbb{R}^{d}}}\frac{|B(\square_{d+1}[(t,x),(s,y)])|}{|t-s|^{H_{0}}\prod\limits_{j=1}^{d}|x_{j}-y_{j}|^{H_{j}}\Big|\log\Big((2R)^{d+1}|t-s|^{-H_{0}}\prod\limits_{j=1}^{d}|x_{j}-y_{j}|^{-H_{j}}\Big)\Big|^{\frac{1}{2}}} <\infty\qquad \text{a.s.}
\end{aligned}
\end{equation*}
Therefore,
\begin{equation}\label{eq:mult-Holder continuity}
\begin{aligned}
&\sup_{\substack{(t,x),(s,y)\\ \in[0,T]\times \mathbb{R}^{d}}}\frac{|B(\square_{d+1}[(t,x),(s,y)])|}{|t-s|^{H_{0} - \theta}\prod\limits_{j=1}^{d}|x_{j}-y_{j}|^{H_{j} - \theta}\left(1 + |x|^{(d+1)\theta} + |y|^{(d+1)\theta}\right)}\\
&\lesssim\sup_{|t-s|^{H_{0}} \prod\limits_{j=1}^{d}|x_{j}-y_{j}|^{H_{j}}\le 1}\frac{|B(\square_{d+1}[(t,x),(s,y)])|\cdot|t-s|^{-H_{0}}\cdot\prod\limits_{j=1}^{d}|x_{j}-y_{j}|^{-H_{j}}}{\left|\log\left(1 + (2R)^{d+1}\right)\right|^{\frac{1}{2}} + \Big|\log\Big(|t-s|^{-H_{0}}\prod\limits_{j=1}^{d}|x_{j}-y_{j}|^{-H_{j}}\Big)\Big|^{\frac{1}{2}}}\\
&\quad + \sup_{|t-s|^{H_{0}}\prod\limits_{j=1}^{d}|x_{j}-y_{j}|^{H_{j}} > 1}\frac{|B(\square_{d+1}[(t,x),(s,y)])|}{|t-s|^{H_{0}}\prod\limits_{j=1}^{d}|x_{j}-y_{j}|^{H_{j}}\left|\log\left(1 + (2R)^{d+1}\right)\right|^{\frac{1}{2}}}\le  2M ,\quad \text{a.s.}
\end{aligned}
\end{equation}

Let $ w^{k}:= (x_{1},...,x_{k},y_{k+1},...,y_{d})\in\mathbb{R}^{d}$; $z^{k}:= (0,...,0,y_{k},0,...,0)\in\mathbb{R}^{d}$, where $y_{k}$ is at the $k$-th component of $z^{k}$. 
Notice that $B(t,x) = 0$ if any component of $(t,x)$ is zero. We have
\begin{equation*}
\begin{aligned}
&B(\square_{d+1}[(t,\omega^{k});(s,z^{k})])\\
&= B(\square_{d+1}[(t,x_1,...,x_k,y_{k+1},...,y_d);(s,0,...,0,y_k,0,...,0)])\\
&=(-1)^{d+1} B(t,x_1,...,x_{k-1},x_k,y_{k+1},...,y_{d}) +  (-1)^{d} B(t,x_1,...,x_{k-1},y_{k},y_{k+1},...,y_{d}) \\
&\quad + (-1)^{d-1} B(s,x_1,...,x_{k-1},y_{k},y_{k+1},...,y_{d}) + (-1)^{d} B(s,x_1,...,x_{k-1},x_k,y_{k+1},...,y_{d}).
\end{aligned}
\end{equation*}
Therefore, we have
\begin{equation}\label{eq:two increment and d+1 increment}
B(t,x) - B(t,y) - B(s,x) + B(s,y) = (-1)^{d+1} \sum_{k=1}^{d} B(\square_{d+1}[(t,w^{k});(s,z^{k})]).
\end{equation}

Combining \eqref{eq:mult-Holder continuity} with \eqref{eq:two increment and d+1 increment}, and setting $\bar{H}:=\min_{j=1,...,d}H_{j}$, we have a.s.,
\begin{equation*}
\begin{aligned}
&\sup_{\substack{(t,x),(s,y)\\ \in[0,T]\times \mathbb{R}^{d}}}\frac{|B(t,x) - B(t,y) - B(s,x) + B(s,y)|}{|t-s|^{H_{0} - \theta}|x - y|^{\bar{H} - \theta}\left(1 + |x|^{2\theta + \sum_{j=1}^{d}H_{j} - \bar{H}} + |y|^{2\theta + \sum_{j=1}^{d}H_{j} - \bar{H}}\right)}\\
&\lesssim \sup_{\substack{(t,x),(s,y)\\ \in[0,T]\times \mathbb{R}^{d}}}\frac{|B(\square_{d+1}[(t,x),(s,y)])|}{|t-s|^{H_{0} - \theta}\prod\limits_{j=1}^{d}|x_{j}-y_{j}|^{H_{j} - \theta}\left(1 + |x|^{(d+1)\theta} + |y|^{(d+1)\theta}\right)} \lesssim M,
\end{aligned}
\end{equation*}
which completes  the proof.
\end{proof}

\section{Some preliminary proofs on nonlinear Young integrals}\label{append:B}

The following so-called  \emph{sewing lemma}, which plays a key role in defining the nonlinear Young integral, extends Lemma~2.1 and Corollary~2.3 in \cite{feyel}, and can be proved in the same way as in the proof of Theorem~2.5 in \cite{FrizZhang-2018}. 
\begin{lemma}\label{lem:sewing} 
Let $A: \{(s, t); 0 \leq s \le t \leq T\} \rightarrow \mathbb{R}$ be a continuous function satisfying
\begin{equation*}
|\delta A_{s,u,t}| \le \sum_{i = 1}^{l} w_{i}(s,t)^{1+\varepsilon_{i}},\quad 0\le s\le u\le t \le T,
\end{equation*}
for some $\varepsilon_{i}>0,\ i=1,2,...,l$, where $w_{1}$, $w_{2}$,..., $w_{l}$ are controls and $$\delta A_{s,u,t} := A(s,t) - A(s,u) - A(u,t).$$ 

Then, there exists a unique function $\mathcal{A}: [0,T]\rightarrow \mathbb{R}$ satisfying that
\begin{equation*}
|\mathcal{A}_{s,t} - A(s,t)|\le \frac{l^{\varepsilon_{0}}}{1-2^{-\varepsilon_{0}}}\sum_{i = 1}^{l} w_{i}(s,t)^{1+\varepsilon_{i}},\quad  0 \le s \le t \le T,
\end{equation*}
where $\mathcal{A}_{s,t} := \mathcal{A}(t) - \mathcal{A}(s)$, $\varepsilon_{0} := \min\{\varepsilon_{i}, i=1, \dots, l\}$,
and \[\mathcal{A}_{s,t}=\lim_{|\pi|\downarrow 0} \sum\limits_{[t_{i},t_{i+1}]\in\pi} A(t_{i}, t_{i+1}).\] Here, $|\pi|:=\max_{i}|t_i-t_{i-1}|$  is the maximum length of subintervals of the partition $\pi:=\{[t_i,t_{i+1}] ;\  s=t_0<t_1<\cdots<t_n=t\}$, which is also called the mesh of $\pi$. 
\end{lemma}

The proof of the following lemma is similar to \cite[Proposition~5.2]{FrizVictoir} and hence is omitted. 
\begin{lemma}\label{lem:control's property}
Suppose that for every $0\le s \le t \le T$, $|\Gamma(s,t)|\le w(s,t)^{\theta}$ for some $\theta > 1$, where $w$ is a control function.  Then, we have
\[z_{r} := \lim_{|\pi|\downarrow 0}\sum_{[t_{i},t_{i+1}]\in\pi} \Gamma(t_{i}\land r,t_{i+1}\land r)=0,\ \ \ \text{for all }r\in[0,T].
\] 
\end{lemma}

Some proofs for Section~\ref{sec:nonlinear-Y-I} are listed below.

\begin{proof}[Proof of Proposition~\ref{prop:bounds}]
\makeatletter\def\@currentlabelname{the proof}\makeatother
\label{proof:bounds}
First, we will prove that the nonlinear integral is well-defined. Since there exists $K>0$ such that $\|x\|_{\infty;[0,T]}\le K$, we can assume $\eta\in C^{\tau,\lambda}([0,T]\times\mathbb{R}^{d})$ without losing generality. Denoting $A(s,t):=y_{s}\eta(t,x_{s}) - y_{s}\eta(s,x_{s})$, we have
\begin{equation}\label{e:delta-A-sut}
\begin{aligned}
|\delta A_{s,u,t}|&\le|y_s \eta(t,x_s) - y_s \eta(u,x_s) - y_s \eta(t,x_u) + y_s \eta(u,x_u)|\\
&\quad  + | y_s \eta(t,x_u) - y_s \eta(u,x_u)- y_u \eta(t,x_u) + y_u \eta(u,x_u) |\\
&\le \|\eta\|_{\tau,\lambda} \|y\|_{\infty;[s,t]}|t-u|^{\tau} |x_{s} -x_{u}|^{\lambda} + \|\eta\|_{\tau,\lambda}|t-u|^{\tau}|y_s - y_u|\\
&\le \|\eta\|_{\tau,\lambda} |t-s|^{\tau} \left(\|y\|_{\infty;[s,t]} \left(\|x\|_{p_1\text{-var};[s,t]}^{p_1}\right)^{\frac{\lambda}{p_1}} + \left(\|y\|_{p_2\text{-var};[s,t]}^{p_2}\right)^{\frac{1}{p_2}}\right).
\end{aligned}
\end{equation}
Recall that $|t-s|,$ $ \|x\|_{p_{1}\text{-var;[s,t]}}^{p_{1}},$ and $\|y\|_{p_{2}\text{-var;[s,t]}}^{p_{2}}$ are all controls. Then, by Exercise~1.9 in \cite{FrizVictoir} and the assumptions $\tau+\frac{\lambda}{p_{1}}>1$ and $\tau + \frac{1}{p_{2}}>1$, we can define the following controls,
\begin{equation}\label{e:w-young}
\begin{aligned}
w_{1}(s,t)&:=\left\{\|\eta\|_{\tau,\lambda}\|y\|_{\infty;[s,t]} |t-s|^{\tau}  \|x\|^\lambda_{p_{1}\text{-var};[s,t]}\right\}^{\frac{1}{1+\delta}},\quad 
w_{2}(s,t):=\left\{\|\eta\|_{\tau,\lambda}|t-s|^{\tau}\|y\|_{p_{2}\text{-var};[s,t]}\right\}^{\frac{1}{1+\delta}}.
\end{aligned}
\end{equation}
By \eqref{e:delta-A-sut} we have 
\begin{equation}\label{e:delta-A}
|\delta A_{s,u,t}|\le w_{1}(s,t)^{1 + \delta} + w_{2}(s,t)^{1 + \delta}\text{ with } \delta := \min\left\{\tau + \frac{1}{p_{2}} - 1,\tau + \frac{\lambda}{p_{1}} - 1\right\}.
\end{equation}
Therefore, we can apply Lemma~\ref{lem:sewing}  to $A(s,t)$, which defines $\int_a^b y_r \eta(dr, x_r)$.

Second, we will prove the estimates~\eqref{e:p-var2} and \eqref{e:p-var1}. Assume $\eta\in C^{\tau,\lambda;\beta}([0,T]\times\mathbb{R}^{d})$. Then, a calculation similar to \eqref{e:delta-A-sut} leads to, for  $0\le s\le u\le t\le T$,
\begin{equation*}
\begin{aligned}
\left|\delta A_{s,u,t}\right| & \le  \|\eta\|_{\tau,\lambda;\beta} |t - s|^{\tau} \Big((1+\|x\|^{\beta}_{\infty;[s,t]})\|x\|^{\lambda}_{p_{1}\text{-var};[s,t]}\|y\|_{\infty;[s,t]} + (1+\|x\|^{\beta+\lambda}_{\infty;[s,t]}) \|y\|_{p_{2}\text{-var};[s,t]} \Big)\\
&\le \hat{w}_1(s,t)^{1+\delta}+\hat{w}_2(s,t)^{1+\delta},	
\end{aligned}
\end{equation*}
where  
\begin{equation*}
\begin{aligned}
&\hat{w}_{1}(s,t):=\left\{\|\eta\|_{\tau,\lambda;\beta}(1+\|x\|^{\beta}_{\infty;[s,t]})\|y\|_{\infty;[s,t]} |t - s|^{\tau}\|x\|^{\lambda}_{p_{1}\text{-var};[s,t]}\right\}^{\frac{1}{1+\delta}},\\
& \hat{w}_{2}(s,t):=\left\{\|\eta\|_{\tau,\lambda;\beta}(1+\|x\|^{\beta+\lambda}_{\infty;[s,t]}) |t - s|^{\tau} \|y\|_{p_{2}\text{-var};[s,t]}\right\}^{\frac{1}{1+\delta}},
\end{aligned}
\end{equation*}
are controls, as shown by Exercise~1.9 in \cite{FrizVictoir}. By Lemma~\ref{lem:sewing}, we have for $0\le s\le t\le T$, 
\begin{equation}\label{e:inequality-estimate-Young}
\begin{aligned}
\left|\int_{s}^{t} y_{r} \eta(dr,x_{r}) - A(s,t)\right| &\le \frac{2^{\delta}}{1-2^{-\delta}}\Big(\hat{w}_1(s,t)^{1+\delta}+\hat{w}_2(s,t)^{1+\delta}\Big).
\end{aligned}
\end{equation}
By the definition of $\|\eta\|_{\tau,\lambda;\beta}$ in \eqref{e:norm1}, we also have the following estimate for  $s\le u\le v\le t$,
\begin{equation}\label{e:inequality-estimate-Young2}
|A(u,v)|=|y_u\eta(u,x_{u}) - y_u\eta(v,x_{u})|\le \|\eta\|_{\tau,\lambda;\beta}|v-u|^{\tau}\left(1+\|x\|^{\beta+\lambda}_{\infty;[u,v]}\right)\|y\|_{\infty;[u,v]}.
\end{equation}
Now we are ready to estimate $\|\int_{\cdot}^{t}y_{r}\eta(dr,x_{r})\|_{\frac{1}{\tau}\text{-var};[s,t]}$. 	Letting $\pi$ be a partition on $[s,t]$, combining \eqref{e:inequality-estimate-Young} with
\eqref{e:inequality-estimate-Young2}, we have 
\begin{equation*}
\begin{aligned}
\sum_{[t_{i},t_{i+1}]\in\pi}\left|\int_{t_{i}}^{t_{i+1}}y_{r}\eta(dr,x_{r})\right|^{\frac{1}{\tau}} &\le \sum_{[t_{i},t_{i+1}]\in\pi}\Bigg( \frac{2^{\delta}}{1-2^{-\delta}}\Big(\hat{w}_1(t_{i}, t_{i+1}))^{1+\delta}+\hat{w}_2(t_i,t_{i+1})^{1+\delta}\Big) \\
&\qquad   + \|\eta\|_{\tau,\lambda;\beta}|t_{i+1}-t_i|^{\tau}\left(1+\|x\|^{\beta+\lambda}_{\infty;[t_i,t_{i+1}]}\right)\|y\|_{\infty;[t_i,t_{i+1}]}\Bigg)^{\frac1\tau} \\
&\le \left(\frac{2^{\delta }}{1-2^{-\delta}}\right)^{\frac{1}{\tau}}\|\eta\|^{\frac{1}{\tau}}_{\tau,\lambda;\beta} |t-s| \Bigg(\Big(1+\|x\|^{\beta}_{\infty;[s,t]}\Big)\|y\|_{\infty;[s,t]}\|x\|^{\lambda}_{p_{1}\text{-var};[s,t]}\\ 						&\qquad + \Big(1+\|x\|^{\beta + \lambda}_{\infty;[s,t]}\Big)\|y\|_{p_{2}\text{-var};[s,t]} + \Big(1+\|x\|^{\beta+\lambda}_{\infty;[s,t]}\Big)\|y\|_{\infty;[s,t]}\Bigg)^{\frac{1}{\tau}},
\end{aligned}
\end{equation*}
which implies that 
\begin{equation*}
\begin{aligned}
\left\|\int_{\cdot}^{t}y_{r}\eta(dr,X_{r})\right\|_{\frac{1}{\tau}\text{-var};[s,t]}&\le \frac{2^{\delta }}{1-2^{-\delta}}\|\eta\|_{\tau,\lambda;\beta} |t-s|^{\tau} \Bigg(\Big(1+\|x\|^{\beta}_{\infty;[s,t]}\Big)\|y\|_{\infty;[s,t]}\|x\|^{\lambda}_{p_{1}\text{-var};[s,t]}\\ 
&\quad + \Big(1+\|x\|^{\beta + \lambda}_{\infty;[s,t]}\Big)\|y\|_{p_{2}\text{-var};[s,t]} + \Big(1+\|x\|^{\beta+\lambda}_{\infty;[s,t]}\Big)\|y\|_{\infty;[s,t]}\Bigg).
\end{aligned}
\end{equation*}
Then we have \eqref{e:p-var2}, and \eqref{e:p-var1} can be proved similarly by verifying the following fact,
\begin{equation}\label{e:inequality-estimate-Young3}
\left|\int_{s}^{t} y_{r} \eta(dr,x_{r}) - A(s,t)\right| \le \frac{2^{\delta}}{1-2^{-\delta}}\Big(w_1(s,t)^{1+\delta}+w_2(s,t)^{1+\delta}\Big).
\end{equation}
\end{proof}

\begin{proof}[Proof of Lemma~\ref{lem:estimate for Young integral-for g is bounded}]
\makeatletter\def\@currentlabelname{the proof}\makeatother
\label{proof:g(Y)}
Let  $A(s,t):= y_{s}\eta(t,x_{s}) - y_{s}\eta(s,x_{s}).$
Similar to \eqref{e:delta-A-sut} in Proposition~\ref{prop:bounds}, we have for  $0\le s\le u\le t\le T$,
\begin{equation*}
\begin{aligned}
\left|\delta A_{s,u,t}\right| &\le \|\eta\|_{\tau,\lambda;\beta} \|y\|_{\infty;[s,t]}|t-u|^{\tau} |x_{s} -x_{u}|^{\lambda}(1+\|x\|^{\beta}_{\infty;[s,t]})\\ 
&\qquad\qquad\qquad\qquad\qquad + \|\eta\|_{\tau,\lambda;\beta}|t-u|^{\tau}|y_s - y_u|(1+\|x\|^{\beta+\lambda}_{\infty;[s,t]})\\
&\le \|\eta\|_{\tau,\lambda;\beta} |t - s|^{\tau} \Big((1+\|x\|^{\beta}_{\infty;[s,t]})\|x\|^{\lambda}_{p\text{-var};[s,t]}\|y\|_{\infty;[s,t]}\\						&\qquad\qquad\qquad\qquad\qquad + 2^{\varepsilon}(1+\|x\|^{\beta+\lambda}_{\infty;[s,t]}) \|y\|^{\varepsilon}_{\infty;[s,t]}\|y\|^{1-\varepsilon}_{p\text{-var};[s,t]} \Big)\\
&\le \tilde{w}_1(s,t)^{1+\delta'} + \tilde{w}_2(s,t)^{1+\delta'},	
\end{aligned}
\end{equation*}
where $\delta' := \min\left\{\tau + \frac{1-\varepsilon}{p} - 1, \tau + \frac{\lambda}{p} - 1\right\}$, and 
\begin{equation*}
\begin{aligned}
&\tilde{w}_{1}(s,t):=\left\{\|\eta\|_{\tau,\lambda;\beta} |t - s|^{\tau}(1+\|x\|^{\beta}_{\infty;[s,t]})\|x\|^{\lambda}_{p\text{-var};[s,t]}\|y\|_{\infty;[s,t]}\right\}^{\frac{1}{1+\delta'}},\\
&\tilde{w}_{2}(s,t):=\left\{2^{\varepsilon}\|\eta\|_{\tau,\lambda;\beta} |t - s|^{\tau}(1+\|x\|^{\beta+\lambda}_{\infty;[s,t]}) \|y\|^{\varepsilon}_{\infty;[s,t]}\|y\|^{1 - \varepsilon}_{p\text{-var};[s,t]}\right\}^{\frac{1}{1+\delta'}},
\end{aligned}
\end{equation*}
are control functions.  Then the desired result follows from the proof of Proposition~\ref{prop:bounds}.
\end{proof}

\begin{proof}[Proof of Lemma~\ref{lem:delta-g}]
\makeatletter\def\@currentlabelname{the proof}\makeatother
\label{proof:epsilon}
Assume $x = (x^1,x^2,...,x^N)^{\top}$ and $y = (y^1,y^2,...,y^N)^{\top}$, where $x^{i},y^{i}\in C^{p\text{-var}}([0,T])$ for $i=1,2,...,N$. Let $s\le t$. By the Newton-Leibniz formula and the triangular inequality, we have
\begin{equation*}
\begin{aligned}
&|g(y_t) - g(y_s) - g(x_t) +g(x_s)| \\
&\le \Big|\sum_{i=1}^{N}\left[\int_{0}^{1}\left(\partial_{y^i}g(y^{1}_t,...,y^{i-1}_t,y^{i}_t + \lambda(x^{i}_t - y^{i}_t),x^{i+1}_t,...,x^{N}_t)\right)d\lambda(x^{i}_{t} - y^{i}_t - x^{i}_{s} + y^{i}_s)\right]\Big|\\
&\quad + \Big|\sum_{i=1}^{N}\Big[\int_{0}^{1}\big(\partial_{y^i}g(y^{1}_t,...,y^{i-1}_t,y^{i}_t + \lambda(x^{i}_t - y^{i}_t),x^{i+1}_t,...,x^{N}_t) \\
&\quad \quad - \partial_{y^i}g(y^{1}_s,...,y^{i-1}_s,y^{i}_s + \lambda(x^{i}_s - y^{i}_s),x^{i+1}_s,...,x^{N}_s)\big)d\lambda(x^{i}_{s} - y^{i}_{s})\Big]\Big|\\
&\lesssim \|\nabla g\|_{\infty;\mathbb{R}^N}|(x_{t} - y_{t}) - (x_{s} - y_{s})| + \|\nabla^2 g\|_{\infty;\mathbb{R}^{N}}  (|x_{t} - x_{s}| + |y_{t} - y_{s}|)|x_{s} - y_{s}|,
\end{aligned}
\end{equation*}
which implies the desired result.
\end{proof}

{\bf Acknowledgment}   We wish to thank Peter Friz for helpful discussions and comments. J. Song is partially supported by National Natural Science Foundation of China (No. 12471142); JS and HZ are supported by the Fundamental Research Funds for the Central Universities. HZ is partially supported by NSF of China and Shandong (no. 12031009, ZR2023MA026); Young Research Project of Tai-Shan (no. tsqn202306054);
DFG CRC/TRR 388 ``Rough Analysis, Stochastic Dynamics and Related Fields'', Projects B04 and B05.

\bibliographystyle{plain}
\bibliography{Reference-BSDE}

\begin{thebibliography}{10}

\bibitem{aida2013wong}
Shigeki Aida and Kosuke Sasaki.
\newblock Wong-{Z}akai approximation of solutions to reflecting stochastic
  differential equations on domains in {E}uclidean spaces.
\newblock {\em Stochastic Process. Appl.}, 123(10):3800--3827, 2013.

\bibitem{AKQ14}
Tom Alberts, Konstantin Khanin, and Jeremy Quastel.
\newblock The continuum directed random polymer.
\newblock {\em J. Stat. Phys.}, 154(1-2):305--326, 2014.

\bibitem{Becherer2025}
Dirk Becherer and Yuchen Sun.
\newblock Rough backward {SDE}s with discontinuous {Y}oung drivers.
\newblock {\em arXiv preprint arXiv:2505.20437}, 2025.

\bibitem{bertini1995stochastic}
Lorenzo Bertini and Nicoletta Cancrini.
\newblock The stochastic heat equation: {F}eynman-{K}ac formula and
  intermittence.
\newblock {\em J. Stat. Phys.}, 78(5-6):1377--1401, 1995.

\bibitem{BuckdahnZhang}
Rainer Buckdahn, Christian Keller, Jin Ma, and Jianfeng Zhang.
\newblock Fully nonlinear stochastic and rough {PDE}s: classical and viscosity
  solutions.
\newblock {\em Probab. Uncertain. Quant. Risk}, 5:Paper No. 7, 59, 2020.

\bibitem{Bugini2025a}
Fabio Bugini, Peter~K. Friz, and Wilhelm Stannat.
\newblock Parameter dependent rough {SDE}s with applications to rough {PDE}s.
\newblock {\em arXiv preprint arXiv:2409.11330}, 2024.

\bibitem{Caruana2009}
Michael Caruana and Peter~K. Friz.
\newblock Partial differential equations driven by rough paths.
\newblock {\em J. Differential Equations}, 247(1):140--173, July 2009.

\bibitem{caruana2011rough}
Michael Caruana, Peter~K. Friz, and Harald Oberhauser.
\newblock A (rough) pathwise approach to a class of non-linear stochastic
  partial differential equations.
\newblock {\em Ann. Inst. H. Poincar\'e{} C Anal. Non Lin\'eaire},
  28(1):27--46, 2011.

\bibitem{Catellier-Gubinelli-2016}
R\'emi Catellier and Massimiliano Gubinelli.
\newblock Averaging along irregular curves and regularisation of {ODE}s.
\newblock {\em Stochastic Process. Appl.}, 126(8):2323--2366, 2016.

\bibitem{clark1970representation}
John M.~C. Clark.
\newblock The representation of functionals of {B}rownian motion by stochastic
  integrals.
\newblock {\em Ann. Math. Statist.}, 41:1282--1295, 1970.

\bibitem{DP97}
Richard W.~R. Darling and \'Etienne Pardoux.
\newblock Backwards {SDE} with random terminal time and applications to
  semilinear elliptic {PDE}.
\newblock {\em Ann. Probab.}, 25(3):1135--1159, 1997.

\bibitem{Delarue-Diel16}
François Delarue and Roland Diel.
\newblock Rough paths and 1d {SDE} with a time dependent distributional drift:
  application to polymers.
\newblock {\em Probab. Theory Related Fields}, 165(1-2):1--63, 2016.

\bibitem{dellacherie2011probabilities}
Claude Dellacherie and Paul-Andr\'e Meyer.
\newblock {\em Probabilities and potential}, volume~29 of {\em North-Holland
  Mathematics Studies}.
\newblock North-Holland Publishing Co., Amsterdam-New York, 1978.

\bibitem{DiehlFriz}
Joscha Diehl and Peter~K. Friz.
\newblock Backward stochastic differential equations with rough drivers.
\newblock {\em Ann. Probab.}, 40(4):1715--1758, 2012.

\bibitem{Diehl2017b}
Joscha Diehl, Peter~K. Friz, and Wilhelm Stannat.
\newblock Stochastic partial differential equations: a rough paths view on weak
  solutions via {Feynman}–{Kac}.
\newblock {\em Annales de la Faculté des sciences de Toulouse :
  Mathématiques}, 26(4):911--947, December 2017.

\bibitem{DiehlZhang}
Joscha Diehl and Jianfeng Zhang.
\newblock Backward stochastic differential equations with {Y}oung drift.
\newblock {\em Probab. Uncertain. Quant. Risk}, 2:Paper No. 5, 17, 2017.

\bibitem{Dunlap-Gu}
Alexander Dunlap and Yu~Gu.
\newblock A forward-backward {SDE} from the 2{D} nonlinear stochastic heat
  equation.
\newblock {\em Ann. Probab.}, 50(3):1204--1253, 2022.

\bibitem{ElKaroui1997}
Nicole El~Karoui, Shige Peng, and Marie-Claire Quenez.
\newblock Backward stochastic differential equations in finance.
\newblock {\em Math. Finance}, 7(1):1--71, 1997.

\bibitem{fernique}
Xavier Fernique.
\newblock Int\'egrabilit\'e{} des vecteurs gaussiens.
\newblock {\em C. R. Acad. Sci. Paris S\'er. A-B}, 270:A1698--A1699, 1970.

\bibitem{feyel}
Denis Feyel and Arnaud de~La~Pradelle.
\newblock Curvilinear integrals along enriched paths.
\newblock {\em Electron. J. Probab.}, 11:no. 34, 860--892, 2006.

\bibitem{freidlin1985functional}
Mark Freidlin.
\newblock {\em Functional integration and partial differential equations},
  volume 109 of {\em Annals of Mathematics Studies}.
\newblock Princeton University Press, Princeton, NJ, 1985.

\bibitem{friz2021rough}
Peter~K. Friz, Antoine Hocquet, and Khoa L{\^e}.
\newblock Rough stochastic differential equations.
\newblock {\em arXiv preprint arXiv:2106.10340}, 2021.

\bibitem{FLZ24}
Peter~K. Friz, Khoa L{\^e}, and Huilin Zhang.
\newblock Controlled rough {SDE}s, pathwise stochastic control and dynamic
  programming principles.
\newblock {\em arXiv preprint arXiv:2412.05698}, 2024.

\bibitem{Friz2014a}
Peter~K. Friz and Harald Oberhauser.
\newblock Rough path stability of (semi-)linear {SPDE}s.
\newblock {\em Probab. Theory Related Fields}, 158(1-2):401--434, 2014.

\bibitem{FrizVictoir}
Peter~K. Friz and Nicolas~B. Victoir.
\newblock {\em Multidimensional stochastic processes as rough paths}, volume
  120 of {\em Cambridge Studies in Advanced Mathematics}.
\newblock Cambridge University Press, Cambridge, 2010.
\newblock Theory and applications.

\bibitem{FrizZhang-2018}
Peter~K. Friz and Huilin Zhang.
\newblock Differential equations driven by rough paths with jumps.
\newblock {\em J. Differential Equations}, 264(10):6226--6301, 2018.

\bibitem{galeati2021nonlinear}
Lucio Galeati.
\newblock Nonlinear {Y}oung differential equations: a review.
\newblock {\em J. Dynam. Differential Equations}, 35(2):985--1046, 2023.

\bibitem{Gubinelli-Imkeller-Perkowski}
Massimiliano Gubinelli, Peter Imkeller, and Nicolas Perkowski.
\newblock Paracontrolled distributions and singular {PDE}s.
\newblock {\em Forum Math. Pi}, 3:e6, 75, 2015.

\bibitem{Gubinelli2010}
Massimiliano Gubinelli and Samy Tindel.
\newblock Rough evolution equations.
\newblock {\em Ann. Probab.}, 38(1), January 2010.

\bibitem{Hairer-KPZ}
Martin Hairer.
\newblock Solving the {KPZ} equation.
\newblock {\em Ann. of Math. (2)}, 178(2):559--664, 2013.

\bibitem{hairer2022generating}
Martin Hairer and Xue-Mei Li.
\newblock Generating diffusions with fractional {B}rownian motion.
\newblock {\em Comm. Math. Phys.}, 396(1):91--141, 2022.

\bibitem{hinz2011burgers}
Michael Hinz.
\newblock Burgers' system with a fractional {B}rownian random force.
\newblock {\em Stochastics}, 83(1):67--106, 2011.

\bibitem{hinz2013elementary}
Michael Hinz, Elena Issoglio, and Martina Z\"ahle.
\newblock Elementary pathwise methods for nonlinear parabolic and transport
  type stochastic partial differential equations with fractal noise.
\newblock In {\em Modern stochastics and applications}, volume~90 of {\em
  Springer Optim. Appl.}, pages 123--141. Springer, Cham, 2014.

\bibitem{hinz2009gradientI}
Michael Hinz and Martina Z\"ahle.
\newblock Gradient-type noises. {I}. {P}artial and hybrid integrals.
\newblock {\em Complex Var. Elliptic Equ.}, 54(6):561--583, 2009.

\bibitem{hinz2009gradientII}
Michael Hinz and Martina Z\"ahle.
\newblock Gradient type noises. {II}. {S}ystems of stochastic partial
  differential equations.
\newblock {\em J. Funct. Anal.}, 256(10):3192--3235, 2009.

\bibitem{hinz2012semigroups}
Michael Hinz and Martina Z\"ahle.
\newblock Semigroups, potential spaces and applications to ({S}){PDE}.
\newblock {\em Potential Anal.}, 36(3):483--515, 2012.

\bibitem{Hocquet2024}
Antoine Hocquet and Alexandra Neamţu.
\newblock Quasilinear rough evolution equations.
\newblock {\em Ann. Appl. Probab.}, 34(5), October 2024.

\bibitem{HZ25}
Ulrich Horst and Huilin Zhang.
\newblock Pontryagin {M}aximum {P}rinciple for rough stochastic systems and
  pathwise stochastic control.
\newblock {\em arXiv preprint arXiv:2503.22959}, 2025.

\bibitem{Hsu}
Elton~P. Hsu.
\newblock {\em Reflecting {B}rownian motion, boundary local time and the
  {N}eumann problem}.
\newblock PhD thesis, Stanford University, 1984.

\bibitem{HuLe}
Yaozhong Hu and Khoa L\^e.
\newblock Nonlinear {Y}oung integrals and differential systems in {H}\"older
  media.
\newblock {\em Trans. Amer. Math. Soc.}, 369(3):1935--2002, 2017.

\bibitem{HuLiMi}
Yaozhong Hu, Juan Li, and Chao Mi.
\newblock B{SDE}s generated by fractional space-time noise and related {SPDE}s.
\newblock {\em Appl. Math. Comput.}, 450:Paper No. 127979, 30, 2023.

\bibitem{hu2012feynman}
Yaozhong Hu, Fei Lu, and David Nualart.
\newblock Feynman-{K}ac formula for the heat equation driven by fractional
  noise with {H}urst parameter {$H<1/2$}.
\newblock {\em Ann. Probab.}, 40(3):1041--1068, 2012.

\bibitem{HuNualartSong-2011}
Yaozhong Hu, David Nualart, and Jian Song.
\newblock Feynman-{K}ac formula for heat equation driven by fractional white
  noise.
\newblock {\em Ann. Probab.}, 39(1):291--326, 2011.

\bibitem{IssoglioRusso}
Elena Issoglio and Francesco Russo.
\newblock A {F}eynman-{K}ac result via {M}arkov {BSDE}s with generalised
  drivers.
\newblock {\em Bernoulli}, 26(1):728--766, 2020.

\bibitem{Jing-2012}
Shuai Jing.
\newblock Nonlinear fractional stochastic {PDE}s and {BDSDE}s with {H}urst
  parameter in {$(1/2,1)$}.
\newblock {\em Systems Control Lett.}, 61(5):655--665, 2012.

\bibitem{kazamaki1994continuous}
Norihiko Kazamaki.
\newblock {\em Continuous exponential martingales and {BMO}}, volume 1579 of
  {\em Lecture Notes in Mathematics}.
\newblock Springer-Verlag, Berlin, 1994.

\bibitem{kunita1997stochastic}
Hiroshi Kunita.
\newblock {\em Stochastic flows and stochastic differential equations},
  volume~24 of {\em Cambridge Studies in Advanced Mathematics}.
\newblock Cambridge University Press, Cambridge, 1997.
\newblock Reprint of the 1990 original.

\bibitem{lejay2010controlled}
Antoine Lejay.
\newblock Controlled differential equations as {Y}oung integrals: a simple
  approach.
\newblock {\em J. Differential Equations}, 249(8):1777--1798, 2010.

\bibitem{Liang2024a}
Jiahao Liang and Shanjian Tang.
\newblock Mild solution of semilinear rough stochastic evolution equations.
\newblock {\em arXiv preprint arXiv:2401.16815}, 2024.

\bibitem{liang2023multidimensional}
Jiahao Liang and Shanjian Tang.
\newblock Multidimensional backward stochastic differential equations with
  rough drifts.
\newblock {\em Trans. Amer. Math. Soc.}, 378(1):201--257, 2025.

\bibitem{lions1984stochastic}
Pierre-Louis Lions and Alain-Sol Sznitman.
\newblock Stochastic differential equations with reflecting boundary
  conditions.
\newblock {\em Comm. Pure Appl. Math.}, 37(4):511--537, 1984.

\bibitem{lyons1994differential}
Terry Lyons.
\newblock Differential equations driven by rough signals. {I}. {A}n extension
  of an inequality of {L}. {C}. {Y}oung.
\newblock {\em Math. Res. Lett.}, 1(4):451--464, 1994.

\bibitem{PardouxPeng1990}
\'Etienne Pardoux and Shige Peng.
\newblock Adapted solution of a backward stochastic differential equation.
\newblock {\em Systems Control Lett.}, 14(1):55--61, 1990.

\bibitem{pp92}
\'Etienne Pardoux and Shige Peng.
\newblock Backward stochastic differential equations and quasilinear parabolic
  partial differential equations.
\newblock In {\em Stochastic partial differential equations and their
  applications ({C}harlotte, {NC}, 1991)}, volume 176 of {\em Lect. Notes
  Control Inf. Sci.}, pages 200--217. Springer, Berlin, 1992.

\bibitem{PardouxPeng1994}
\'Etienne Pardoux and Shige Peng.
\newblock Backward doubly stochastic differential equations and systems of
  quasilinear {SPDE}s.
\newblock {\em Probab. Theory Related Fields}, 98(2):209--227, 1994.

\bibitem{PardouxZhanggeneralized}
\'Etienne Pardoux and Shuguang Zhang.
\newblock Generalized {BSDE}s and nonlinear {N}eumann boundary value problems.
\newblock {\em Probab. Theory Related Fields}, 110(4):535--558, 1998.

\bibitem{pei2021averaging}
Bin Pei, Yuzuru Inahama, and Yong Xu.
\newblock Averaging principle for fast-slow system driven by mixed fractional
  {B}rownian rough path.
\newblock {\em J. Differential Equations}, 301:202--235, 2021.

\bibitem{Peng-1991}
Shige Peng.
\newblock Probabilistic interpretation for systems of quasilinear parabolic
  partial differential equations.
\newblock {\em Stochastics Stochastics Rep.}, 37(1-2):61--74, 1991.

\bibitem{SongSongZhang}
Jian Song, Xiaoming Song, and Qi~Zhang.
\newblock Nonlinear {F}eynman-{K}ac formulas for stochastic partial
  differential equations with space-time noise.
\newblock {\em SIAM J. Math. Anal.}, 51(2):955--990, 2019.

\bibitem{BSDEYoung-II}
Jian Song, Huilin Zhang, and Kuan Zhang.
\newblock Backward stochastic differential equations with nonlinear {Y}oung
  drivers {II}.
\newblock In preparation, 2025.

\bibitem{tanaka1979stochastic}
Hiroshi Tanaka.
\newblock Stochastic differential equations with reflecting boundary condition
  in convex regions.
\newblock {\em Hiroshima Math. J.}, 9(1):163--177, 1979.

\bibitem{taylor1996partial}
Michael~E. Taylor.
\newblock {\em Partial differential equations. {I}}, volume 115 of {\em Applied
  Mathematical Sciences}.
\newblock Springer-Verlag, New York, 1996.
\newblock Basic theory.

\bibitem{Wong-Yang-Zhang-22}
Chi~Hong Wong, Xue Yang, and Jing Zhang.
\newblock {N}eumann boundary problems for parabolic partial differential
  equations with divergence terms.
\newblock {\em Potential Anal.}, 56(4):723--744, 2022.

\bibitem{Yong1999}
Jiongmin Yong and Xun~Yu Zhou.
\newblock {\em Stochastic Controls: Hamiltonian Systems and HJB Equations},
  volume~43.
\newblock Springer Science \& Business Media, 1999.

\bibitem{young1936}
Laurence~Chisholm Young.
\newblock An inequality of the {H}\"older type, connected with {S}tieltjes
  integration.
\newblock {\em Acta Math.}, 67(1):251--282, 1936.

\bibitem{zhang2017backward}
Jianfeng Zhang.
\newblock {\em Backward stochastic differential equations}, volume~86 of {\em
  Probability Theory and Stochastic Modelling}.
\newblock Springer, New York, 2017.
\newblock From linear to fully nonlinear theory.

\end{thebibliography}

\end{document}